\documentclass[reqno]{amsart}

\usepackage[T1]{fontenc}

\usepackage{hyperref}
\usepackage{tabularx} 

\usepackage{wasysym}

\usepackage{url}

\usepackage{varwidth}
\usepackage{enumitem}
\usepackage{subcaption}  

\usepackage{esint}
\usepackage{amssymb,centernot}
\usepackage{mathtools} 
\usepackage{tikz}
\usetikzlibrary{arrows.meta}

\usepackage{pgfplots}
\usepackage{amsmath,amsthm}

\pgfplotsset{compat=1.18}

\usepackage{graphicx}
\usepackage{amsrefs,enumitem,pgfkeys}
\usepackage{soul}
\usepackage{dirtytalk}
\usepackage{mathrsfs}
\usepackage{MnSymbol,wasysym}
\usepackage{fancyhdr}
\usepackage{time}

\usepackage[margin=1in]{geometry}

\newcommand\restr[2]{{
  \left.\kern-\nulldelimiterspace 
  #1 
  \vphantom{\big|} 
  \right|_{#2} 
  }}

\usepackage{comment}
\newtheorem{theorem}{Theorem}
\newtheorem{lemma}[theorem]{Lemma}
\newtheorem{corollary}[theorem]{Corollary}
\newtheorem{proposition}[theorem]{Proposition}
\theoremstyle{definition}

\newtheorem{remark}[theorem]{Remark}

\renewcommand{\Subset}{\subset\joinrel\subset}

\makeatletter
\newcommand\avsuminner[2]{%
  {\sbox0{$\m@th#1\sum$}%
    \vphantom{\usebox0}%
    \ooalign{%
      \hidewidth
      \smash{\,\rule[.23em]{8.8pt}{1.1pt} \relax}%
      \hidewidth\cr
      $\m@th#1\sum$\cr
    }%
  }%
}
\makeatother

\makeatletter
\newcommand\avsuminnerr[2]{%
  {\sbox0{$\m@th#1\sum$}%
    \vphantom{\usebox0}%
    \ooalign{%
      \hidewidth
      \smash{\,\rule[.23em]{6pt}{0.7pt} \relax}%
      \hidewidth\cr
      $\m@th#1\sum$\cr
    }%
  }%
}
\makeatother

\newcommand{\cref}[1]{Corollary \ref{c.#1}}

\numberwithin{theorem}{section}
\numberwithin{equation}{section}

\newcommand{\diam}{\textup{diam} }

\newcommand{\N}{\mathbb{N}}

\newcommand{\Z}{\mathbb{Z}}
\newcommand{\R}{\mathbb{R}}
\newcommand{\Oa}{\mathcal{O}}

\newcommand{\T}{\mathbb{T}}

\usepackage{mathtools}

\newcommand{\pt}{\partial}

\newcommand{\bca}{\begin{cases}}
\newcommand{\eca}{\end{cases}}

\newcommand{\lb}{\left(}

\newcommand{\rb}{\right)}
\newcommand{\lmb}{\left[}
\newcommand{\rmb}{\right]}
\newcommand{\lma}{\left\{}
\newcommand{\rma}{\right\}}
\newcommand{\Ha}{{\mathcal{H}}}

\newcommand{\Ra}{{\mathcal{R}}}

\newcommand{\La}{{\mathcal{L}}}

\newcommand{\dd}{~\textup{d}}
\newcommand{\dV}{~\textup{d} V}
\newcommand{\dA}{~\textup{d} A}
\newcommand{\drho}{~\textup{d} \rho}
\newcommand{\dtrho}{~\textup{d} \Tilde{\rho}}

\newcommand{\gd}{\nabla}
\newcommand{\rta}{\rightarrow}
\newcommand{\lw}{\left|}
\newcommand{\rw}{\right|}
\newcommand{\be}{\begin{equation}}
\newcommand{\ee}{\end{equation}}

\newcommand{\bt}{\begin{thm}}
\newcommand{\et}{\end{thm}}
\newcommand{\bc}{\begin{cor}}
\newcommand{\ec}{\end{cor}}
\newcommand{\bl}{\begin{lem}}
\newcommand{\el}{\end{lem}}
\newcommand{\norm}[1]{\left\lVert#1\right\rVert}
\newcommand{\avg}[1]{\langle#1\rangle}
\newcommand{\normm}[1]{{\left\vert\kern-0.25ex\left\vert\kern-0.25ex\left\vert #1 
    \right\vert\kern-0.25ex\right\vert\kern-0.25ex\right\vert}}

\newcommand{\dist}{\operatorname{dist}}

\newcommand{\e}{\varepsilon}

\newcommand{\weakcv}{\rightharpoonup}

\def\XXint#1#2#3{{\setbox0=\hbox{$#1{#2#3}{\int}$ }
\vcenter{\hbox{$#2#3$ }}\kern-.6\wd0}}

\usetikzlibrary{decorations.pathmorphing, arrows.meta, patterns}

\usepackage{marvosym}

\def\XXint#1#2#3{{\setbox0=\hbox{$#1{#2#3}{\int}$ }
\vcenter{\hbox{$#2#3$ }}\kern-.6\wd0}}

\newcommand{\ep}{\varepsilon}
\newcommand{\osc}{\mathop{\textup{osc}}}

\begin{document}

\title[Second-order expansion of the Steklov spectrum]{An indefinite Coulomb interaction from the Steklov spectrum\\ of perforated manifolds}

\author[Zhonggan Huang]{Zhonggan Huang}
\thanks{Department of Mathematics, The University of Utah. (Current address: Westlake University)}
\email{zhonggan@math.utah.edu}
\author[Raghavendra Venkatraman]{Raghavendra Venkatraman}
\thanks{Department of Mathematics, The University of Utah.}
\email{raghav@math.utah.edu}
\keywords{Steklov and Laplace-Beltrami eigenvalues on manifolds, quantitative homogenization, Coulomb-type interaction energy, reduced Green function}
\subjclass{35P15, 58J50, 35B27, 35C20, 35J08}

\begin{abstract}
We establish optimal convergence rates for Steklov eigenvalues and
harmonically extended eigenfunctions toward their weighted
Laplace--Beltrami counterparts on a closed manifold perforated by
many small geodesic balls. The holes have radii that scale critically with respect to their spacing in the sense that the boundary area of each hole balances with the weighted volume of its Voronoi cell. We then derive a higher-order expansion
of the Steklov eigenvalues; in dimensions two and three, we identify
two next-to-leading terms in an asymptotic expansion in~$\e$. The expansion is governed by an indefinite
Coulomb-type energy of the discrepancy between the boundary and
bulk measures, mediated by the reduced Green function of the
limiting operator. The proof relies on sharp estimates for certain auxiliary functions that we introduce in order to quantify the discrepancy measure between the surface measure on the holes and their background density. Our paper serves to bridge an emerging literature in spectral geometry with one on systems of points interacting via Coulomb-type energies.
\end{abstract}

\maketitle

\section{Introduction}
\label{s.intro}

\smallskip \noindent This paper is motivated by a recent thread of research in spectral geometry originating from the seminal works of Fraser-Schoen~\cites{fraser_schoen_2011,fraser_schoen_2016}. These works show that a properly immersed surface~$\Sigma\subset\mathbb{B}^n$ is a \emph{free boundary minimal surface} if and only if its coordinate functions are Steklov eigenfunctions with eigenvalue~$1$. Furthermore, metrics maximizing the perimeter-normalized first Steklov eigenvalue~$\sigma_1(\Sigma,g)\,L_g(\pt\Sigma)$ are exactly those induced by such immersions via first eigenfunctions. 

The natural question that these works raise, about how large normalized Steklov eigenvalues can be when maximized over metrics in a certain class, was considered in the paper~\cites{girouard_continuity_2021}. As a fundamental tool, the authors of this paper consider a maximizing sequence of subdomains~$\Omega_\e$ of a manifold~$M$, whose boundary measures, suitably normalized, weakly converge to a diffuse measure on~$M$ with prescribed positive, smooth density~$\beta$ (a ``background charge''). With this construction at hand, they carry out what they call \emph{manifold homogenization}: perforate a closed manifold~$M$ with (near maximal) Riemannian metric~$g$ by many small, well-separated and uniformly distributed geodesic balls. Imposing a Steklov boundary condition on these holes, they show that the associated Steklov spectrum of the perforated domain converges to the spectrum of a weighted Laplacian on all of~$M$. This procedure generates, among other things, metrics with large normalized Steklov eigenvalues. A parallel analysis for Euclidean domains appears in~\cites{girouard_steklov_2021}. The authors in~\cites{karpukhin2024steklov} further investigate the large-topology limit of infinitely many boundary components, showing through eigenvalue optimization. Their analysis leads to the construction of examples of free boundary minimal surfaces in
a compact domain converging to closed minimal surfaces in the boundary.

With the exception of~\cite{karpukhin2024steklov}, the convergence results described above are largely qualitative. In the present paper we prove \textit{optimal estimates} for this convergence, and then turn to the question of the \textit{next-order expansion}. 

To be precise, we prove that Steklov eigenvalues converge to their Laplace--Beltrami counterparts at the optimal rate~$\omega_d(\e) := \e^{\frac{d}{d-1}}$ (with an extra~$|\log \e|$ factor in~$d=2$). In addition, we prove a rate of convergence for eigenfunctions at rate~$\sqrt{\omega_d(\ep)}$ in~$H^1(M)$ norm after they have been suitably extended to all of~$M$. Finally, to show that this rate is optimal,  we then compute the next-order term in the expansion of a Steklov eigenvalue~$\sigma_\ep$ about its Laplace--Beltrami counterpart~$\lambda$. The next-order term is an \emph{interaction energy of indefinite Coulomb type} between the holes, mediated by the Green function of the operator~$\Delta+\lambda\beta$. The perforated manifold behaves, to this order, like a system of (smeared out) charges on~$M$ interacting with a background charge~$\beta$ through an indefinite electrostatic kernel. 

Our second result on interaction energies serves to bridge a body of literature in spectral geometry to another actively emerging body of work in the Calculus of Variations -- Coulomb gases~\cites{Serfaty}. In terms of this latter  body of research, the discrepancy between a Steklov eigenvalue on the perforated domain and the corresponding Laplace--Beltrami eigenvalue is given by \textit{a modulated energy} mediated by a Coulomb-type kernel, evaluated on the discrepancy measures between the surface area on the perforation boundaries, and the given background charge~$\beta$ that these boundaries approximate (See~\eqref{eq.intro.mu} below). 

\subsection{Setup and main results}
Let~$(M,g)$ be a closed, connected, smooth Riemannian manifold of dimension~$d\ge 2$, and let~$\beta \in C^\infty(M;\R_+)$ be a weight satisfying~$S^{-1}\le \beta \le S$ and~$\norm{\beta}_{C^1(M)}\le S$. For~$\ep>0$ let~$S_\ep\subset M$ be a maximally~$\ep$-separated point set inducing a partition by Voronoi cells~$\{V_p^\ep\}_{p\in S_\ep}$. For each~$p \in S_\ep$ let~$r_{\ep,p}>0$ be the unique radius determined by the \emph{local balance of mass}
\be
\label{eq.intro.balance}
\Ha^{d-1}\big(\pt B_{r_{\ep,p}}(p)\big) = \int_{V_p^\ep}\beta\dV,
\ee
so that~$r_{\ep,p}\simeq \ep^{d/(d-1)}$ (See Remark \ref{r.generalizeDefofOmega} for how to modify the $\beta$ in \cites{girouard_large_2021} to fit in the condition \eqref{eq.intro.balance}). The perforated domain is
\be
\label{eq.intro.Omega}
\Omega=\Omega_\ep := M\setminus \bigcup_{p\in S_\ep} \overline{B}_{r_{\ep,p}}(p).
\ee
We consider the Steklov problem on~$\Omega_\ep$ and the weighted closed eigenvalue problem on~$M$,
\be
\label{eq.intro.problems}
\bca
-\Delta v_k^\ep = 0 & \textup{in }\Omega_\ep\\
\pt_\nu v_k^\ep = \sigma_k^\ep v_k^\ep & \textup{on }\pt\Omega_\ep,
\eca
\qquad \textup{ and }\qquad
-\Delta u_k = \lambda_k \beta u_k \ \textup{ in } M,
\ee
with spectra~$0=\sigma_0^\ep < \sigma_1^\ep \le \cdots$ and~$0=\lambda_0<\lambda_1\le\cdots$. (Since~$M$ is closed, \eqref{eq.intro.problems} is the natural analogue of a weighted Laplace--Beltrami problem, and we refer to~$\lambda_k$ as Laplace--Beltrami eigenvalues.) We write~$u_k^\ep$ for the harmonic extension of~$v_k^\ep$ to~$M$, and we encode the discrepancy between bulk and surface measures through
\be
\label{eq.intro.mu}
\dd\mu_\ep := \beta \dV - \restr{\dA}{\pt\Omega_\ep}.
\ee
The balance of mass~\eqref{eq.intro.balance} makes~$\mu_\ep(V_p^\ep)=0$ for every cell: each hole is \emph{neutral}, which suppresses the monopole contribution and is ultimately responsible for the sharp rates below. The natural small parameter of the problem is
\be
\label{eq.intro.omegad}
\omega_d(\ep) :=
\bca
\ep^{\frac{d}{d-1}} & \textup{if } d\ge 3,\\
\ep^2|\log\ep| & \textup{if } d=2.
\eca
\ee
Precise definitions and standing assumptions are collected in Section~\ref{section.preliminaries}.

The starting point is the qualitative convergence result of~\cites{girouard_large_2021,girouard_steklov_2021}.

\begin{theorem}[\cites{girouard_large_2021,girouard_steklov_2021}]
    \label{t.qualitativeconverge}
    Let $(v_k^\ep, \sigma_k^\ep)$ be any $k$-th eigenpair of the Steklov eigenvalue problem \eqref{eq.Steklov} and $u_k^\ep$ be the unique harmonic extension of $v_k^\ep$ to the whole manifold $(M,g)$, then for a $k$-th eigenpair $(u_k,\lambda_k)$ to \eqref{eq.Neumann} we have as $\ep\rta0$
    \be
\label{eq.converge.stektoNeum.value.GiLa}
\sigma_k^\ep \rta \lambda_k \textup{ and }u_k^\ep \weakcv u_k \textup{ in }H^1(M),
    \ee
    after passage to a subsequence in $u_k^\ep$. Moreover, if an eigenvalue $\lambda$ to \eqref{eq.Neumann} admits $m\ge 1$ eigenfunctions $u_{k},\dots,u_{k+m-1}$, which forms an orthonormal basis for the eigenspace, then there is a family of $m\times m$ orthogonal matrices $M(\ep)$ and eigenfunctions (harmonically extended) $u_k^\ep,\dots,u_{k+m-1}^\ep$ such that as $\ep\rta0$
    \be
\label{eq.converge.stektoNeum.function.GiLa}
M(\varepsilon) \begin{bmatrix} u_k^\varepsilon \\ \vdots \\ u_{k+m-1}^\varepsilon \end{bmatrix} \rightharpoonup \begin{bmatrix} u_k \\ \vdots \\ u_{k+m-1} \end{bmatrix} \quad \text{in } H^1(M).
    \ee
\end{theorem}

Our first main result upgrades~\eqref{eq.converge.stektoNeum.value.GiLa} and~\eqref{eq.converge.stektoNeum.function.GiLa} to optimal rates: the eigenvalues converge at rate~$\omega_d(\ep)$ and the eigenfunctions at rate~$\sqrt{\omega_d(\ep)}$, uniformly across eigenvalue clusters of multiplicity~$m\ge 1$.

\begin{theorem}[Optimal convergence rate for the eigenvalues and eigenvectors]
    \label{t.quantitative.StekNeum}
   Under the assumptions of Theorem \ref{t.qualitativeconverge} with an additional local balance of mass condition \eqref{eq.localbalanceofmass}. Then we have the following optimal estimate for $\ep < \ep_0(M,d,g,S,\lambda_k)$
    \be
\label{eq.Optestimate.stektoNeum}
|\lambda_k - \sigma_k^\ep| \le C(M,d,g,S,\lambda_k)~ \omega_d(\ep).
    \ee
  Moreover, if an eigenvalue $\lambda$ to \eqref{eq.Neumann} admits $m\ge 1$ eigenfunctions $u_{k},\dots,u_{k+m-1}$ that satisfy
    \[
  \int_{M} \beta u_i u_j \dV = \delta_{ij} \textup{ for all }i,j=k,\dots, k+m-1,
  \]
then there is a family of $m\times m$ orthogonal matrices $\Oa_\ep$ such that $\Oa_\ep \rta \Oa_0$ as $\ep\rta0$, and Steklov eigenfunctions (harmonically extended) $u_k^\ep,\dots,u_{k+m-1}^\ep$ such that
  \[
  \int_{\pt\Omega} u_i^\ep u_j^\ep \dA = \delta_{ij} \textup{ for all }i,j=k,\dots, k+m-1
  \]
 and for $\ep < \ep_0(M,d,g,S,\lambda_k)$
    \be
\label{eq.Optestimate.function}
\max_{k\le j \le k+m-1}\norm{u_j - (\Oa_\ep)_{jl}u_l^\ep}_{L^2(\pt\Omega)} + \norm{u_j - (\Oa_\ep)_{jl} u_l^\ep}_{H^1(M)} \le C \sqrt{\omega_d(\ep)},
    \ee
    where we have used the Einstein summation convention. In particular, if we denote $\Tilde{u}_k$ as the $L_\beta^2(M)$-projection of $u_k^\ep$, then we have
     \be
\label{eq.Optestimate.function.proj}
\norm{u_k^\ep - \Tilde{u}_k}_{L^2(\pt\Omega)} + \norm{u_k^\ep - \Tilde{u}_k}_{H^1(M)} \le C \sqrt{\omega_d(\ep)}.
    \ee
    The constants $C$ above depend at most on $M,d,g,S$ and $\lambda$.
\end{theorem}

The rate~$\omega_d(\ep)$ is optimal, and we obtain the next-order term in the expansion of the Steklov eigenvalue~$\sigma_\e$ about the Laplace--Beltrami eigenvalue~$\lambda.$ For expository purposes we state this theorem for a simple eigenvalue, but in Section~\ref{s.expansion} we carry out the expansion about a general eigenvalue~$\lambda$ with multiplicity~$m\geq 1.$ Our expansion implies that the leading order correction when~$\sigma_\e$ is expanded about~$\lambda$ is a \textit{Coulomb-type interaction energy} between the centers~$S_\e$ of the perforated domain. However, as expected for~$\lambda>0,$ in our context the Coulomb interaction energy is an indefinite functional. Towards making this explicit, we introduce some notation. For a fixed base point~$y\in M,$ we let~$G_\lambda$ denote the Green function satisfying
\begin{equation} \label{e.greenfunctdef}
    (\Delta + \lambda \beta)G_\lambda(\cdot,y) = \delta_y - \beta(\cdot)U(\cdot)U(y)~ \mbox{ in } M,\,\mbox{ with } \int_M G_\lambda(x,y) U(x)\beta(x)\,\dV(x) = 0\,. 
\end{equation}
In Appendix~\ref{s.app} we describe the construction and some properties of this Green function~$G_\lambda$ that we use in the sequel. It is immediate that for nearby~$x\neq y \in M,$ the function~$G_\lambda$ has a logarithmic singularity in~$d=2$ and a Coulomb-type singularity~$|x-y|^{2-d}$ in~$d\geqslant 3.$ 

\smallskip \noindent 
Naturally, the second mean-zero condition in~\eqref{e.greenfunctdef} is a normalization in order to make~$G_\lambda$ unique, since the PDE only defines it up to an arbitrary multiple of~$U.$ Making an analogy to the terminology used in~\cites{Serfaty} the reader will observe that the interaction energy has the form of an \textit{indefinite modulated energy}. To see the indefiniteness explicitly, it suffices to recall that by the properties listed in Appendix~\ref{s.app}, for any sufficiently regular~$f$ we have
\begin{equation*}
\iint_{M\times M}
G_\lambda(x,y)f(x)f(y)
\beta(x)\beta(y)\,\dV_x\dV_y
=
\sum_{\lambda_j<\lambda}
\frac{
\left(\displaystyle\int_M f u_j\beta\,\dV\right)^2
}{
\lambda-\lambda_j
}
-
\sum_{\lambda_j>\lambda}
\frac{
\left(\displaystyle\int_M f u_j\beta\dV\right)^2
}{
\lambda_j-\lambda
}.
\end{equation*}
Let $\lambda=\lambda_k$ be a simple weighted Laplace--Beltrami
eigenvalue, let $U$ be an associated eigenfunction normalized in
$L^2(M,\beta\dV)$, and let $\sigma_\e=\sigma_k^\e$ be the
corresponding Steklov eigenvalue. Then the following expansions hold.
\begin{theorem}
    \label{t.interaction energy}
    If~$\lambda$ is the~$k$-th simple Laplace--Beltrami eigenvalue, and~$\sigma_\e$ the~$k$-th Steklov eigenvalue, then if~$d=2,$ then 
    \begin{equation}
        \label{e.d=2interaction}
        {\sigma_\e - \lambda}=  \underbrace{ \lambda^2\iint_{M\times M} G_\lambda(x,y) U(x)U(y)\dd \mu_\e(x) \dd \mu_\e(y)}_{\e^2|\log \e|} ~+~ \underbrace{+\lambda \int_M U^2 \dd\mu_\ep-2\int_{M\setminus \Omega} |\gd U|^2 \dV }_{\ep^2} ~+~ o( \e^2)\,,
    \end{equation}
if $d = 3$, then
\begin{equation}
   {\sigma_\e - \lambda}= \underbrace{\lambda^2 \iint_{M\times M}G_\lambda(x,y) U(x)U(y)\dd \mu_\e(x) \dd \mu_\e(y) -\frac{3}{2} \int_{M\setminus \Omega} |\gd U|^2 \dV}_{\ep^{3/2}} ~+~ \underbrace{ \lambda\int_M U^2 \dd\mu_\ep}_{\ep^2} ~+~ o(\ep^2)\,,
\end{equation}
and if~$d \geqslant 4, $ then 
\begin{equation}
    {\sigma_\e - \lambda}= \underbrace{\lambda^2\iint_{M\times M} G_\lambda(x,y) U(x)U(y)\dd \mu_\e(x) \dd \mu_\e(y) -\frac{d}{d-1} \int_{M\setminus \Omega} |\gd U|^2 \dV}_{\ep^{d/(d-1)}} ~+~ O(\omega_d(\ep)^{3/2})\,. 
\end{equation}
\end{theorem}
\begin{remark}
    \label{r.multiple}
    As earlier mentioned, a suitable modification of this theorem to cover the case of a higher multiplicity~$\lambda$ is possible; and in Section~\ref{s.expansion} (Specifically Lemma \ref{l.expansion}) we provide details of how such an expansion can also be proven in that case. 
\end{remark}
\begin{remark}
    \label{r.dynamical}
    In~\cites{girouard_steklov_2021} the authors, motivated both by spectral geometric questions in Euclidean domains, prove the convergence of the Steklov eigenvalues in critically perforated Euclidean domains to the so-called \textit{dynamical} eigenvalue in the~$\e \to 0$ limit. We expect that the tools developed in our paper should adapt to that setting as well, to quantify the next-order term in the expansion, though we have not checked the details. 
\end{remark}

\smallskip \noindent The indefinite character of the interaction energy term in Theorem~\ref{t.interaction energy} is related to the fact that the operator~$(\Delta + \lambda \beta)$ that~$G_\lambda$ ``inverts'' (on the orthogonal complement of the kernel of the operator), is indefinite, since the associated quadratic form is~$\displaystyle((\Delta + \lambda \beta) u, u) = \int_M  (\lambda \beta u^2 - |\nabla u|^2 )\dd x.$ 

\subsection{Roadmap of the proof.}
\label{ss.ideas}

The analysis of the optimal convergence rate is based on the following formula obtained through a simple integration by parts (See \eqref{eq.atrickintegrate} and \eqref{eq.atrickintegrate.2} for details)
\be
\label{eq.atrickintegrate.intro}
    (\sigma_k^\ep - \lambda_k) \int_{\pt\Omega} u^\ep ~u~ \dA = \lambda_k\int_M u^\ep u \dd \mu_\ep -  \int_{M \setminus \Omega} \gd u^\ep \cdot \gd u \dV=: J_1 + J_2 .
 \ee
It is known in \cite{girouard_large_2021} that there are appropriately chosen $k$-th Steklov eigenfunctions $u^\ep$ and Laplace--Beltrami eigenfunctions $u$ so that $\int_{\pt\Omega} u^\ep ~u~ \dA \rta \norm{u}_{L_\beta^2(M)}^2$ as $\ep\rta 0$. Therefore, to understand the difference $\sigma_k^\ep -\lambda_k$, it suffices to estimate $J_1$ and $J_2$.

\subsubsection{Global cell function}
To bound the term $J_1$, the authors in \cites{girouard_large_2021,girouard_steklov_2021} introduced the following \emph{local} cell functions $\Psi_p^\ep$ on each Voronoi cell $V_p^\ep$
\be
\label{eq.local.cellfunction.intro}
\bca
    \Delta \Psi_p^\ep = c_{\ep,p} & \textup{ in } V_p^\ep \setminus \bar B_{r_{\ep,p}}(p)\\[5pt]
    \pt_\nu \Psi_p^\ep = 1 & \textup{ on }\pt B_{r_{\ep,p}}(p)\\[5pt]
    \pt_\nu \Psi_p^\ep = 0 &\textup{ on }\pt V_p^\ep,
\eca
\ee
where $c_{\ep,p}\approx \beta(p)$ as $\ep\rta 0$. They proved that these cell functions satisfy $\norm{\gd \Psi_p^\ep}_{L^2(V_p^\ep)} \le C \ep^{\frac{d+1}{2}}$ (See \cite{girouard_large_2021}*{Equation (36)}), which leads to the sub-optimal bound
\be
\label{eq.acrude.J1bound.intro}
|J_1| \le C \lb \sum_{p\in S_\ep} \norm{\gd \Psi_p^\ep}_{L^2(V_p^\ep)}^2 \rb^{1/2} \le C \ep^{1/2}.
\ee

Our strategy is different: the use of a cell formula that is based on a single Voronoi cell at a time, is reminiscent of periodic homogenization. However, when the points are merely~$\e$-maximally separated but otherwise disordered, it is more natural to make an analogy with stochastic homogenization, where the cell problem is defined globally, rather than on a single cell. Specifically we introduce the following \emph{global} cell function $\Phi_\ep$ that satisfies
\be
\label{eq.cellfunction.PDEform.intro}
\bca
\Delta \Phi_\ep = \beta & \textup{ in }M\setminus \pt\Omega_\ep\\
(\pt_\nu^+ - \pt_\nu^-)\Phi_\ep = 1 & \textup{ on }\pt\Omega_\ep,
\eca
\ee
where $\pt_\nu^{\pm}$ are outer normal derivatives from in-/outside $\Omega_\ep$. We emphasize that this definition of a cell function, aside from being globally defined,  is independent of the Voronoi cells generated by the point set~$S_\e.$ We also note that the problem \eqref{eq.cellfunction.PDEform.intro} is naturally well-posed by the assumption \eqref{eq.intro.balance}.

On the other hand, Voronoi cells do enter the picture in a crucial way to construct estimates for our global cell functions $\Phi_\ep$. We refer to Theorem \ref{t.summaryofPhi} for the fine estimates of $\Phi_\ep$ that allow one to bound $|J_1|$ (optimally) by $C\omega_d(\ep)$ in Lemma \ref{l.Optestimate.stektoNeum.J1}. We also mention that there is a simple example on the 2D torus showing the optimality of the $L^\infty$ bound of $\Phi_\ep$ in Theorem \ref{t.summaryofPhi}.

\subsubsection{Estimation of the Dirichlet-to-Neumann term}

The second main obstacle is a fine estimation for $J_2$ in \eqref{eq.atrickintegrate.intro}. By a naive Cauchy-Schwarz inequality one obtains the following sub-optimal bound
\[
|J_2| \le C \norm{\gd u^\ep}_{L^2(M\setminus \Omega)} \norm{u}_{C^1(M)} |M\setminus \Omega|^{1/2} \le C \sqrt{\omega_d(\ep)}.
\]
To make an optimal estimate, we flatten the manifold locally near each point $p\in S_\ep$, and write (allowing a small error)
\[
J_2 \approx - \sum_{p\in S_\ep} \gd_z u(p) \cdot \int_{B_{r_{\ep,p}}(0)} \gd_z u^\ep \dd z = -\sum_{p\in S_\ep} \int_{\pt B_{r_{\ep,p}}(0)} \lb\gd_z u(p) \cdot \frac{z}{|z|} \rb u^\ep \dd \Ha^{d-1}(z),
\]
where for each fixed $p\in S_\ep$,  $z\in \R^d \cong T_pM$ is a geodesic coordinate system. One can bound $J_2$ by $C\ep^{d/(d-1)}$ for $d\ge 2$ by integrating by parts in the annulus $B_{R_\ep}(0)\setminus B_{r_{\ep,p}}(0)$ (with $R_\ep = O(\ep)$) in the right hand side of the above display with the following auxiliary function
\[
w(z):= \gd_z u(p)\cdot\lb\frac{z}{|z|}\rb\,\frac{r_{\ep,p}^{d-1}\bigl(R_\ep^{d} - |z|^{d}\bigr)}{|z|^{d-1}\bigl(R_\ep^{d} - r_{\ep,p}^{d}\bigr)}.
\]
We refer to Lemma \ref{l.Optestimate.stektoNeum.J2} for details.

\subsubsection{Corrector problem and the second-order expansion}
For the convenience of presentation, we assume that the $k$-th Laplace--Beltrami eigenvalue is simple and $\sigma_\ep$ is the $k$-th Steklov eigenvalue. We refer to Section \ref{s.expansion} for the case when $\lambda$ has higher multiplicities.

The proof of the expansion of the Steklov eigenvalues $\sigma_\ep$  in Theorem \ref{t.interaction energy} is based on the following ``corrector'' equation
\begin{equation}
\label{eq.CorrectorZep.intro}
\bca
(\Delta+\lambda\beta)Z_\ep
=
\lambda U\dd\mu_\ep
-
\delta_\ep\beta U\dV &
\\[5pt]
\int_M Z_\ep U\beta\dV&
\eca
\textup{ and }\quad
\delta_\ep
:=
\lambda\int_M U^2\dd\mu_\ep.
\end{equation}
Here $U$ is an eigenfunction associated with $\lambda$, unique up to
sign, normalized by
\[
\|U\|_{L^2_\beta(M)}=1.
\]
Formally one expands the $k$-th Steklov eigenfunction $U_\ep$ as
\[
U_\ep = U + Z_\ep + E_\ep \textup{ and } \sigma_\ep = \lambda + \delta_\ep + \rho_\ep.
\]
By Lemma \ref{l.expansion} we can show that 
\[
\delta_\ep \le C \ep^2 \textup{ and }\rho_\ep = \lambda \int_M Z_\ep U \dd\mu_\ep - \frac{d}{d-1} \int_{M\setminus \Omega} |\gd U|^2 \dV + O(\omega_d(\ep)^{3/2}).
\]
Note that the leading term above that involves $Z_\ep$ can be rewritten by recalling the $\lambda$-Green functions defined in \eqref{e.greenfunctdef}
\be\label{eq.Zrepresentation.intro}
\int_M Z_\ep U\dd\mu_\ep
=
\lambda
\iint_{M\times M}
G_\lambda(x,y)U(x)U(y)\dd\mu_\ep(x)\dd\mu_\ep(y)\,.
\ee
We refer the reader to Corollary \ref{c.Zepsilon} for details.

\subsection{Literature review} \label{ss.lit} The study of partial differential equations in perforated domains has a long history, and a survey of prior work in this generality is a formidable task. Here, we describe work that is closely related to our paper. 

Rather than proceed chronologically, due to our motivation from the spectral geometry~\cites{girouard_large_2021,girouard_steklov_2021}, we first describe the connection to that literature a bit further. The recent preprint~\cites{chu_stern_2025} offers a variational viewpoint on minimal surface doublings by harnessing a Coulomb-type interaction energy. Roughly speaking, in~\cites{chu_stern_2025}, a multiplicity-2 minimal surface~$\Sigma\subset (N^3,g)$ is constructed by a variational gluing procedure starting from a nondegenerate critical point of an interaction energy related to Green function of the Jacobi operator of~$\Sigma$. The equilibrium measures of the energy~$\displaystyle E_\e(\mu)=\iint G(x,y)\dd\mu(x)\dd\mu(y)$ govern the distribution of catenoidal necks in the large-genus limit.

Our Theorem~\ref{t.interaction energy} exhibits the same mechanism at leading order in a purely spectral setting: the correction~$\sigma_\ep - \lambda$ is a discrete Coulomb energy for the (indefinite) Green function of~$\Delta+\lambda\beta$, evaluated on the signed measure~$\mu_\ep$ carried by the hole centers. In~\cites{chu_stern_2025} the interaction energy is the input of the gluing method used in that paper, since nondegenerate critical points yield minimal surfaces. On the other hand, in our setting, a Coulomb-type interaction energy arises from a higher order asymptotic expansion. We believe that using our high order expansion result in Theorem~\ref{t.interaction energy}, that we have carried out for a fixed metric~$g$ on~$M$, paves the way for variational tools such as quantitative~$\Gamma-$convergence, for applications to optimization problems with respect to~$g$ (see~\cites{chu_stern_2025,Serfaty}). Exploring such opportunities is an interesting open direction.

In the last decade, thanks to the aforementioned connection to spectral geometry, the Steklov problem has received a lot of attention, and we refer the reader to surveys~\cites{girouard_polterovich_2017,colbois_girouard_gordon_sher_2024}. A significant part of this literature concerns questions involving \textit{optimization} of a certain geometric scale-invariant quantity with respect to a class of metrics. Broadly speaking, a natural viewpoint on such problems is that of \textit{shape optimization}, and there is a fairly mature suite of tools that are based on the theory of \textit{homogenization} that are used to address such questions in the unrelated applied field of shape optimization~\cites{allaire1997shape}. What is novel in our paper is that  we carry out a \textit{quantitative second-order homogenization,  on a closed manifold}, on which there is no natural ``grid''.  Let us also mention that recently, quantitative \textit{stochastic} homogenization has been used to obtain optimal spectral convergence rates in the significantly more singular setting of a random geometric graph built from independent and identically (i.i.d) sampled points on a closed manifold~\cites{AV3,GTLV}. 

Making elegant use of this connection, the authors of~\cites{girouard_large_2021,girouard_steklov_2021} anticipate that in certain metric optimization problems, maximizing sequences develop microstructure with boundaries concentrating on many small components can approximate a bulk measure. It is this observation underlies the geometric set-up of those papers, and hence,  also our paper.  

Outside of spectral geometry, there is
 a vast literature on elliptic and even time-dependent problems posed on  perforated domains to which we believe the tools we build in this paper are of independent interest, and we do not attempt  to survey this literature exhaustively. 
The homogenized effect of many small holes was first identified by~\cites{marchenko_khruslov_2006,PV}, and independently popularized by the authors of~\cites{cioranescu_murat_1997}, who coined the term \emph{``le terme \'etrange venu d'ailleurs''} for the capacitary potential appearing in the limit of the Dirichlet problem at the critical scaling~$r_\ep \sim \ep^{d/(d-2)},$ where~$\e$ is the size of the periodic spacing of the inclusions and~$r_\e$ their radius. We refer the reader to~\cites{rauch_taylor_1975} for the probabilistic ``crushed ice'' problem. Spectral problems in perforated Euclidean domains have also received attention with a wide range of motivations, from fluids to wave propagation in highly heterogeneous/composite media, see~\cites{vanni,chiado_piat_nazarov_piatnitski_2012,cherednichenko_dondl_rosler_2018,borisov_2024,cherednichenko2025effective}. Finally, the question of variational problems on heterogeneous domains also arises in problems from materials science, see~\cites{calderer_desimone_golovaty_panchenko_2014,canevari_zarnescu_2020,canevari_cherednichenko_zarnescu_2025}.

Interaction energies of Coulomb-type arise in numerous contexts in mathematical physics~\cites{lewin2022coulomb}. Our viewpoint in this paper is closest to the study of Ginzburg-Landau vortices modeling superconductivity, see ~\cites{sandier2012ginzburg,serfatyICM2018}, where a similar Coulomb interaction energy arises after a renormalization procedure. More broadly, such interaction energies also arise in measure quantization problems and quadrature problems on closed manifolds, see~\cites{betermin-sandier,Serfaty} for a broad discussion and further references. 

 \subsection{Organization of the paper.} In Section~\ref{section.preliminaries}
we collect all the notation that is used in the manuscript. In Section~\ref{section.estimatesforglobalcellfunction} and particularly Theorem \ref{t.summaryofPhi} we prove optimal estimates for the global cell function introduced in \eqref{eq.cellfunction.PDEform.intro}. We use these estimates to prove Theorem~\ref{t.quantitative.StekNeum} in Section~\ref{th1.2}. We prove Theorem~\ref{t.interaction energy} in Section~\ref{s.expansion}. Finally, in Appendix~\ref{s.app} we collect a few auxiliary results about the $\lambda$-Green functions $G_\lambda$ that arises in the interaction energy. 
\section*{Declaration on AI use}
The main idea, calculations and the writing of this paper is completely accomplished by the human authors. There is no AI-written content in the paper. AI was only used to proofread and copy-edit the paper. 
\section*{Acknowledgements}
The research of R.V. was supported, in part, by an NSF grant DMS-2407592.  Z. H. was supported partially by an NSF grant No. DMS-2407235 and DMS-2407592. The authors acknowledge helpful discussions with William Feldman. The authors also acknowledge Braxton Osting and Haoyu Chen for pointing out a sign error in an earlier draft of the article. Finally, both authors are happy to acknowledge an NSF RTG grant (NSF DMS-2136198) to the University of Utah Department of Mathematics.

\section{Preliminaries}
\label{section.preliminaries}

\subsection{Notation and assumptions}

\begin{enumerate}

    \item Let $d\ge 2$ and $(M,g)$ denote a $d$-dimensional smooth connected compact Riemannian manifold without boundary.

    \item For any $q\in M$, we denote
    \be
    \label{eq.define.exponential}
    \exp_q: T_qM \rta M
    \ee
    as the exponential map based at $q$. We denote $\textup{inj}(M,g)>0$ as the largest radius $r$ such that the map $\exp_q$ is injective on $B_r(0)$ for all $q\in M$.
    \item We denote$\dV$ as the Riemannian volume measure and$\dA=\dd\Ha^{d-1}$ as the ($d-1$)-dimensional Hausdorff measure on the manifold $M$. 
    \item For every $\ep>0$ we denote $S_\ep$ a point set that is maximally $\ep$-separated, and for each $p\in S_\ep$ we define the following Voronoi cell
\be
\label{eq.Voronoicell}
V_p^\ep:=\lma x\in M \;;\; \dist_g(x,p) < \dist_g(x,q) \textup{ for all }q\in S_\ep \rma.
\ee
\item Note that for $\ep< \ep_0(M,d,g)$ and each $p\in S_\ep$, the Voronoi cells $V_p^\ep$ are geodesically convex with piecewise smooth boundary, and for some $C(M,g)>0$ the volume of each cell satisfies
\be
\label{eq.Voronoicell.volumeestimate}
C^{-1}\ep^d \le |V_p^\ep| \le C \ep^d.
\ee
\item  By maximality of $S_\ep$ we also have $B_{\ep/4}(p)\subset V_p^\ep \subset B_{\ep}(p)$. We refer to \cites{girouard_large_2021} for the construction of such sets in detail.

\item For any domain $U\subset M$ that does not intersect $S_\ep$ and any function $f\in C(M)$ that is $C^1$ on $\overline{U}$ and $M\setminus U$ respectively, we define
\be
\label{eq.directionofNormalderivative.positive}
\pt_\nu^+ f(x) : = \lim_{U\ni y \rta x} \gd f(y) \cdot \nu(x),
\ee
where $\nu(x)$ is the unit outer normal direction at $x$ with respect to $U$. Similarly, we define 
\be
\label{eq.directionofNormalderivative.negative}
\pt_\nu^- f(x) : = \lim_{\overline{U}\not\ni y \rta x} \gd f(y) \cdot \nu(x).
\ee

\item Let $\beta\in C^\infty(M; \R_+)$ be a weight function. In general, we will assume that $\beta$ satisfies the following uniform bound
\be
\label{eq.uniformbound.weights}
S^{-1} \le \beta \le S \textup{ and }\norm{\beta}_{C^1(M)} \le S.
\ee
for some positive constant $S$.

\item (Perforated domains) We define for each $\ep>0$ a domain $\Omega=\Omega_\ep\subset M$ as
\be
\label{eq.defineOmega}
\Omega_\ep : = M \setminus \bigcup_{p\in S_\ep} \overline{B}_{r_{\ep,p}}(p),
\ee
where for $\ep<\ep_0(M,d,g)$ and each $p\in S_\ep$, the radius $r_{\ep,p}>0$ is the unique positive number such that
\be
\label{eq.localbalanceofmass}
A(\pt B_{r_{\ep,p}}(p)) = \int_{V_p^\ep} \beta \dV.
\ee
\begin{remark}
    \label{r.generalizeDefofOmega}
    The definition of the radii $r_{\ep,p}$ and the domains $\Omega$ is made after introducing the weight function $\beta$ in order to meet the balance of mass condition \eqref{eq.localbalanceofmass}. {In practice one can choose appropriate $\beta$ to satisfy this condition.} Specifically in the assumption {\cite{girouard_large_2021}*{Equation (12)}}, we have
    \[
    A(\pt B_{r_{\ep,p}}(p)) = \beta(p) |V_p^\ep| = \int_{V_p^\ep} \beta \dV + O(\ep^{d+1}),
    \]
    where the error can be compensated by replacing $\beta$ by $\beta_\ep := \beta + \phi_\ep$ for some $\phi_\ep \in C^\infty(M)$ such that $$\int_M \phi_\ep \dV = 0, \, \norm{\phi_\ep}_{L^\infty(M)}\le C\ep \textup{ and }\norm{\phi_\ep}_{C^1(M)}\le C $$
    for some $C=C(M,d,g)$. Such a $\phi_\ep$ can be obtained by  rescaling  standard cut-off functions, centering them at each $p\in S_\ep$ and then summing them up over $p$.
\end{remark}

\item We denote
\be\label{eq.defineAeps.and.mueps}
\dA_\ep : = \restr{\dA}{\pt \Omega _\ep} \textup{ and } \dd\mu_\ep : = \beta\dV - \dA_\ep.
\ee

\item We denote $R_\ep = \ep/8$ so that $\dist(B_ {R_\ep}(p), \pt V_p^\ep) \ge \ep/8 $, and also define
\be\label{eq.defineOmegastar}
\Omega^*:= M \setminus \bigcup_{p\in S_\ep} \overline{B}_{R_\ep}(p) \subset \Omega.
\ee

\item We often abuse the notation and denote
\be\label{eq.denote.Bp.Qp}
r_\ep=r_{\ep,p}, \ V_p=V_p^\ep,\ B_p:=B_{r_{\ep,p}}(p)\textup{ and }Q_p:= B_{R_\ep} (p).
\ee
Note that by \eqref{eq.localbalanceofmass}, there is a constant $C(M,d,g,S)>0$
\be
\label{eq.radii.estimate}
C ^{-1} \ep^{\frac{d}{d-1}}\le r_{\ep,p} \le C \ep^{\frac{d}{d-1}} \ \textup{ and }\ C^{-1} \ep^{\frac{1}{d-1}} \le \frac{r_{\ep,p}}{R_\ep} \le C \ep^{\frac{1}{d-1}}.
\ee

\item Throughout this paper the following small parameter will play an important role. We define
\begin{equation}
    \label{e.omegadef}
    \omega_d(\e) := \begin{dcases}
       \e^{\frac{d}{d-1}} \quad  \quad  \ \  \mbox{ if } d \geqslant 3, \\
       \e^2 |\log \e| \quad \mbox{ if } d = 2\,. 
    \end{dcases}
\end{equation}
Note that $\ep^{\frac{d}{d-1}}\le \omega_d(\ep)$ for $d\ge 2$.

\end{enumerate}

\subsection{Steklov and weighted Laplace--Beltrami eigenvalue problem}
\label{subsection.Steklov.Neumann.definition}

In this section we define respectively the Steklov and Laplace--Beltrami eigenvalue problems.
For the Steklov eigenvalues we solve for $\ep>0$ and $k \ge 0$
\be
\label{eq.Steklov}
\bca
-\Delta v_k^\ep = 0 & \textup{ in }\Omega_\ep\\
\pt_\nu v_k^\ep = \sigma_k^\ep v_k^\ep & \textup{ on }\pt \Omega_\ep.
\eca
\ee
Here $\Omega_\ep$ is defined in \eqref{eq.defineOmega}. For the Laplace--Beltrami eigenvalues we solve for $k\ge 0$
\be
\label{eq.Neumann}
-\Delta u_k = \lambda_k \beta u_k  \textup{ in }M.
\ee
Here we allow a weight function $\beta$ that satisfies \eqref{eq.uniformbound.weights}. Correspondingly we can obtain a sequence of Steklov eigenpairs
\be\label{eq.Steklov.eigenpairs}
0= \underbrace{\sigma_0^\ep}_{v_0^\ep} < \underbrace{\sigma_1^\ep}_{v_1^\ep} \le \cdots\le \underbrace{\sigma_k^\ep}_{v_k^\ep} \le \cdots
\ee
and Laplace--Beltrami eigenpairs
\be\label{eq.Neumann.eigenpairs}
0= \underbrace{\lambda_0}_{u_0} <\underbrace{\lambda_1}_{u_1}\le \cdots \le \underbrace{\lambda_k}_{u_k} \le \cdots
\ee
We denote for $k\ge0, \ep>0$ the function $u_k^\ep$ as the unique harmonic extension of $v_k^\ep$. For each $l,k\ge 0$ and $\ep>0$, we may without loss normalize 
\be
\label{eq.normalize.Steklov.and.Neumann}
\int_{\pt\Omega} v_l^\ep v_k^\ep \dA =\int_{\pt\Omega} u_l^\ep u_k^\ep \dA = \int_M \beta u_l u_k \dV = \delta_{lk}.
\ee

\section{Estimates of a global cell function}
\label{section.estimatesforglobalcellfunction}
In this section, we aim to study the unique cell function $\Phi_\ep\in H^1(M)$ that satisfies 
\be
\label{eq.cellfunction}
\int_{M} \gd \Phi_\ep \cdot \gd \eta  \dV+\int_M \eta \dd\mu_\ep =0 \  \textup{ and } \ \int_{M}\beta \Phi_\ep \dV = 0,
\ee
for every $\eta\in H^1(M)$, where
\[
\dd\mu_\ep = \beta\dV - \restr{\dA}{\pt\Omega}=:\beta\dV - \dA_\ep.
\]
In PDE form, the function $\Phi_\ep$ satisfies
\be
\label{eq.cellfunction.PDEform}
\bca
\Delta_g \Phi_\ep = \beta & \textup{ in }M\setminus \pt\Omega\\
(\pt_\nu^+ - \pt_\nu^-)\Phi_\ep = 1 & \textup{ on }\pt\Omega,
\eca
\ee
where $\pt_\nu^{\pm}$ are as defined in \eqref{eq.directionofNormalderivative.positive} and \eqref{eq.directionofNormalderivative.negative}.
To understand the behavior of $\Phi_\ep$ for small $\ep>0$, we consider the following decomposition
\be
\label{eq.decomposePhi}
\Phi_\ep = \bar{\Phi}_\ep + \Ra_\ep.
\ee
Here $\bar{\Phi}_\ep$ is of the following form
\be
\label{eq.majortermincellfunction}
\bar{\Phi}_\ep = \sum_{p\in S_\ep} \Psi_p + \kappa_\ep,
\ee
where $\kappa_\ep$ is a normalizing constant so that $\int_{M} \beta \bar{\Phi}_\ep \dV=0$ and each $\Psi_p$ satisfies the following PDE
\be
\label{eq.equationforPsip}
\bca
\Delta_g \Psi_p = \beta & \textup{ in }Q_p\setminus \pt B_p\\
(\pt_\nu^+ - \pt_\nu^-)\Psi_p = 1 & \textup{ on }\pt B_p\\
\Psi_p = 0 & \textup{ on }\partial Q_p\,,
\eca
\ee
and extended by~$\Psi_p \equiv 0$ on~$M \setminus\overline{Q}_p$, where $B_p:=B_{r_{\ep,p}}(p)$ and $Q_p : = B_ {R_\ep}(p)$ are defined in \eqref{eq.denote.Bp.Qp}. For the convenience of notation, we denote 
\[
\Omega^*:= M \setminus \bigcup_{p\in S_\ep} \overline{Q}_p \subset \Omega.
\]
By the definition of $\bar{\Phi}_\ep$, we know that the remainder $\Ra_\ep$ satisfies the following PDE
\be
\label{eq.equationforRaep}
\bca
\Delta_g \Ra_\ep  = \beta & \textup{ in } \Omega^*\\
\Delta_g \Ra_\ep = 0 & \textup{ in }M\setminus \overline{\Omega^*}\\
(\pt_\nu^+ - \pt_\nu^-) \Ra_\ep = \pt_\nu^-\bar{\Phi}_\ep=:\Tilde{\pi} & \textup{ on }\pt \Omega^*.
\eca
\ee

In the following subsections, we are going to make precise estimates on the functions $\bar{\Phi}_\ep$ and $\Ra_\ep$. As a summary, we prove the following theorem.

\begin{theorem}
    \label{t.summaryofPhi}
    Let $\Phi_\ep$ be as defined in \eqref{eq.cellfunction}, $\bar{\Phi}_\ep$ and $\Ra_\ep$ be as defined in \eqref{eq.equationforPsip} and \eqref{eq.equationforRaep} respectively. Then for all $d\ge 2$ and $\ep< \ep_0(M,d,g)$
    \be
\label{eq.Phibar.summary}
\norm{\bar{\Phi}_\ep}_{L^\infty(M)} \le C {\omega_d(\ep)}, \ \norm{\bar{\Phi}_\ep}_{H^1(M)}\le C \sqrt{\omega_d(\ep)}, \ \norm{\bar{\Phi}_\ep }_{L^1(M)}\le C\ep^2\textup{ and } \norm{\bar{\Phi}_\ep}_{L^2(M)} \le C\ep^{\frac{d}{d-1}}.
    \ee
Moreover, we have 
\be
\label{eq.Raep.summary}
\norm{\Ra_\ep}_{L^2(M)}\le C \ep^2 , \ \norm{\Ra_\ep}_{L^2(\pt\Omega)} \le C \ep^{\frac{d}{d-1}}, \ \norm{\Ra_\ep}_{L^\infty(M)} \le C \ep \textup{ and }\norm{\Ra_\ep}_{H^1(M)} \le C \ep.
\ee
In particular, 
\be
\label{eq.Phi.summary}
\norm{\Phi_\ep}_{L^2(M)} \le C \ep^{\frac{d}{d-1}} , \ \norm{\Phi_\ep}_{L^1(M)}\le C\ep^2, \ \norm{\Phi_\ep}_{L^\infty(M)}\le C\ep \textup{ and }\norm{\Phi_\ep}_{H^1(M)}\le C \sqrt{\omega_d(\ep)}.
\ee
If further the points $p\in S_\ep$ are also the mass centers of their corresponding Voronoi cells, that is if we have for each $p\in S_\ep$
\be
\label{eq.centroid.voronoi}
\int_{V_p^\ep} \exp_p^{-1}(x) \beta(x) \dV = 0 \in T_pM,
\ee
then we can improve the $L^\infty$ bound to
\be
\label{eq.Phi.Ra.improve.summary}
\norm{\Ra_\ep}_{L^\infty(M)}\le C \ep^2|\log\ep|\textup{ and }\norm{\Phi_\ep}_{L^\infty(M)}\le C \omega_d(\ep).
\ee
The above constants $C>0$ depend at most on $d, M , g$ and $S$.
\end{theorem}
\begin{remark}
    It can be shown with a simple example that the scalings obtained in this theorem for the $L^\infty$-norm of $\Phi_\ep$ are optimal, in the sense that for~$M = \mathbb{T}^d$ there is a choice of maximally $\ep$-separated points~$\{x_i\}$ such that the associated ``cell function''~$\Phi_\e$ achieves this scaling. See Section~\ref{ss.counterex}. 
\end{remark}

\subsection{Radial model and estimates for \texorpdfstring{$\bar{\Phi}_\ep$}{P}}
\label{subsection.radialmodel.barPhi}
To estimate $\bar{\Phi}_\ep$, it suffices to estimate each $\Psi_p$. To that end, we need the following radial model functions.  For each $p\in S_\ep$, we define for $0\le r \le R_\ep,$
\begin{equation}
\label{eq.Modelfunction}
L_p^\ep(r):=
\begin{dcases}
r_{\ep,p} \min\left\{ \ln \left( \frac{R_\ep}{r} \right), \ln \left( \frac{R_\ep}{r_{\ep,p}} \right) \right\} & \text{if } d=2 \\[12pt]
\frac{r_{\ep,p}^{d-1}}{(d-2)} \min\left\{ \left( \frac{1}{r^{d-2}} - \frac{1}{R_\ep^{d-2}} \right), \left( \frac{1}{r_{\ep,p}^{d-2}} - \frac{1}{R_\ep^{d-2}} \right) \right\} & \text{if } d\ge 3.
\end{dcases}
\end{equation}
We also define for $0\le r \le R_\ep$
\be
\label{eq.Modelfunction.Bulkadapted}
\bar{L}_{p}^\ep(r): = L_p^\ep(r) + \frac{\beta(p)}{2d}(r^2 - R_\ep^2).
\ee
All these functions are extended by $\bar{L}_{p}^\ep(r)=L_p^\ep(r) =0$ for $r> R_\ep$. We define
\be
\label{eq.defineLa.eps}
\La_\ep(x)  = \sum_{p\in S_\ep}  \La_{\ep,p}(x):= \sum_{p\in S_\ep} \bar{L}_{p}^\ep(\dist(x,p)).
\ee
Note that the support of each $\La_{\ep,p}$ is contained in $\overline{Q}_p\Subset V_p^\ep$.

\begin{lemma}
    \label{l.difference.modelandPhibar}
    For each $p\in S_\ep$ and $\ep < \ep_0(M,d,g)$ the difference $\Psi_p-\La_{\ep,p}$ satisfies
    \be
\label{eq.difference.modelandPhibar}
\norm{\Psi_p-\La_{\ep,p}}_{L^\infty(Q_p)}\le C \ep^3 \textup{ and }\norm{\Psi_p-\La_{\ep,p}}_{C^1(\overline{Q}_p)}\le C \ep^2,
    \ee
    where the constant $C$ depends only on $M,d,g$ and $S$.
\end{lemma}

\begin{corollary}
    \label{cor.boundonpitilde}
    With $\Tilde{\pi}$ as defined in \eqref{eq.equationforRaep}, for $\ep<\ep_0(M,d,g)$ we have
    \be
\label{eq.boundonpitilde}
\norm{\Tilde{\pi}}_{L^\infty(\pt\Omega^*)} \le C \ep \textup{ and }\max_{p\in S_\ep} \osc_{\pt Q_p} \Tilde{\pi} \le C \ep^2.
    \ee
\end{corollary}

\begin{proof}
    Since $\La_{\ep,p}$ is radial in a neighbor of $p\in S_\ep$, we know that $\pt_\nu^- \La_{\ep,p}$ is constant on each $\pt Q_p$, and then the second inequality in \eqref{eq.boundonpitilde} by \eqref{eq.difference.modelandPhibar}. The first inequality also follows from \eqref{eq.boundonpitilde} and a direct estimation on the explicit formula of $\pt_\nu^-\La_{\ep,p}$.
\end{proof}

\begin{corollary}
    
\label{cor.phibar.model.norms}
For $d\ge 2$ and $\ep<\ep_0(M,d,g)$, the normalizing constant $\kappa_\ep$ in \eqref{eq.majortermincellfunction} satisfies $|\kappa_\ep|\le C\ep^2$ and there is 
\begin{equation}
\label{eq.phibar.linfty.h1}
 \norm{\bar{\Phi}_\ep}_{L^\infty(M)}\le C\omega_d(\ep)\textup{ and }
 \norm{\gd \bar{\Phi}_\ep}_{L^2(M)}\le C\sqrt{\omega_d(\ep)}.
\end{equation}
For $1\le \gamma<\infty$,
\begin{equation}
\label{eq.phibar.Lq.correct}
 \norm{\bar{\Phi}_\ep}_{L^\gamma(M)}\le C\mathcal \theta_{d,\gamma}(\ep),
\end{equation}
where
\begin{equation}
\label{eq.Pdq.definition}
 \theta_{2,\gamma}(\ep):=\ep^2,
\end{equation}
and for $d\ge3$,
\begin{equation}
 \theta_{d,\gamma}(\ep):=
 \begin{cases}
   \ep^2, & 1\le \gamma<\dfrac d{d-2},\\[6pt]
   \ep^2|\log\ep|^{1/\gamma}, & \gamma=\dfrac d{d-2},\\[6pt]
   \displaystyle\ep^{\frac d{d-1}+\frac d{\gamma(d-1)}}, & \gamma>\dfrac d{d-2}.
 \end{cases}
\end{equation}
In particular, $\norm{\bar{\Phi}_\ep}_{L^2(M)}\le C\ep^{\frac{d}{d-1}}$ .
\end{corollary}

\begin{proof}
The $L^\infty$ bound in \eqref{eq.phibar.linfty.h1} follows immediately from Lemma \ref{l.difference.modelandPhibar} and the explicit formula in \eqref{eq.Modelfunction} and \eqref{eq.Modelfunction.Bulkadapted}. One also has by the same lemma
    \[
    \begin{split}
       \norm{\sum_{p\in S_\ep} \gd \Psi_p}_{L^2(M)}&\le C \lmb\sum_{p\in S_\ep} \int_{r_{\ep,p}}^{R_\ep} \lb\frac{d}{dr}\bar{L}_{p}^\ep \rb^2 r^{d-1} \dd r\rmb^{1/2} + C \ep^2\\
     &\le C \lmb\sum_{p\in S_\ep} \int_{r_{\ep,p}}^{R_\ep} \lb\frac{d}{dr}{L}_{p}^\ep \rb^2 r^{d-1} \dd r\rmb^{1/2} + C \ep\\
     &\le C \lmb\sum_{p\in S_\ep} r_{\ep,p}^{2(d-1)}\int_{r_{\ep,p}}^{R_\ep}  r^{-(d-1)} \dd r\rmb^{1/2} + C \ep\\
     &\le C \sqrt{\omega_d(\ep)}.
    \end{split}
    \]
To prove \eqref{eq.phibar.Lq.correct}, we observe that by Lemma \ref{l.difference.modelandPhibar}, one has for $\ep<\ep_0(M,d,g)$ and $1\le \gamma <\infty$
    \[
    \norm{\sum_{p\in S_\ep} \Psi_p}_{L^\gamma(M)} \le C \lmb\sum_{p\in S_\ep} \int_{r_{\ep,p}}^{R_\ep} \lb\bar{L}_{p}^\ep \rb^\gamma r^{d-1} \dd r\rmb^{1/\gamma} + C \ep^3,
    \]
    and in particular $\displaystyle|\kappa_\ep|\le C \sum_{p\in S_\ep} \int_{0}^{R_\ep} \lw\bar{L}_{p}^\ep \rw r^{d-1} \dd r + C \ep^3$. 
    It is thus enough to estimate the explicit model $\bar   L_p^\ep$.  
In the case $d=2$, we have there is a constant $C$ such  that 
\[
\begin{split}
   \lmb\sum_{p\in S_\ep} \int_{r_{\ep,p}}^{R_\ep} \lb\bar{L}_{p}^\ep \rb^\gamma r \dd r\rmb^{1/\gamma} &\le \lmb\sum_{p\in S_\ep} \int_{0}^{R_\ep} \lb {L}_{p}^\ep \rb^\gamma r \dd r\rmb^{1/\gamma} + C \ep^2\\
   &\le C \ep^{2-2/\gamma} \lb \int_{0}^{R_\ep} \log^\gamma(R_\ep/r) r \dd r  \rb^{1/\gamma} + C\ep^2\\
 &\le C \ep^{2} \lb \int_{0}^{1} \log^\gamma(1/s) s \dd s \rb^{1/\gamma} +C \ep^2\\
 &\le C \ep^2.
\end{split}
\]
In dimension $d\ge 3$, we have by \eqref{eq.radii.estimate}
\[
\begin{split}
   \lmb\sum_{p\in S_\ep} \int_{r_{\ep,p}}^{R_\ep} \lb\bar{L}_{p}^\ep \rb^\gamma r^{d-1} \dd r\rmb^{1/\gamma} &\le \lmb\sum_{p\in S_\ep} \int_{r_{\ep,p}}^{R_\ep} \lb {L}_{p}^\ep \rb^\gamma r^{d-1} \dd r\rmb^{1/\gamma} + C \ep^2\\
   &\le C \ep^{d-d/\gamma} \lb\int_{r_{\ep,p}}^{R_\ep} r^{d-1} \lb \frac{1}{r^{d-2}} - \frac{1}{R_\ep^{d-2}} \rb^\gamma \dd r\rb^{1/\gamma} + C\ep^2\\
   &\le C \ep^{d-d/\gamma} R_\ep^{d/\gamma - (d-2)} \lb\int_{r_{\ep,p}/R_\ep}^1 s^{d-1}(1/s^{d-2}-1)^\gamma \dd s\rb^{1/\gamma}+C\ep^2\\
   &\le C\ep^2\lb\int_{C^{-1}\ep^{\frac{1}{d-1}}}^1 s^{d-1-\gamma(d-2)} \dd s\rb^{1/\gamma}.
\end{split}
\]
This proves the corollary.

\end{proof}

\begin{proof}[Proof of Lemma \ref{l.difference.modelandPhibar}]

Recall that in geodesic polar coordinate $(r,\varphi)$ based at $p$, for every smooth radial test function $f=f(r)$ 
\be
\label{eq.localCalcofLaplacian}
\Delta f = f'' + \frac{d-1}{r} f' + f' \pt_r \log\sqrt{|g|},
\ee
where for $0\le r<r_0(M,g)$ there is $C=C(\norm{g}_{C_r^2(M)})>0$ such that
\be
\label{eq.boundofChristofelSymbol}
|\pt_r \log\sqrt{|g|}(r,\varphi)| \le C r. 
\ee
The identity in \eqref{eq.localCalcofLaplacian} can be found in {\cite{aubin_nonlinear_1982}*{Page 106}}. This implies that
\[
|\Delta \La_{\ep,p} (x)- \beta(p)| \le  C \ep
\]
for all $x\in Q_p \setminus \pt B_p$.
By using the explicit form \eqref{eq.Modelfunction} and \eqref{eq.Modelfunction.Bulkadapted} the functions $\Psi_p - \La_{\ep,p}\in C(M)$ satisfies
\be
\label{eq.Laequation}
\bca
|\Delta (\Psi_p - \La_{\ep,p})| \le| \beta - \beta(p)| +C\ep\le C\ep & \textup{ in } Q_p\\
\Psi_p - \La_{\ep,p} = 0 & \textup{ elsewhere}.
\eca
\ee
If we write $h_\ep:=\Psi_p - \La_{\ep,p}$, then in a geodesic coordinate $z\in \R^d\cong T_pM$, the function $h_\ep$ solves
\[
\bca
|\pt_{z_i} (A_{ij}(z) \pt_{z_j} h_\ep(z)) | \le C \ep & \textup{ for }z\in B_{R_\ep}(0)\\
h_\ep = 0 & \textup{ on }\pt B_{R_\ep}(0),
\eca
\]
where $|A_{ij}(z)-\delta_{ij}| \le C|z|^2$ and $|\gd_z A_{ij}|(z) \le C |z|$ for $C=C(M,d,g)$ when $\ep<\ep_0(M,d,g)$. Note that after rescaling $h_\ep$ in the form $ h_\ep(R_\ep x)/R_\ep^3 = : \Tilde{h}_\ep(x)$, we have
\[
\bca
|\pt_{x_i} (A_{ij}^\ep(x) \pt_{x_j} \Tilde{h}_\ep(x)) | \le C  & \textup{ for }x\in B_1(0)\\
\Tilde{h}_\ep = 0 & \textup{ on }\pt B_1(0),
\eca
\]
where $A_{ij}^\ep(x)=A_{ij}(R_\ep x)$ satisfies $|A_{ij}^\ep-\delta_{ij}| \le C\ep^2$ and $|\gd_x A_{ij}^\ep|(x) \le C \ep$ for $C=C(M,d,g)$ when $\ep<\ep_0(M,d,g)$. By standard $W^{2,\gamma}$-estimates for elliptic equations with sufficiently large $\gamma$ (See {\cite{gilbarg_elliptic_2001}*{Theorem 9.13}}) and Sobolev Imbedding theorem, we obtain 
\[
\norm{\Tilde{h}_\ep}_{C^1(\overline B_{1}(0))}\le \norm{\Tilde{h}_\ep}_{W^{2,\gamma}(B_{1}(0))}\le C
\]
for some $C>0$ independent of $\ep$. This implies \eqref{eq.difference.modelandPhibar} by using the relation $\Psi_p - \La_{\ep,p} = R_\ep^3\Tilde{h}_\ep(z/R_\ep)$.

\end{proof}

\subsection{An \texorpdfstring{$H^1$}{H}-estimate for the remainder \texorpdfstring{$\Ra_\ep$}{Rs}}

Let us denote $\Tilde{\pi}:=\pt_\nu^- \bar{\Phi}_\ep$ on $\pt \Omega^*$ as in \eqref{eq.equationforRaep}, and then by integration by parts, the function $\Ra_\ep$ satisfies
\be
\label{eq.weakformofRa}
\int_{M} \gd \Ra_\ep\cdot \gd \eta \dV = -\int_{\Omega^*} \beta \eta \dV + \int_{\pt \Omega^*} \Tilde{\pi} \eta \dA,
\ee
for all $\eta \in H^1(M)$. Moreover, we have because of the choice of $\kappa_\ep$ in \eqref{eq.majortermincellfunction}
\be
\label{eq.normalizationofRa}
\int_{M} \beta \Ra_\ep \dV = 0.
\ee

\begin{lemma}
    \label{l.H1estimateforRaep}
    The remainder function $\Ra_\ep $ satisfies the following estimate
    \be
\label{eq.H1estimateforRaep}
\norm{\gd \Ra_\ep}_{L^2(M)} \le C \ep
    \ee
    for all $\ep<\ep_0(M,d,g)$ and some constant $C\lb M,d,g, S\rb>0$.
\end{lemma}

\begin{proof}

By taking $\eta = \Ra_\ep$ in \eqref{eq.weakformofRa} we obtain by also using \eqref{eq.normalizationofRa}
\be\label{eq.H1estimateRa.eq1}
\begin{split}
    \int_{M} |\gd \Ra_\ep|^2 \dV &= -\int_{\Omega^*} \beta \Ra_\ep \dV + \int_{\pt \Omega^*} \Tilde{\pi} \Ra_\ep \dA\\
    &=\sum_{p\in S_\ep} \lmb \int_{Q_p} \beta \Ra_\ep \dV +   \int_{\pt Q_p} \Tilde{\pi} \Ra_\ep \dA \rmb.
\end{split}
\ee
Note that by the compatibility condition of $\Psi_p$ in \eqref{eq.equationforPsip} and the balance of mass in \eqref{eq.localbalanceofmass}, there is 
\be
\label{eq.H1estimateRa.usebalanceofmass}
\begin{split}
    \int_{\pt Q_p} \Tilde{\pi} \dA &= \int_{\pt Q_p} \pt_\nu^- \Psi_p \dA\\
&=\int_{\pt B_p} (\pt_\nu^+ - \pt_\nu^-)\Psi_p \dA - \int_{Q_p} \beta \dV\\
&=A(\pt B_p) - \int_{Q_p} \beta \dV\\
&= \int_{V_p^\ep\setminus Q_p} \beta \dV\\
\end{split}
\ee
This implies that for any constants $h_p$ that depend on each $p\in S_\ep$, there is 
\be
\label{eq.H1estimateRA.addsubtract}
\begin{split}
     \int_{M} |\gd \Ra_\ep|^2 \dV &=\sum_{p\in S_\ep} \lmb \int_{Q_p} \beta \Ra_\ep \dV +   \int_{\pt Q_p} \Tilde{\pi} \Ra_\ep \dA \rmb\\
&=\sum_{p\in S_\ep} \lmb \int_{Q_p} \beta (\Ra_\ep -h_p) \dV +   \int_{\pt Q_p} \Tilde{\pi} (\Ra_\ep -h_p) \dA 
+ h_p \int_{V_p^\ep} \beta \dV\rmb\\
&=\sum_{p\in S_\ep} \lmb \int_{Q_p} \beta (\Ra_\ep -h_p) \dV +   \int_{\pt Q_p} \Tilde{\pi} (\Ra_\ep -\Tilde{h}_p) \dA \right.\\
&\qquad \left.
+(h_p-\Tilde{h}_p)\int_{\pt Q_p} \Tilde{\pi} \dA+ h_p \int_{V_p^\ep} \beta \dV\rmb\\
&=:\sum_{p\in S_\ep} \lmb J_1(p) + J_2(p) + J_3(p) + J_4(p)\rmb.
\end{split}
\ee
We choose
\be
\label{eq.H1estiamteRa.choosehp}
h_p : = \frac{\int_{V_p^\ep }\beta \Ra_\ep \dV}{ \int_{V_p^\ep} \beta \dV} \textup{ and }\Tilde{h}_p:=\frac{\int_{B_{2\ep}(p)} \Ra_\ep \dV}{ |B_{2\ep}(p)|}.
\ee
By the choice of $h_p$, we immediately obtain by \eqref{eq.normalizationofRa} that 
\be
\label{eq.H1estimateRa.J4}
\sum_{p\in S_\ep } J_4(p) = \int_{M} \beta \Ra_\ep \dV = 0.
\ee
By Cauchy-Schwarz inequality, we obtain
\be
\label{eq.H1estimateRa.J1}
   |J_1(p)|\le \frac{\norm{\beta}_{L^2(V_p^\ep)}^2}{\int_{V_p^\ep}\beta \dV} \lb \iint_{B_{2\ep}(p)\times B_{2\ep}(p)} \lb \Ra_\ep(x) - \Ra_\ep(y) \rb^2 \dV(x)\dV(y) \rb^{1/2}
\ee
In a geodesic coordinate $z_1,z_2\in \R^d \cong T_pM$, we have 
\[
\begin{split}
   |J_1(p)|&\le C\lb M,d,g,S\rb  ~ \lb \iint_{B_{2\ep}(0)\times B_{2\ep}(0)} \lb \Ra_\ep(z_1) - \Ra_\ep(z_2) \rb^2 \dd z_1\dd z_2 \rb^{1/2}\\
 &=C\lb M,d,g,S\rb  ~|B_{2\ep}(0)|^{1/2} ~ \lb \int_{B_{2\ep}(0)} \lb \Ra_\ep (z_2) - \frac{\int_{B_{2\ep}(0)} \Ra _ \ep \dd z_1}{|B_{2\ep}(0)|}\rb^2 \dd z_2 \rb^{1/2}
\end{split}
\]
By the standard Poincar\'{e} inequality for $B_{2\ep}(0)\subset \R^d$, we have
\be
\label{eq.H1estimateRa.J1.equation2}
\begin{split}
    |J_1(p)| &\le C\lb M,d,g,S\rb ~ \ep ~ |B_{2\ep}(0)|^{1/2} ~ \lb \int_{B_{2\ep}(0)} |\gd_z \Ra_\ep|^2 \dd z \rb^{1/2}\\
&\le C\lb M,d,g,S\rb  ~ \ep ~ |B_{2\ep}(p)|^{1/2} ~ \norm{\gd \Ra_\ep}_{L^2(B_{2\ep}(p))}.
\end{split}
\ee
For the second term $J_2$, we also write in a geodesic coordinate $z$ and obtain by \eqref{eq.boundonpitilde}
\[
\begin{split}
  |J_2(p)| &\le  C \ep A(\pt Q_p)^{1/2} \norm{\Ra_\ep - \Tilde{h}_p}_{L^2(\pt Q_p)}\\
  &\le  C \ep A(\pt Q_p)^{1/2} \lb \int_{\pt B_{R_\ep}(0)} \lb \Ra_\ep(z) - \Tilde{h}_p \rb^2 \dd \theta(z) \rb^{1/2},
\end{split}
\]
where $C=C(M,d,g,S)$ and $\theta$ is the standard Lebesgue measure on the round sphere $\pt B_{R_\ep}(0) \subset \R^d$. By the Poincar\'{e} inequality and trace theorem on $ B_{2\ep}(0)$ (when restricted to $\pt B_{R_\ep}(0)$), we have
\be
\label{eq.H1estimateRa.J2}
\begin{split}
    |J_2(p)| & \le C \ep^{1+1/2} A(\pt Q_p)^{1/2} \lmb \norm{\gd_z \Ra_\ep}_{L^2(B_{2\ep}(0))} + \ep^{-1}\norm{\Ra_\ep - \Tilde{h}_p}_{L^2(B_{2\ep}(0))} \rmb\\
    &\le C \ep^{1+1/2} A(\pt Q_p)^{1/2} \lmb \norm{\gd \Ra_\ep}_{L^2(B_{2\ep}(p))} + \ep^{-1}\norm{\Ra_\ep - \Tilde{h}_p}_{L^2(B_{2\ep}(p))} \rmb\\
    &\le C \ep^{1+1/2} A(\pt Q_p)^{1/2} \lmb \norm{\gd \Ra_\ep}_{L^2(B_{2\ep}(p))}\right.\\
   &\qquad \left. + \ep^{-1}|B_{2\ep}(p)|^{-1/2}\lb \iint_{B_{2\ep}(p)\times B_{2\ep}(p)} \lb \Ra_\ep(x) - \Ra_\ep(y) \rb^2 \dV(x)\dV(y) \rb^{1/2} \rmb\\
    &\le C\ep^{1+1/2} A(\pt Q_p)^{1/2}\norm{\gd \Ra_\ep}_{L^2(B_{2\ep}(p))}
\end{split}
\ee
where $C=C(M,d,g,S)>0$ and the last inequality follows from the same argument in $J_1(p)$ above.
On the other hand, we observe that 
\be
\label{eq.difference.hp.htildp}
\begin{split}
    |h_p-\Tilde{h}_p|&=\int_{B_{2\ep}(p)} \Ra_{\ep} \lmb\frac{\beta}{\int_{V_p^\ep}\beta\dV} - \frac{1}{|B_{2\ep}(p)|}\rmb \dV\\
    &=\int_{B_{2\ep}(p)} \lb \Ra_{\ep} - \Tilde{h}_p \rb \lb\frac{\beta}{\int_{V_p^\ep}\beta\dV} - \frac{1}{|B_{2\ep}(p)|}\rb \dV\\
    &\le C(S) |V_p^\ep|^{-1/2} \norm{\Ra_\ep - \Tilde{h}_p}_{L^2(B_{2\ep}(p))}\\
    &\le C(M,d,g,S)~ \ep~ |B_{2\ep}(p)|^{-1/2} \norm{\gd \Ra_\ep}_{L^2(B_{2\ep}(p))}
\end{split}
\ee
similar to the arguments in \eqref{eq.H1estimateRa.J1.equation2}. By using \eqref{eq.boundonpitilde} and \eqref{eq.difference.hp.htildp}, we obtain that
\be
\label{eq.H1estimateRa.J3}
|J_3(p)| \le C(M,d,g,S) \ \ep^2 \ |B_{2\ep}(p)|^{-1/2} \ A(\pt Q_p) \norm{\gd \Ra_\ep}_{L^2(B_{2\ep}(p))}.
\ee
The proof of this lemma is complete by combining \eqref{eq.H1estimateRa.J4}, \eqref{eq.H1estimateRa.J1}, \eqref{eq.H1estimateRa.J2}, \eqref{eq.H1estimateRa.J3} and the fact that $B_{2\ep}(p)$ intersects with at most $N=N(M,d,g)<\infty$ many $B_{2\ep}(p')$ for $p,p'\in S_\ep$ due to the $\ep$-separatedness.

\end{proof}

\subsection{An \texorpdfstring{$L^\infty$}{L}-estimate for \texorpdfstring{$\Ra_\ep$}{R}}

Now we prove a sharp estimate for the global oscillation of $\Ra_\ep$ as defined in \eqref{eq.equationforRaep}, which also satisfies the weak form equation \eqref{eq.weakformofRa}.

\begin{lemma}
    \label{l.globalosc.Raep}
    Let $\Ra_\ep$ be as defined in \eqref{eq.weakformofRa} and \eqref{eq.normalizationofRa}, then for $\ep<\ep_0(M,d,g)$ there is a constant $C(M,d,g,S)>0$ such that 
    \be
\label{eq.globalosc.Raep}
\norm{\Ra_\ep}_{L^\infty(M)} \le C \ep.
    \ee
\end{lemma}

\begin{remark}
    \label{r.thematterofcentroidVoronoiTessellation}
    This lemma is sharp and we will show the sharpness by providing an example in Section \ref{ss.counterex}.
\end{remark}

\begin{proof}
    Recall that in the sense of distributions the remainder function $\Ra_\ep$ satisfies the following equation
\be
\label{eq.equationforRa.distri}
\Delta_g \Ra_\ep = \beta 1_{\Omega^*} \dV - \Tilde{\pi} \restr{\dA}{\pt\Omega^*} =: \drho.
\ee
Let $G_x(y)$ be the Green function on $(M,g)$ satisfying 
\be
\label{eq.Greensfunction}
\Delta_g G_x = \delta_x - \frac{\beta}{\int_M \beta \dV}.
\ee
Note that $G$ differs from the standard Green's function by a smooth function and it is uniquely determined up to an additive constant. By applying $G_x$ to \eqref{eq.equationforRa.distri} and the assumption \eqref{eq.normalizationofRa}, we obtain for every $x\in M$
\be
\label{eq.aGreenrepofRa}
\Ra_\ep(x) = \int_{M} G_x(y) \drho(y).
\ee
Note that the integral on the right hand side is well-defined for all $y$ because the singularity of $G_x(y)$ at $y=x$ is integrable both on $M$ and $\pt \Omega^*$.
To make an $L^\infty$ estimate, we decompose 
\be
\label{eq.decompose.drho}
\rho = \sum_{q\in S_\ep} \rho_q :=\sum_{p\in S_\ep} \restr{\rho}{V_q^\ep}.
\ee
Note that by \eqref{eq.boundonpitilde}, for $\ep< \ep_0(M,d,g)$
\be\label{eq.totalvariation.rhoq}
\max_{q\in S_\ep}\norm{\rho_q}_{TV} \le C \ep^d.
\ee
Now we can write for $x\in V_p^\ep$ (we may without loss assume that $x$ is in the interior of $V_p^\ep$) for some $p\in S_\ep$
\be
\label{eq.decomposeerror}
\begin{split}
    \Ra_\ep(x) & = \int_{M} G_x \drho\\
    &=\sum_{q\in S_\ep} \int_{V_q^\ep} G_x \drho_q\\
    &=\underbrace{\sum_{\dist(q,p)\le 8\ep}\int_{V_q^\ep}G_x \drho_q }_{A_1} ~+~\underbrace{\sum_{\dist(q,p)>8\ep} \int_{V_q^\ep} G_x \drho_q}_{A_2}. 
\end{split}
\ee
We first claim that 
\be
\label{eq.claimA1bound}
|A_1| \le C \ep
\ee
for some $C=C(M,d,g,S)>0$. To prove this we observe that if we denote 
$$
T_{x}:= \bigcup_{\dist(q,p)\le 2^3\ep} \pt Q_p \textup{ and }T_{x,r}:=\{y\in T_x\;:\;\dist(y,x)\le r\}
$$
then we have
\[
\begin{split}
    |A_1| & \le C \sum_{\dist(q,p)\le 2^3 \ep} \int_{V_q^\ep} |G_x| \dV + \int_{T_x} |G_x|\dA\\
    &=: A_{11} + A_{12},
\end{split}
\]
where $A_{11}$ satisfies the bound
\[
|A_{11}| \le C \bca \displaystyle \int_{C\ep^{d/(d-1)}}^{8\ep} r \dd r & \textup{ when }d\ge 3\\[6pt]
\displaystyle \int_{C\ep^2}^{8\ep} r \log \lb\frac{1}{r}\rb \dd r & \textup{ when }d=2
\eca\le C \ep^2 |\log \ep|\le C \ep,
\]
where we have recalled the bound of the Green functions that for $y\in M$ (See \cite{aubin_nonlinear_1998}*{Chapter 4})
\be\label{eq.bound.green}
G_x(y)\le \bca
\displaystyle C\frac{1}{\dist(y,x)^{d-2}}  & \textup{ when }d\ge 3\\[12pt]
\displaystyle  C\log\lb\frac{1}{\dist(y,x)}\rb + C_0 & \textup{ when }d=2.
\eca
\ee
To bound $A_{12}$, we apply \eqref{eq.boundonpitilde} and obtain
\be
\label{eq.boundA1}
\begin{split}
 \lw A_{12}\rw &\le C\ep \int_{T_{x,2^{-5}\ep}} |G_x| \dA + C \ep\int_{T_{x}\setminus T_{x,2^{-5}\ep}} |G_x| \dA\\
&=: A_{121} + A_{122},
\end{split}
\ee
where it is not difficult to derive the bound for $A_{122}$ by applying \eqref{eq.bound.green}
\[
A_{122} \le C \ep^{3-d}|\log \ep| A(T_x) \le C \ep^2 |\log \ep|.
\]
By applying \eqref{eq.bound.green} again, we have the following inequality
\[
\begin{split}
    A_{121} & = C\ep \sum_{k=0}^\infty \int_{T_{x,2^{-5-k}\ep} \setminus T_{x,2^{-6-k}\ep} } |G_x| \dA\\
    &\le C \ep \sum_{k=0}^\infty \int_{T_{x,2^{-5-k}\ep} \setminus T_{x,2^{-6-k}\ep} } |G_x| \dA\\
    &\le C \ep \sum_{k=0}^\infty \lb\frac{2^k}{\ep}\rb^{d-2} \lw\log \lb \frac{2^k}{\ep} \rb\rw  A(T_{x,2^{-5-k}\ep}).
\end{split}
\]
Note that $B_{2^{-5-k}\ep}(x)$ intersects at most one $\pt Q_p$ with $p\in S_\ep$ when $ k\ge 0$. We fix this $p$ whenever it exists and then if we denote $x^*\in \pt Q_p$ a point satisfying $\dist(x,\pt Q_p)= \dist(x,x^*)$, then the set $T_{x,2^{-5-k}\ep}$, by triangle inequality, is contained in the spherical cap of $\pt Q_p$ centered at $x^*$ having geodesic radius (with respect to the manifold $M$ instead of $\pt Q_p$) $2^{-4-k}\ep<R_\ep = 2^{-3}\ep$ as defined in \eqref{eq.defineOmegastar}. This implies that 
\[
A(T_{x,2^{-5-k}\ep}) \le C \lb 2^{-k}\ep \rb^{d-1},
\]
and hence
\[
A_{121} \le C \ep \sum_{k=0}^\infty \lb\frac{2^k}{\ep}\rb^{d-2} \lw\log \lb \frac{2^k}{\ep} \rb\rw  \lb 2^{-k}\ep \rb^{d-1} \le C \ep^2 |\log \ep| \sum_{k=0}^\infty k 2^{-k} \le C \ep^2 |\log \ep|
\]
as desired.
To estimate $A_2$, we fix an arbitrary $q\in S_\ep$ such that $\dist(q,p)>8\ep$, and then we have by the balance of mass \eqref{eq.normalizationofRa}
\be
\label{eq.estimateA2}
\begin{split}
   \lw\int_{V_q^\ep} G_x \drho_q\rw & =  \lw\int_{V_q^\ep} G_x  - G_x(q) \drho_q\rw\\
 &\le C \ep \|\rho_q\|_{TV} \sup_{y\in V_q^\ep} |\gd G_x|(y)\\
 &\le C \sup_{y\in V_q^\ep}\frac{\ep^{d+1}}{\dist(y,x)^{d-1}},
\end{split}
\ee
where we have used the gradient estimate for the Green's function and the constant $C=C(M,d,g,S)>0$ may be different from line to line. Because $x\in V_p^\ep$ and $\dist(p,q) > 8\ep$, we observe that for all $y\in V_q^\ep$
\[
\dist(y,x) \ge \dist(p,q) - 2 \ep > \frac{3}{4} \dist(p,q).
\]
This shows that 
\be
\label{eq.estimateA2.2}
   \lw\int_{V_q^\ep} G_x \drho_q\rw \le  C \frac{\ep^{d+1}}{\dist(p,q)^{d-1}}.
\ee
We divide the following set dyadically into
\[
S_\ep \setminus \{q\;;\; \dist(q,p)\le 2^3\ep\} = \bigcup_{k = 0}^N \left\{q\in S_\ep \;:\; \dist(q,p)\in (2^{k+3}\ep, 2^{k+4}\ep]\right\}=:\bigcup_{k = 0}^N \mathcal{O}_k,
\]
where $2^{N + 4}\ep \approx \diam(M)$ and therefore $N = O(|\log\ep|)$. Note that each $\Oa_k \subset \{x\in M\;:\; \dist_g(x,p)\le 2^{k+4}\ep\}$, but the latter can only have at most $C 2^{kd}$ points in $S_\ep$ due to the $\ep$-separatedness of $S_\ep$. This, combining \eqref{eq.estimateA2} and \eqref{eq.estimateA2.2}, implies that 
\[
\begin{split}
    |A_2| & \le C\sum_{\dist(q,p)> 2^3\ep} \frac{\ep^{d+1}}{\dist(p,q)^{d-1}}\\
    &\le C \sum_{k=0}^N \# \Oa_k \cdot \frac{\ep^{d+1}}{(2^k\ep)^{d-1}} \\
    &\le C \sum_{k=0}^N 2^{kd} \cdot \frac{\ep^{d+1}}{(2^k\ep)^{d-1}} \\
    &\le C \ep.
\end{split}
\]
This completes the proof.

\end{proof}

\subsection{An example showing sharpness of Lemma \texorpdfstring{\ref{l.globalosc.Raep}}{sd}} \label{ss.counterex}
In this subsection, we discuss an example in 2D torus that illustrates the sharpness of Lemma \ref{l.globalosc.Raep}.

We denote $\ep = \frac{1}{2D}$ for each $D\in \N_+$. Note that given a point set of the form $\ep \Z^2 \cap (-1/2,1/2)^2$, the rescaled point set $(1-\eta\ep) \ep \Z^2 \cap (-1/2,1/2)^2$ forms a maximally $(1-\eta\ep)\ep$-separated set in the torus as long as $\eta < \frac{\sqrt{2}-1}{1+(\sqrt{2}-1)\ep}$.

We denote 
\be
\label{eq.Seps}
S_\ep:=\ep \Z^2 \cap (-1/2,1/2)^2 \textup{ and } L_\ep=L_\ep^\eta : = (1-\eta\ep) \ep \Z^2 \cap (-1/2,1/2)^2,
\ee
where $\ep = 1/(2D)$ for positive integers $D$.

\begin{figure}[htbp]
    \centering
    \begin{minipage}{0.47\linewidth}
        \centering
        \includegraphics[width=0.7\linewidth]{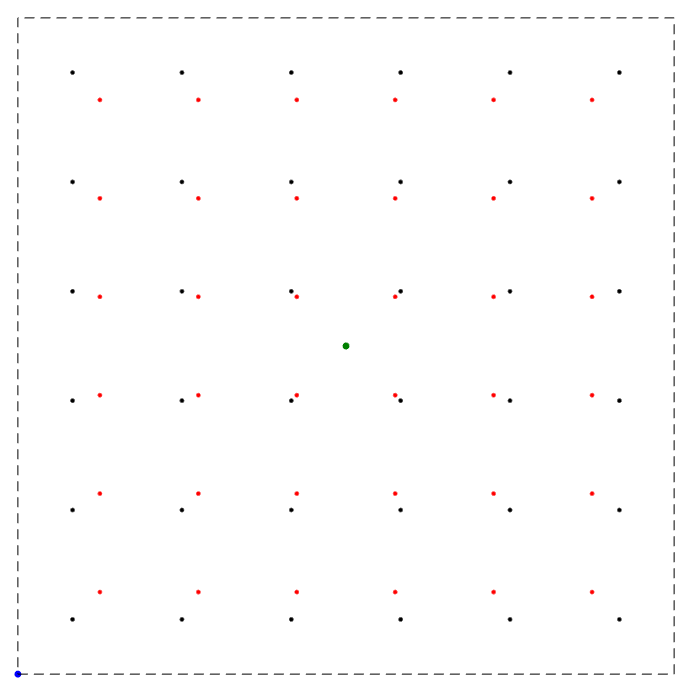}
    \end{minipage}
    \begin{minipage}{0.47\linewidth}
        \centering
        \includegraphics[width=0.7\linewidth]{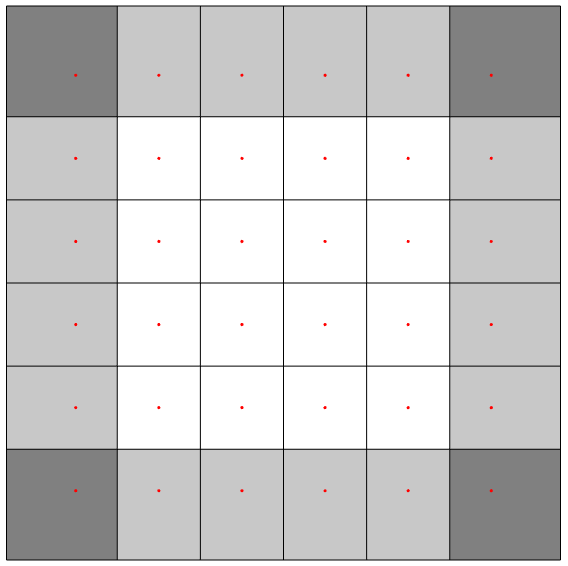}
       
    \end{minipage}
    \caption{Left: A picture of points \texorpdfstring{$\frac{1}{6}\Z^2 \cap (-1/2,1/2)^2$}{sd} (the black ones) and its scaling (the red ones) at the origin with factor \texorpdfstring{$1-\eta\ep = 0.95$}{as}. Note that there are \texorpdfstring{$6*6=36$}{sd} points, with the black ones being maximally $\frac{1}{6}$-separated and the red ones being maximally \texorpdfstring{$\frac{0.95}{6}$}{sd}-separated. The origin $(0,0)$ is colored by green and the point $(-0.5,-0.5)$ is colored blue. Right: the square $(-1/2,1/2)^2$ is separated into the Voronoi cells generated by the rescaled points. Darker cells have larger area. Note that there are three classes of cells of the same area.
    } \label{fig:twofigures}
\end{figure}

Let us now construct the cell functions $\Phi_\ep$. For simplicity, we start with the standard Green's function
\be
\label{eq.Greensfunction.2}
\bca
\displaystyle\Delta G(x,p) =\dV - \delta_p   & \textup{ in }\T^2\\[5pt]
\displaystyle\int_{\T^2} G(x,p) \dV(x) =0, &
\eca
\ee
where $x,p\in \T^2.$

Recall that by standard theory, the Green function $G(x,p)$ can be decomposed as
\be
\label{eq.decomposeGreen}
G(x,p) = g_p(x) + h_p(x),
\ee
where $h_p$ is a globally smooth function and $g_p$ is supported in the ball $B_{1/2}(p)$, and in geodesic polar coordinate $(r,\theta)$ we have the formula
\be
g_p(r,\theta) = g_p(r) = \frac{\phi(r)}{2\pi} \log \lb \frac{1}{r}    \rb,
\ee
where $0\le \phi\le 1$ is a smooth cut-off function on $\overline{\R}_+$ that is 1 near zero and 0 for $r>1/4$. Because $\T^2$ is flat, we know that $\Delta g_p = -\delta_p$ in $B_{1/2}(p)$ and therefore $\Delta h_p = 1$ in that ball.
We denote $\dV$ the Lebesgue measure and $\dA$ the 1D Hausdorff measure on $\T^2$. Let us now recall that the cell function $\Phi_\ep$ satisfies
\be
\label{eq.cellfunction.2}
\Delta \Phi_\ep =\dV-  \restr{\dA}{\pt\Omega},
\ee
where $\Omega \subset \T^2$ is a subdomain obtained by removing disks of the form $B_{r_{\ep}}(p)$ for $p\in S_\ep$ and the radii $r_{\ep} = \frac{\ep^2}{2\pi}$. 
We can solve \eqref{eq.cellfunction.2} by explicitly modifying the Green's function. Specifically, we write $G(x,p)=g_p + h_p$ as in \eqref{eq.Greensfunction.2} and then we write
\be\label{eq.definePsip}
\Psi_p(x) : = \Tilde{g}_p + h_p + c_\ep
\ee
where $c_\ep=\int_{B_{r_\ep}(p)}G(x,p) \dV(x) =O(\ep^4|\log \ep|)$ is a normalizing constant so that $\int_{\T^2}\Psi_p \dV =0$ and
\be
\label{eq.definegtilde}
\Tilde{g}_p(r):= \bca
g_p(r_\ep) & \textup{ for }r\le r_\ep\\
g_p(r)   & \textup{ elsewhere.}
\eca
\ee
Each $\Psi_p$ satisfies  
\be
\label{eq.equationPsiep}
\Delta \Psi_p = \dV- \frac{1}{2\pi r_\ep} \restr{\dA}{\pt B_{r_\ep}(p)}
\ee
In particular, by testing this equation by~$|x-p|^2$  above (we do not distinguish functions on $\T^2$ and their periodic extension on $\R^2$) and integrate over $p + (-1/2,1/2)^2$, we obtain
\be
\label{eq.scaleofPsip.atboundary}
\int_{\pt (p+(-1/2,1/2)^2)}\Psi_p \dA = -\int_{(-1/2,1/2)^2} |x|^2 \dV(x) + O(\ep^4) =-\frac{1}{6} + O(\ep^4).
\ee

\vspace{0.5cm}
\noindent Note that it is not difficult to check the following identity
\be
\label{eq.anewformulaforPhiep}
\Phi_\ep(x) = \sum_{p\in S_\ep} 2\pi r_\ep \Psi_p(x).
\ee

Let us now define the corresponding cell function $\Tilde{\Phi}_\ep$ of the perturbed point set $L_\ep$. Note that $L_\ep$ is less regular compared to $S_\ep$ and there are three different types of Voronoi cells as shown in Figure \ref{fig:twofigures}. We classify the points $q\in L_\ep$ into three types $i=1,2,3$ by the order of the area in the corresponding Voronoi cell, the points with largest Voronoi cell area are collected in $L_{1,\ep}$ and the least in $L_{3,\ep}$. Correspondingly we define $2\pi r_{i,\ep}$ to be the area of the Voronoi cells of points in $L_{i,\ep}$. We define, similarly to \eqref{eq.anewformulaforPhiep}
\be
\label{eq.anewformulaforPhiep.perturbed}
\Tilde{\Phi}_\ep(x) =\sum_{i=1}^3 \sum_{q\in L_{i,\ep}} 2\pi r_{i,\ep} \Tilde \Psi_q(x),
\ee
where for $q\in L_{i,\ep}$
\be\label{eq.tildePsi}
\Delta \tilde \Psi_{q} = \dV- \frac{1}{2\pi r_{i,\ep}} \restr{\dA}{\pt B_{r_{i,\ep}}(q)}.
\ee
This makes sure that $\Tilde{\Phi}_\ep$ satisfies the equation similar to \eqref{eq.cellfunction.2}
\be
\label{eq.cellfunction.Phiep}
\Delta \tilde \Phi_\ep =\dV -  \restr{\dA}{\pt\Tilde{\Omega}},
\ee
where $\Tilde{\Omega}$ is defined similarly to $\Omega$ by removing disks of the form $B_{r_{i,\ep}}(q)$ for $q\in L_{i,\ep}$ and $i=1,2,3$.

\begin{proposition}
    \label{prop.comparetwocellfunction}
    For  $\ep = 1/(2D)$ with $D$ being any positive integer, there is
    \be
\label{eq.comparetwocellfunction}
\osc_{\T^2}\Phi_\ep \le C \ep^2 |\log \ep|.
    \ee
However, for the perturbed point set $L_\ep$, 
    \be
\label{eq.comparetwocellfunction.2}
|\Tilde{\Phi}_\ep(x_{\textup{green}})|= \lw\Tilde{\Phi}_\ep(x_{\textup{green}})-\avg{\Tilde{\Phi}_\ep}_{\T^2}\rw \ge C \ep
    \ee
    for a constant $C>0$ independent of $\ep$. Here $x_{\textup{green}}$ is the origin as illustrated in Figure \ref{fig:twofigures}.
\end{proposition}

\begin{remark}
    \label{r.centroidmattersforLinfinitybound.cellfunction}
    The reason for this dramatic change of the global oscillation from $O(\ep^2 |\log \ep|)$ of $\Phi_\ep$ to $O(\ep)$ of $\Tilde{\Phi}_\ep$ is that the points $q\in L_\ep$ are not the mass centers of the corresponding Voronoi cells.
\end{remark}

\begin{proof}

We will prove \eqref{eq.comparetwocellfunction} in Proposition \ref{prop.comparetwocellfunction} in a more general setting, so we only worry about \eqref{eq.comparetwocellfunction.2}.
To prove \eqref{eq.comparetwocellfunction.2}, we just take derivative in $\eta$ of $\Tilde{\Phi}_\ep(x_{\textup{green}})$ 
\[
\begin{split}
\pt_\eta \Tilde{\Phi}_\ep(x_{\textup{green}})&= 2\pi\sum_{i=1}^3 \sum_{q\in L_{i,\ep}^\eta} \lb \pt_\eta r_{i,\ep} \tilde\Psi_q +  r_{i,\ep} \pt_\eta\tilde\Psi_q \rb\\
   &\overset{\eta=0}{=}~2\pi\sum_{i=1}^3 \sum_{q\in S_{i,\ep}} \pt_\eta r_{i,\ep} \Psi_{q} , 
\end{split}
\]
where $S_{i,\ep}:=\lim_{\eta\rta0^+} L_{i,\ep}$ and the last equality holds due to the symmetry of $S_\ep$.
Note that 
\[
r_{1,\ep} = \frac{1}{2\pi}(1+\eta/2 - \eta \ep)^2 \ep^2, \ r_{2,\ep}=\frac{1}{2\pi}(1-\eta\ep)(1+\eta/2-\eta\ep)\ep^2 \textup{ and }r_{3,\ep}=\frac{1}{2\pi}(1-\eta\ep)^2\ep^2.
\]
This shows that the derivatives at $\eta = 0$ takes the following form
\[
\pt_\eta r_{1,\ep} = \frac{1}{2\pi}\ep^2 - \frac{1}{\pi}\ep^3, \ \pt_\eta r_{2,\ep}=\frac{1}{4\pi}\ep^2 - \frac{1}{\pi}\ep^3 \textup{ and }\pt_\eta r_{3,\ep} = -\frac{1}{\pi}\ep^3.
\]
This shows that
\be
\label{eq.derivativeinetaofPhi}
\restr{\pt_\eta}{\eta=0} \Tilde{\Phi}_\ep(x_{\textup{green}}) = \frac{\ep^2}{2}\sum_{p\in S_{2,\ep}} \Psi_p(x_{\textup{green}}) +\underbrace{\ep^2\sum_{p\in S_{3,\ep}}\Psi_p(x_{\textup{green}})}_{O(\ep^2)}  - \ \underbrace{4\pi\ep \Phi_\ep(x_{\textup{green}})}_{O(\ep^3|\log \ep|)}.
\ee
 Here the second term in the last line is of order $O(\ep^2)$ is because $S_{3,\ep}$ has only 4 elements, and the third term has order $O(\ep^3|\log \ep|)$ is due to \eqref{eq.comparetwocellfunction}.
Now, for the first term in \eqref{eq.derivativeinetaofPhi}, we have
\[
\begin{split}
    \frac{\ep^2}{2}\sum_{p\in S_{2,\ep}} \Psi_p(x_{\textup{green}})& = \frac{\ep}{2}\sum_{p\in S_{2,\ep}} \ep\Psi_p(x_{\textup{green}})\\
&\approx \frac{\ep}{2} \int_{\pt (x_{\textup{green}}+(-1/2,1/2)^2)} \Psi_{x_{\textup{green}}} \dA\\
&=-\frac{\ep}{12},
\end{split}
\]
where the last line is according to \eqref{eq.scaleofPsip.atboundary}. This finishes the proof.

\end{proof}

\subsection{A sharper \texorpdfstring{$L^2(M)$}{L2}-estimate of \texorpdfstring{$\Ra_\ep$}{R}}

It is not difficult to show that by Lemma \ref{l.globalosc.Raep}
$$\norm{\Ra_\ep}_{L^2(M)}\le C\norm{\Ra_\ep}_{L^\infty(M)}\le C\ep.$$ In this subsection, we improve this result by utilizing the equation \eqref{eq.equationforRa.distri}.

\begin{lemma}
    \label{l.sharperL2estimateforRaep}
    Under the assumptions of Lemma \ref{l.globalosc.Raep}, then there is a constant $C(M,d,g,S)>0$ such that 
    \be
\label{eq.sharperL2estimateforRaep}
\norm{\Ra_\ep}_{L^2(M)} \le C\ep^2.
    \ee
    
\end{lemma}

To prove the above lemma we need the following auxiliary function.
\be
\label{eq.definecellXi}
\xi_\ep(x):=\min_{p\in S_\ep} \frac{1}{2} \dist^2(x,p).
\ee

\begin{lemma}
    \label{l.atechnicallemma.L2estimateRa}
The auxiliary function $\xi=\xi_\ep$ as defined in \eqref{eq.definecellXi} is a globally Lipschitz function such that
\[
\xi(x) = \frac{1}{2}\dist^2(x,p)
\]
for every $x\in V_p^\ep$ by the definition of the Voronoi cells. Moreover, we have the following estimates
\be
\label{eq.property.Xi}
\norm{\xi}_{L^\infty(M)} \le C \ep^2 \textup{ and the Lipschitz constant }\textup{Lip}_M(\xi) \le C \ep,
\ee
and for every $x\in V_p^\ep$ (not including the boundary), there is 
\be
\label{eq.property.Xi.hessian}
|\gd^2\xi(x) - I_x| \le C(M,d,g) \ep^2,
\ee
where $I_x: T_xM \rta T_xM$ is the identity map.
\end{lemma}

\begin{proof}
    The proof of this lemma can be immediately obtained by the computations done in \cite{pennec_hessian_2017}*{Section 2}.
\end{proof}

\begin{proof}[Proof of Lemma \ref{l.sharperL2estimateforRaep}]
Recall the equation \eqref{eq.equationforRa.distri}
\[
\Delta_g \Ra_\ep = \beta 1_{\Omega^*} \dV - \Tilde{\pi} \restr{\dA}{\pt\Omega^*} = : \drho
\]
and the decomposition \eqref{eq.decompose.drho}
\[
\rho = \sum_{q\in S_\ep} \rho_q, \textup{ where }\rho_q := \restr{\rho}{V_q^\ep}.
\]
For every $\eta \in C^2(M)$, it suffices to show the functional $l(\eta):=\int_M \eta \drho$ satisfies
\be
\label{eq.functionalclaim.1}
|l(\eta)| \le C(M,d,g,S) \ep^2 \norm{\eta}_{H^2(M)}.
\ee
Indeed, suppose we have \eqref{eq.functionalclaim.1}, then for every $f\in C^\infty(M)$  we solve for 
\be
\label{eq.anauxiliaryeta}
\bca
\Delta_g \eta_f = f - \frac{\int_M f \dV}{\int_M \beta \dV} \beta & \textup{ on }M \\
\int_M \eta_f \dV = 0.
\eca
\ee
By standard theory, there is a constant $C(M,d,g,S)>0$ such that 
\[
\norm{\eta_f}_{H^2(M)} \le C \norm{f}_{L^2}.
\]
This implies that 
\[
\lw\int_{M}\Ra_\ep f \dV\rw = |l(\eta_f)| \le C \ep^2 \norm{\eta_f}_{H^2(M)} \le C \ep^2 \norm{f}_{L^2(M)},
\]
which finishes the proof by a standard duality argument.
Let us now prove the claim \eqref{eq.functionalclaim.1}.  We will be using multiple times the Poincar\'{e} inequality on Voronoi cells $V_p^\ep$, in which the bounding coefficients have to be uniform over different cells. We refer to \cite{lo_homogenisation_2025}*{Proposition 3.10} for more details.
For every $p\in S_\ep$ we denote $$\Tilde{\rho}_p:=\avg{\beta}_{V_p^\ep\setminus B_{R_\ep}(p)} \restr{\dV}{V_p^\ep\setminus B_{R_\ep}(p)} - \avg{\Tilde{\pi}}_{\pt B_{R_\ep}(p)}\restr{\dA}{\pt B_{R_\ep}(p)}.$$ We then have for all constants $c_1,c_2$
\be
\label{eq.computeleta}
\begin{split}
   \int_M \eta \drho_p & =\int_{V_p^\ep\setminus B_{R_\ep}(p)} \eta \lb \beta - \avg{\beta}_{V_p^\ep\setminus B_{R_\ep}(p)}\rb\dV -\int_{\pt B_{R_\ep}(p)} \eta \lb \Tilde{\pi} - \avg{\Tilde{\pi}}_{\pt B_{R_\ep}(p)}\rb\dA   + \int_{M}  \eta   \dtrho_p\\
   &=\int_{V_p^\ep\setminus B_{R_\ep}(p)} (\eta-c_1) \lb \beta  - \avg{\beta}_{V_p^\ep\setminus B_{R_\ep}(p)}\rb\dV -\int_{\pt B_{R_\ep}(p)} (\eta-c_1) \lb \Tilde{\pi} - \avg{\Tilde{\pi}}_{\pt B_{R_\ep}(p)}\rb\dA   \\
   &\qquad+\int_{M}  \eta - c_2  \dtrho_p \\
 &=  \int_{V_p^\ep\setminus B_{R_\ep}(p)} (\eta-c_1) \lb \beta  - \avg{\beta}_{V_p^\ep\setminus B_{R_\ep}(p)}\rb\dV-\int_{\pt B_{R_\ep}(p)} (\eta-c_1) \lb \Tilde{\pi} - \avg{\Tilde{\pi}}_{\pt B_{R_\ep}(p)}\rb\dA   \\
 &\qquad+\int_{M}  \eta - c_2 - \gd \eta \cdot \gd \xi \dtrho_p + \int_{M}  \gd \eta \cdot \gd \xi \dtrho_p\\
 &=: J_1 - \Tilde{J}_1 + J_2 + J_3.
\end{split}
\ee
By \eqref{eq.uniformbound.weights}, we know 
\be
\label{eq.boundJ1.H2}
\begin{split}
    |J_1| & \le C(S) \ep \norm{\eta - c_1}_{L^2(V_p^\ep)} |V_p^\ep|^{1/2}\\
    &\le C(M,d,g,S) \ep^2 \norm{\gd \eta}_{L^2(V_p^\ep)} |V_p^\ep|^{1/2},
\end{split}
\ee
where in the last line we choose $c_1=\avg{\eta}_{V_p^\ep}$ so that we can apply the Poincar\'{e} inequality. By Corollary \ref{cor.boundonpitilde} we know that $\lw \Tilde{\pi} - \avg{\Tilde{\pi}}_{\pt B_{R_\ep}(p)}\rw \le C \ep^2$ and hence by applying the trace theorem and Poincar\'{e} inequality together we have
\be
\label{eq.boundtildeJ1.H2}
\begin{split}
    |\Tilde{J}_1| & \le C\ep^2 \norm{\eta - c_1}_{L^2(\pt B_{R_\ep}(p))} A(\pt B_{R_\ep}(p))^{1/2}\\
    &\le C(M,d,g,S) \ep^2 \norm{\gd \eta}_{L^2(V_p^\ep)} |V_p^\ep|^{1/2}.
\end{split}
\ee
For $J_2$ we denote $\hat{\eta} := \eta - c_2 - \gd \eta \cdot \gd \xi$ and decompose
\[
J_2 = \avg{\beta}_{V_p^\ep\setminus B_{R_{\ep}}(p)}\int_{V_p^\ep\setminus B_{R_{\ep}}(p)} \hat{\eta} \dV  - \avg{\Tilde{\pi}}_{\pt B_{R_\ep}(p)}\int_{\pt B_{R_\ep}(p)}\hat{\eta} \dA =: J_{21} - J_{22}.
\]
Note that by choosing  $c_2 $ so that $\avg{\hat{\eta}}_{V_p^\ep} = 0$ we have
\[
\begin{split}
    |J_{21}| & \le C(S) \norm{\hat{\eta}}_{L^2(V_p^\ep)} |V_p^\ep|^{1/2}\\
    &\le C(M,d,g,S) \ep \norm{\gd \hat{\eta}}_{L^2(V_p^\ep)} |V_p^\ep|^{1/2}\\
    &\le  C \ep \norm{\gd^2 \eta \gd \xi}_{L^2(V_p^\ep)} |V_p^\ep|^{1/2} +  C \ep \norm{\lb \gd^2 \xi - I \rb\gd \eta }_{L^2(V_p^\ep)} |V_p^\ep|^{1/2}\\
    &\le C \ep^2|V_p^\ep|^{1/2} \lmb\norm{\gd^2 \eta}_{L^2(V_p^\ep)}  +  \norm{\gd \eta}_{L^2(V_p^\ep)}\rmb\\
    &\le C \ep^2 |V_p^\ep|^{1/2} \norm{\eta}_{H^2(V_p^\ep)},
\end{split}
\]
where we have applied \eqref{eq.property.Xi} and \eqref{eq.property.Xi.hessian}. On the other hand, by using trace theorem and Poincar\'{e} inequality together and Corollary \ref{cor.boundonpitilde}, we know that 
\[
\begin{split}
    |J_{22}| & \le C \ep\norm{\hat{\eta}}_{L^2(\pt B_{R_\ep}(p))} A(\pt B_{R_\ep}(p))^{1/2}\\
    &\le C \ep^{1+1/2} \norm{\gd\hat{\eta}}_{L^2( V_p^\ep)} A(\pt B_{R_\ep}(p))^{1/2}\\
    &\le C(M,d,g,S) \ep  \norm{\gd \hat{\eta}}_{L^2(V_p^\ep)} |V_p^\ep|^{1/2}\\
    &\le C \ep^2 |V_p^\ep|^{1/2} \norm{\eta}_{H^2(V_p^\ep)},
\end{split}
\]
where the last line is similar to $J_{21}$. Combining the previous two displays we have 
\be
\label{eq.boundJ2.H2}
|J_{2}| \le C(M,d,g,S) \ep^2 |V_p^\ep|^{1/2} \norm{\eta}_{H^2(V_p^\ep)}.
\ee
For $J_3$, we also decompose the integral
\[
\begin{split}
    J_3 &= \avg{\beta}_{V_p^\ep\setminus B_{R_{\ep}}(p)}\int_{V_p^\ep\setminus B_{R_{\ep}}(p)} \gd \eta \cdot \gd \xi \dV  - \avg{\Tilde{\pi}}_{\pt B_{R_\ep}(p)}\int_{\pt B_{R_\ep}(p)}\gd \eta \cdot \gd \xi \dA\\
&=\int_{V_p^\ep} \beta \gd \eta \cdot \gd \xi \dV + \int_{V_p^\ep} \lb \avg{\beta}_{V_p^\ep\setminus B_{R_{\ep}}(p)} - \beta  \rb \gd \eta \cdot \gd \xi \dV \\
&\qquad- \lb\avg{\beta}_{V_p^\ep\setminus B_{R_{\ep}}(p)}\int_{ B_{R_{\ep}}(p)} \gd \eta \cdot \gd \xi \dV  + \avg{\Tilde{\pi}}_{\pt B_{R_\ep}(p)}\int_{\pt B_{R_\ep}(p)}\gd \eta \cdot \gd \xi\dA\rb\\
&=J_{31} + J_{32} - J_{33}
\end{split}
\]
Note that 
\be\label{eq.J31}
\lw\sum_{p\in S_\ep} J_{31}(p)\rw = \lw\int_{M} \beta \gd \eta \cdot \gd \xi \dV \rw \le C(M,d,g,S) \norm{\eta}_{H^2(M)} \norm{\xi}_{L^\infty(M)} \le C\ep^2,
\ee
so we only worry about $J_{32}$ and $J_{33}$. By a similar argument as in \eqref{eq.boundJ1.H2} and also applying \eqref{eq.property.Xi}, we know that 
\[
|J_{32}| \le C(M,d,g,S) \ep ^2 |V_p^\ep|^{1/2}\norm{\gd \eta}_{L^2(V_p^\ep)}.
\]
We write $J_{33}$ as 
\[
\begin{split}
  J_{33} & =   \avg{\beta}_{V_p^\ep\setminus B_{R_{\ep}}(p)}\int_{ B_{R_{\ep}}(p)} \gd \eta \cdot \gd \xi \dV  + \avg{\Tilde{\pi}}_{\pt B_{R_\ep}(p)}\int_{\pt B_{R_\ep}(p)}\gd \eta \cdot \gd \xi  \dA\\
  &=: J_{331} + J_{332}.
\end{split}
\]
We observe that 
\[
\begin{split}
   |J_{331}| &\le C \lw\int_{B_{R_\ep}(p)} \gd \eta\cdot \gd \xi \dV\rw  \\
   &\le C \lw \int_{B_{R_\ep}(p)} \xi \Delta \eta \dV\rw + C \frac{R_\ep^2}{2}\lw \int_{\pt B_{R_\ep}(p)} \pt_\nu \eta \dA \rw\\
 &\le C(M,d,g,S) \ep^2 |V_p^\ep|^{1/2} \norm{\eta}_{H^2(B_{R_\ep}(p))},
\end{split}
\]
and by Corollary \ref{cor.boundonpitilde} again
\[
\begin{split}
    |J_{332}| & \le C \ep R_\ep\lw \int_{\pt B_{R_\ep}(p)} \pt_\nu \eta \dA\rw\\
    &\le C \ep^2 \lw \int_{B_{R_\ep}(p)} \Delta \eta \dV \rw\\
    & \le C \ep^2 |V_p^\ep|^{1/2} \norm{\eta}_{H^2(B_{R_\ep}(p))}. 
\end{split}
\]
Combining \eqref{eq.J31}  and the previous displays about $ J_{32}$ and $J_{33}$, we obtain 
\be
\label{eq.boundJ3.H2}
\lw \sum_{p\in S_\ep } J_3(p)\rw \le C \ep^2 \norm{\eta}_{H^2(M)}.
\ee
The proof is finished by combining \eqref{eq.computeleta}, \eqref{eq.boundJ1.H2}, \eqref{eq.boundtildeJ1.H2}, \eqref{eq.boundJ2.H2} and \eqref{eq.boundJ3.H2}.

\end{proof}

\subsection{A sharper \texorpdfstring{$L^2(\pt\Omega)$}{L2}-estimate of \texorpdfstring{$\Ra_\ep$}{R}}

\begin{lemma}
     \label{l.sharperL2estimateforRaep.boundary}
    Under the assumptions of Lemma \ref{l.sharperL2estimateforRaep}, then there is a constant $C(M,d,g,S)>0$ such that 
     \be
\label{eq.sharperL2estimateforRaep.boundary}
\norm{\Ra_\ep}_{L^2(\pt\Omega)} \le C\ep^{\frac{d}{d-1}},
    \ee
    where $\Omega \subset M$ is the domain obtained by removing $\overline{B}_p$ with $p\in S_\ep$. In particular, in the case $d=2$, we have
    \[
    \norm{\Ra_\ep}_{L^2(\pt\Omega)} \le C\ep^{2}.
    \]

\end{lemma}
To prove Corollary \ref{l.sharperL2estimateforRaep.boundary}, we require the following auxiliary lemmata.

\begin{lemma}
    \label{l.HarnackonManifolds}
    Let $B_{\delta t}\subset B_{t}$ be two geodesic balls centered at $x\in M$ with $\delta\in (0,1)$ and the radius $0< t < r_0(M,g)$ the injectivity radius of $(M,g)$ (which can be chosen to be independent of $x$ because of compactness of $M$). Suppose $h$ is a continuous harmonic function on $B_t$, then there is a constant $C(M,g)>0$ such that
    \be
\label{eq.harnackonManifold}
\osc_{B_{\delta t}} h \le C  \delta  \osc_{B_t} h.
    \ee
\end{lemma}

\begin{proof}
See \cite{li_geometric_2012}*{Corollary 6.2}. 
\end{proof}

\begin{lemma}\label{l.meanvalueproperty}
Fix $x_0\in M$ and let $0<r<R$ be both small numbers. If $u$ is harmonic on the geodesic ball $B_R(x_0)$, then the mean value
\[
M(t)=\frac{1}{|B_t(x_0)|}\int_{B_t(x_0)} u \dV \textup{ and }S(t)=\frac{1}{A(\partial B_t(x_0))}\int_{\partial B_t(x_0)} u \dA
\]
satisfy
\[
|M(R)-M(r)|+|S(R)-S(r)|\le C\,R^2\sup_{B_R(x_0)}|u|,
\]
where $C$ is a constant depending only on the manifold.
\end{lemma}

\begin{remark}
    \label{r.sphericalmean.annuli}
    The bound on the spherical means can be improved to
    \[
    |S(R)-S(r)|\le C\,R^2\sup_{B_R(x_0)\setminus B_r(x_0)}|u|
    \]
    for $u$ that is harmonic merely on the annulus $B_R(x_0)\setminus B_r(x_0)$.
\end{remark}

\begin{proof}
In a geodesic polar coordinate $(t,\varphi)$ based at $x_0$, we denote for $t<R$
\[
g(t) : = \bca
\log(1/t) & \textup{ if }d=2\\
1/t^{d-2} & \textup{ if }d\ge3.
\eca
\]
By testing $\Delta u = 0$ with $g(t)$ in the annulus $D_{r,R}:=\{(t,\varphi)\;:\;r< t < R\}$, we obtain
\[
\lw\frac{1}{R^{d-1}}\int_{\pt B_R} u \dA - \frac{1}{r^{d-1}}\int_{\pt B_r} u \dA \rw \le \sup_{B_R(x_0)}|u| \int_{D_{r,R}} |\Delta g(t)| \dV.
\]
Note that by \eqref{eq.localCalcofLaplacian} and \eqref{eq.boundofChristofelSymbol}, we have
\[
|\Delta g(t)| \le C t |g'|(t) \le C\frac{1}{t^{d-2}} .
\]
This shows that 
\[
\lw\frac{1}{R^{d-1}}\int_{\pt B_R} u \dA - \frac{1}{r^{d-1}}\int_{\pt B_r} u \dA \rw \le C R^2\sup_{B_R(x_0)}|u| .
\]
Note that here $r<R$ is arbitrary, we also know that for all $\tau_1<\tau_2<R$
\[
\lw\frac{1}{\tau_2^{d-1}}\int_{\pt B_{\tau_2}} u \dA - \frac{1}{\tau_1^{d-1}}\int_{\pt B_{\tau_1}} u \dA \rw \le C R^2\sup_{B_R(x_0)}|u| .
\]
In particular,
\[
\tau_1^{d-1}\int_{\pt B_{\tau_2}} u \dA = \tau_2^{d-1}\int_{\pt B_{\tau_1}} u \dA + O\lb\tau_1^{d-1}\tau_2^{d-1}R^2\sup_{B_R(x_0)}|u|\rb.
\]
Integrating both sides in $\tau_1\in (0,r)$ and $\tau_2\in (0,R)$, we obtain by the co-area formula
\[
\lw\frac{1}{R^{d}}\int_{ B_R} u \dV - \frac{1}{r^{d}}\int_{B_r} u \dV \rw \le C R^2\sup_{B_R(x_0)}|u| .
\]
The proof is finished by recalling the following expansions of volumes and areas of geodesic balls and spheres
\be\label{eq.volumeandarea.expansion}
|B_t(x_0)|
=
\omega_n t^n
\left(
1+ O(t^2)
\right) \textup{ and }A(\partial B_t(x_0))
=
n\omega_n t^{n-1}
\left(
1+O(t^2)
\right).
\ee
For the reference of these expansions, we refer to \cite{chavel_eigenvalues_1984}*{Section XII.8}.

\end{proof}

\begin{proof}[Proof of Lemma \ref{l.sharperL2estimateforRaep.boundary}]

Note that $\Ra_\ep$ is harmonic on $\bigcup_{p\in S_\ep} B_{R_\ep}(p)$, and therefore, by applying Lemma \ref{l.globalosc.Raep} and Lemma \ref{l.HarnackonManifolds} we have for $d\ge 2$
\be
\label{eq.localoscbound.Raep}
\max_{p\in S_\ep}\osc_{\overline{B}_p} \Ra_\ep \le C\frac{\osc_{M} \Ra_\ep}{R_\ep}\max_{p\in S_\ep} r_{\ep,p} \le C \ep^{\frac{d}{d-1}},
\ee
where $\overline{B}_p = \overline{B}_{r_{\ep,p}}(p)$.
On the other hand, we know that by Lemma \ref{l.meanvalueproperty}, there is a universal $C>0$ so that for every $p\in S_\ep$
\[
\lw \avg{\Ra_\ep}_{B_{R_\ep}(p)} - \avg{\Ra_\ep}_{B_p}\rw \le C R_\ep^2 \norm{\Ra_\ep}_{L^\infty(M)} \le C \ep^3.
\]
This shows that 
\[
\begin{split}
   \norm{\Ra_\ep}_{L^2(\pt \Omega)}&\le \lb \sum_{p\in S_\ep} A(\pt B_p) \avg{\Ra_\ep}_{B_p}^2  \rb^{1/2} + C\ep^{\frac{d}{d-1}}\\
 &\le \lb \sum_{p\in S_\ep} A(\pt B_p) \avg{\Ra_\ep}_{B_{R_\ep}(p)}^2  \rb^{1/2} + C\ep^{\frac{d}{d-1}}\\
 &\le C\lb \sum_{p\in S_\ep} |V_p^\ep| \avg{\Ra_\ep}_{B_{R_\ep}(p)}^2  \rb^{1/2} + C\ep^{\frac{d}{d-1}}\\
 &\le C \norm{\Ra_\ep}_{L^2(M)} + C \ep^{\frac{d}{d-1}}\\
 &\le C\ep^{\frac{d}{d-1}},
\end{split}
\]
 where the last line follows from Lemma \ref{l.sharperL2estimateforRaep}.   
\end{proof}

\subsection{An \texorpdfstring{$L^\infty(M)$}{L2}-estimate of \texorpdfstring{$\Ra_\ep$}{R} for centralized \texorpdfstring{$S_\ep$}{S}}

In this subsection we prove an improved $L^\infty$ estimate for $\Ra_\ep$ and therefore $\Phi_\ep$ in the case that the points $p\in S_\ep$ are also the mass centers of $V_p^\ep$ in the sense of \eqref{eq.centroid.voronoi}.

\begin{lemma}
    \label{l.centroidvoronoi.Linfity}
     Under the assumptions of Lemma \ref{l.globalosc.Raep} with the additional assumption \eqref{eq.centroid.voronoi} that 
     \be\label{eq.centroid.voronoi.proof}
     \int_{V_p^\ep} \exp_p^{-1}(x) \beta(x) \dV = 0 \in T_pM
     \ee
  for all $p\in S_\ep$,  then there is a constant $C(M,d,g,S)>0$ such that 
    \be
\label{eq.globalosc.Raep.centroid}
\norm{\Ra_\ep}_{L^\infty(M)} \le C \ep^2|\log \ep|.
    \ee
    In particular, by combining \eqref{eq.phibar.linfty.h1} there is 
       \be
\label{eq.globalosc.Phiep}
\norm{\Phi_\ep}_{L^\infty(M)} \le C \omega_d(\ep).
    \ee
\end{lemma}

\begin{proof}

As in the proof of Lemma \ref{l.globalosc.Raep}, we use the following representation of $\Ra_\ep$
\be
\label{eq.aGreenrepofRa.withdecomp}
\Ra_\ep(x) = \int_{M} G_x(y) \drho(y) = \sum_{p\in S_\ep}  \int_{V_p^\ep} G_x(y) \drho_p(y),
\ee
where $G_x$ is the Green function on $(M,g)$ that solves
\[
\Delta_g G_x = \delta_x - \frac{\beta}{\int_M \beta \dV}.
\]
As in \eqref{eq.decomposeerror}, for $x\in V_p^\ep$ (we may without loss assume that $x$ is in the interior of $V_p^\ep$) for some $p\in S_\ep$, we write
\[
\Ra_\ep(x) = \underbrace{\sum_{\dist(q,p)\le 8\ep}\int_{V_q^\ep}G_x \drho_q }_{A_1} ~+~\underbrace{\sum_{\dist(q,p)>8\ep} \int_{V_q^\ep} G_x \drho_q}_{A_2}
\]
By the same argument as in \eqref{eq.boundA1}, we obtain
\[
|A_1| \le C \ep^2 |\log \ep|.
\]
To estimate $A_2$, we fix an arbitrary $q\in S_\ep$ such that $\dist(q,p)>8\ep$, and then we compute in a geodesic coordinate $z\in \R^d \cong T_qM$
\[
\begin{split}
    \lw\int_{\pt Q_q}\gd G_x(q)\cdot \exp_q^{-1}(y) ~ \Tilde{\pi} \dA\rw & = \lw\int_{\pt B_{R_\ep}(0)} \gd_z G_x(0) \cdot z ~\Tilde{\pi}(z) \sigma (z)\dd \theta(z)\rw\\[5pt]
&\le C \sup_{V_q^\ep}{|\gd G_x|}~ \ep A(\pt Q_q) \osc_{z\in\pt B_{R_\ep}(0)} (\Tilde{\pi}(z) \sigma (z)),
\end{split}
\]
where $\sigma$ is the area element of $\pt Q_q$ in the coordinate $z$, $\theta$ is the Lebesgue measure on the round sphere $\pt B_{R_\ep}(0)$, and the last line above is obtained by the symmetry of the integral. Because in the geodesic coordinate $z$ the metric $g_{ij}(z) = \delta_{ij} + O(|z|^2)$ (See \cite{misner1973gravitation}*{Page 285}), we have by combining \eqref{eq.boundonpitilde}
\[
\osc_{z\in\pt B_{R_\ep}(0)} (\Tilde{\pi}(z) \sigma (z)) \le C \ep^2,
\]
which shows that 
\be\label{eq.centroid.symmetry.geodesicsphere}
 \lw\int_{\pt Q_q}\gd G_x(q)\cdot \exp_q^{-1}(y) ~ \Tilde{\pi} \dA \rw \le C \ep^2 \sup_{V_q^\ep}{|\gd G_x|} |V_q^\ep|.
\ee
Now we have by the balance of mass \eqref{eq.normalizationofRa}, \eqref{eq.centroid.symmetry.geodesicsphere} and the centroid assumption \eqref{eq.centroid.voronoi.proof}
\be
\label{eq.estimateA2.centroid}
\begin{split}
   \lw\int_{V_q^\ep} G_x \drho_q\rw & \le  \lw\int_{V_q^\ep} G_x(y)  - G_x(q) - \gd G_x(q)\cdot \exp_q^{-1}(y)  \drho_q(y)\rw\\
   &\qquad + \lw\int_{\pt Q_q}\gd G_x(q)\cdot \exp_q^{-1}(y) \Tilde{\pi} \dA\rw\\
 &\le C \ep^2 \|\rho_q\|_{TV} \sup_{ V_q^\ep} |\gd^2 G_x| + C \ep^2 |V_q^\ep| \sup_{V_q^\ep}{|\gd G_x|}\\
 &\le C \frac{\ep^{d+2}}{\dist(p,q)^{d}},
\end{split}
\ee
where we have used the Hessian and gradient estimate (See for example \cite{aubin_nonlinear_1982}*{Chapter 4}) for the Green function and the constant $C=C(M,d,g,S)>0$ may be different from line to line.
Similar to the proof of Lemma \ref{l.globalosc.Raep}, we divide the following set dyadically into
\[
S_\ep \setminus \{q\;;\; \dist(q,p)\le 2^3\ep\} = \bigcup_{k = 0}^N \left\{q\in S_\ep \;:\; \dist(q,p)\in (2^{k+3}\ep, 2^{k+4}\ep]\right\}=:\bigcup_{k = 0}^N \mathcal{O}_k,
\]
where $\diam(M)\le 2^{N + 4}\ep \le 2\diam(M)$ and therefore $N = O(|\log\ep|)$. Note that each $\Oa_k \subset \{x\in M\;:\; \dist_g(x,p)\le 2^{k+4}\ep\}$, but the latter can only have at most $C 2^{kd}$ points in $S_\ep$ due to the $\ep$-separatedness of $S_\ep$. This, combining \eqref{eq.estimateA2.centroid}, implies that 
\[
\begin{split}
    |A_2| & \le C\sum_{\dist(q,p)>2^3\ep} \frac{\ep^{d+2}}{\dist(p,q)^{d}}\\
    &\le C \sum_{k=0}^N \# \Oa_k \cdot \frac{\ep^{d+2}}{(2^k\ep)^{d}} \\
    &\le C \sum_{k=0}^N 2^{kd} \cdot \frac{\ep^{d+2}}{(2^k\ep)^{d}} \\
    &\le C \ep^2|\log \ep|.
\end{split}
\]
This completes the proof.
    
\end{proof}

\subsection{Proof of Theorem \texorpdfstring{\ref{t.summaryofPhi}}{2}}

\begin{proof}[Proof of Theorem \ref{t.summaryofPhi}]

The proof is done by combining Corollary \ref{cor.phibar.model.norms}, Lemma \ref{l.H1estimateforRaep}, Lemma \ref{l.globalosc.Raep}, Proposition \ref{prop.comparetwocellfunction}, Lemma \ref{l.sharperL2estimateforRaep}, Lemma \ref{l.sharperL2estimateforRaep.boundary} and Lemma \ref{l.centroidvoronoi.Linfity}.

\end{proof}

\section{Optimal convergence rates of the eigenvalue and functions}
\label{th1.2}
In this section, we prove Theorem \ref{t.quantitative.StekNeum} and in particular address the optimal convergence rate of eigenvalues and eigenfunctions. 

\subsection{Optimal convergence rate of the eigenvalues}
\label{subsection.optconvergerate.eigenvalue}
In this subsection, we focus on proving \eqref{eq.Optestimate.stektoNeum} in Theorem \ref{t.quantitative.StekNeum}. Let us make a specific lemma as follows.

\begin{lemma}
    \label{l.Optestimate.stektoNeum}
    Same assumptions in Theorem \ref{t.quantitative.StekNeum}, then for $\ep<\ep_0(M,d,g,S,k)$ there is a constant $C=C(M,d,g,S,k)>0$ such that
    \be\label{eq.Optestimate.stektoNeum.proof}
   |\lambda_k - \sigma_k^\ep| \le C {\omega}_d(\ep).
    \ee
\end{lemma}

\smallskip \noindent Let us begin with the strategy of proving Lemma \ref{l.Optestimate.stektoNeum}. Denote $u^\ep=u_k^\ep$, $u=u_k$ as the $k$-th (harmonically extended) Steklov and Laplace--Beltrami eigenfunctions respectively as defined and normalized in Section \ref{subsection.Steklov.Neumann.definition}, then we have
\be
\label{eq.atrickintegrate}
\begin{split}
    \sigma_k^\ep \int_{\pt \Omega} u^\ep ~u~ \dA & = \int_{\pt \Omega} \pt_\nu u^\ep ~u ~ \dA\\
    &=\int_{\Omega} \gd u^\ep \cdot \gd u \dV\\
    &=\int_{\pt \Omega} u^\ep \pt_\nu u ~\dA - \int_{\Omega}u^\ep \Delta u \dV\\
    &=\int_{\pt \Omega} u^\ep \pt_\nu u ~ \dA + \lambda_k \int_{\Omega}\beta u^\ep ~u ~\dV.
\end{split}
\ee
This implies that by using the cell function $\Phi_\ep$ as defined in  \eqref{eq.cellfunction}
\be
\label{eq.atrickintegrate.2}
\begin{split}
    (\sigma_k^\ep - \lambda_k) \int_{\pt\Omega} u^\ep ~u~ \dA &= \lambda_k\lmb\int_M u^\ep u \lb\beta\dV - \restr{\dA}{\pt\Omega}\rb\rmb + \int_{\pt\Omega} u^\ep\pt_\nu u \dA - \lambda_k \int_{M\setminus \Omega} \beta  u^\ep u \dV\\
    &= - \lambda_k \int_M \gd \Phi_\ep \cdot \gd (u^\ep u) \dV+ \int_{\pt\Omega} u^\ep\pt_\nu u \dA- \lambda_k \int_{M\setminus \Omega} \beta  u^\ep u \dV\\
    &= - \lambda_k \int_M \gd \Phi_\ep \cdot \gd (u^\ep u) \dV -  \int_{M \setminus \Omega} \gd u^\ep \cdot \gd u \dV\\
& =: J_1 + J_2 .
\end{split}
 \ee
By Theorem \ref{t.qualitativeconverge} and specifically the arguments in \cite{girouard_large_2021}*{Proposition 5.2}, we can choose $u^\ep$ and $u$ so that 
\[
\lim_{\ep\rta0} \int_{\pt\Omega} u^\ep u \dA = \int_M \beta u^2 \dV = 1.
\]
Therefore, it then suffices to estimate $J_1$ and $J_2$ in \eqref{eq.atrickintegrate.2} in the following lemmas.

\begin{remark}
    \label{r.sizeofprojectionforDifferentOrdereigenv}
    One can also do the same integration by parts argument in \eqref{eq.atrickintegrate.2} for eigen-pairs $(\sigma_l^\ep, u_l^\ep)$ and $(\lambda_k,u_k)$ in the case that $l\ne k$ and $|\sigma_l^\ep-\lambda_k| \ge c$ for some $c>0$ and all small $\ep>0$. In this case, we have the same estimate as the eigenvalues for 
    \be
\label{eq.convergencerate.projection.coefficients}
\lw\int_{\pt\Omega} u_l^\ep ~u_k~ \dA \rw \le \frac{1}{c}\lw J_1 + J_2 \rw \le C(M,d,g,S,l,k) \omega_d(\ep).
    \ee
\end{remark}

\begin{lemma}
    \label{l.Optestimate.stektoNeum.J1}
    Under the assumptions of Lemma \ref{l.Optestimate.stektoNeum}, then the term $J_1$ as defined in \eqref{eq.atrickintegrate.2} satisfies
\be
\label{eq.boundJ1.convergerate.done}
|J_1| \le C \omega_d(\ep).
\ee
\end{lemma}

\begin{proof}

Note that by integrating by parts
\be
\label{eq.boundJ1.convergerate}
\begin{split}
   J_1 & = -  \lambda_k\int_M \gd \Phi_\ep \cdot \gd (u^\ep u) \dV \\
&= -\lambda_k \sigma_k^\ep \int_{\pt\Omega}  \Phi_\ep u^\ep u \dA - \lambda_k\int_{\pt\Omega}  \Phi_\ep \Lambda_\ep(u^\ep) u \dA + \lambda_k \int_M \Phi_\ep \lb2 \gd u \cdot \gd u^\ep - \lambda_k u^\ep u \rb \dV\\
&=:J_{11} + J_{12} + J_{13},
\end{split}
\ee
where $\Lambda_\ep$ is the Dirichlet-to-Neumann map on the boundary $\pt \Omega$ and we also have $\Lambda_\ep(u^\ep)= - \pt_\nu^{-}(u^\ep)$. 
By Theorem \ref{t.summaryofPhi}, we decompose $\Phi_\ep = \bar \Phi_\ep + \Ra_\ep$ and obtain
\be
\label{eq.boundJ11.convergerate}
\begin{split}
    |J_{11}| &\le C\int_{\pt\Omega}  |\bar{\Phi}_\ep u^\ep u| \dA + C \int_{\pt\Omega}  |\Ra_\ep u^\ep u| \dA\\
    &\le C\lb \norm{\bar{\Phi}_\ep}_{L^\infty(M)} + \norm{\Ra_\ep}_{L^2(\pt \Omega)} \rb \norm{u}_{L^\infty(M)} \norm{u^\ep}_{L^2(\pt\Omega)}\\
    &\le C \omega_d(\ep).
\end{split}
\ee
To bound $J_{12}$, we recall the decomposition in Section \ref{subsection.radialmodel.barPhi}
\[
\Phi_\ep = \La_\ep + \lb \bar \Phi_\ep - \La_\ep\rb + \Ra_\ep
\]
and then we can write
\be
\label{eq.boundJ12.convergerate}
\begin{split}
  J_{12} &=  \lambda_k\int_{\pt\Omega}  \La_\ep \Lambda_\ep(u^\ep) u \dA  + \lambda_k\int_{\pt\Omega}  (\bar{\Phi}_\ep - \La_\ep)\Lambda_\ep(u^\ep) u \dA  + \lambda_k \int_{\pt\Omega} \Ra_\ep \Lambda_\ep(u^\ep) u \dA\\
  &=\lambda_k \sum_{p\in S_\ep} \lb\La_\ep(p) + \frac{\beta(p)}{2d}r_{\ep,p}^2\rb \int_{\pt B_p} \Lambda_\ep(u^\ep) u \dA + \lambda_k\int_{M\setminus\Omega} \gd \lmb (\bar{\Phi}_\ep - \La_\ep) u \rmb\cdot \gd u^\ep \dV \\
  &\qquad +\lambda_k \int_{M\setminus\Omega} \gd (\Ra_\ep u) \cdot \gd u^\ep \dV \\
  &=:J_{121} +J_{122} + J_{123}.
\end{split}
\ee
We observe that because $u^\ep$ is harmonic in $M\setminus \Omega$, by integrating by parts and the explicit formula \eqref{eq.Modelfunction.Bulkadapted}
\[
|J_{121}| \le C \norm{\La_\ep}_{L^\infty(M)} \norm{u^\ep}_{H^1(M)} \norm{u}_{C^1(M)}|M\setminus\Omega|^{1/2} \le C \omega_d(\ep) \ep^{\frac{d}{2(d-1)}}\le O(\omega_d(\ep)^{3/2}).
\]
Next, we have by applying Lemma \ref{l.difference.modelandPhibar}
\[
|J_{122}| \le C \norm{\gd(\bar{\Phi}_\ep - \La_\ep)}_{C^1(M\setminus\Omega)}  \norm{u}_{C^1(M)}\norm{u^\ep}_{H^1(M)} |M\setminus \Omega|^{1/2}\le C \ep^{2+\frac{d}{2(d-1)}} = o(\ep^2).
\]
As for $J_{123}$, we first observe that $\Ra_\ep$ is harmonic in each $Q_p$ and hence by \eqref{eq.radii.estimate} and the Almgren's (almost) monotonicity formula (See \cites{garofalo_monotonicity_1986})
\be
\label{eq.Dirichletdecay}
\int_{B_p} |\gd \Ra_\ep|^2 \dV \le C\lb \frac{r_{\ep,p}}{R_\ep} \rb^d \int_{Q_p} |\gd \Ra_\ep|^2 \dV \le C \ep^{\frac{d}{d-1}} \int_{Q_p} |\gd \Ra_\ep|^2 \dV.
\ee
Summing over all $p\in S_\ep$ and applying Lemma \ref{l.H1estimateforRaep}, we obtain
\[
\norm{\gd \Ra_\ep}_{L^2(M\setminus \Omega)} \le C \ep^{1+\frac{d}{2(d-1)}}.
\]
This implies that, by Lemma \ref{l.sharperL2estimateforRaep}
\[
|J_{123}|\le C \norm{u}_{C^1(M)}\norm{u^\ep}_{H^1(M)}\lmb \norm{\gd \Ra_\ep}_{L^2(M\setminus \Omega)} + \norm{\Ra_\ep}_{L^2(M)} \rmb \le C \ep^{1+\frac{d}{2(d-1)}} \le C \omega_d(\ep).
\]
Combining the previous estimates about $J_{121}, J_{122}$ and $J_{123}$ we obtain 
\be
\label{eq.boundJ12.convergerate.done}
|J_{12}|\le C \omega_d(\ep).
\ee
The term $J_{13}$ is bounded by the following inequality
\be
\label{eq.boundJ13.convergerate}
|J_{13}| \le C \norm{\Phi_\ep}_{L^2(M)} \norm{u}_{C^1(M)}\norm{u^\ep}_{H^1(M)} \le C \ep^{\frac{d}{d-1}},
\ee
where we have used the estimates in \eqref{eq.Phi.summary}. 

Combining \eqref{eq.boundJ11.convergerate}, \eqref{eq.boundJ12.convergerate.done} and \eqref{eq.boundJ13.convergerate}, we finish the proof of the lemma.

\end{proof}

\begin{lemma}
    \label{l.Optestimate.stektoNeum.J2}
    Under the assumptions of Lemma \ref{l.Optestimate.stektoNeum}, then the term $J_2$ as defined in \eqref{eq.atrickintegrate.2} satisfies
\be
\label{eq.boundJ2.convergerate.done}
|J_2| \le C \ep^{\frac{d}{d-1}}.
\ee
\end{lemma}
To prove this result, we need the following auxiliary lemma.
\begin{lemma}
    \label{l.transferingbetween.geodesicchart}
    Let $B_r(q) \subset (M,g)$ be a geodesic ball of radius $r \ll r_0(M,g)$ the injective radius. Let $z \in B_r(0) \subset \R^d \cong T_q(M)$ be a geodesic coordinate of $B_r(q)$. Then for every $f\in L^1(B_r(q))$ we have (with abuse of notation $f(z):=f(\exp_q(z))$)
    \be
\label{eq.l.transferingbetween.geodesicchart.ball}
\lw \int_{B_r(q)} f \dV -  \int_{B_r(0)} f(z) \dd z\rw \le C r^2 \int_{B_r(q)} |f| \dV .
    \ee
    Similarly if $f \in L^1(\pt B_r(q))$ we have 
    \be
\label{eq.l.transferingbetween.geodesicchart.sphere}
\lw \int_{\pt B_r(q)} f \dA -  \int_{\pt B_r(0)} f(z) \dd \theta(z)\rw \le C r^2 \int_{\pt B_r(q)} |f| \dA ,
    \ee
    where $\theta$ is the Lebesgue measure on the round sphere $\pt B_r(0)\subset \R^d$. Here the constant $C=C(M,d,g)>0$.
\end{lemma}

\begin{proof}
Recall the Taylor expansion of the Riemannian metric and Christoffel symbols in normal coordinates (See \cite{misner1973gravitation}*{Page 285})
\be\label{eq.expansionofmetric}
g_{ij}(z) = \delta_{ij} + O(|z|^2) \textup{ and }{\Gamma^k}_{ij} (z)=O(|z|).
\ee
This implies that the volume element
\[
\sigma(z):=\frac{\lb\exp_q^{-1}\rb_{\#}\lb\dV \rb}{\dd z} = 1 + O(|z|^2)
\]
and the area element
\[
\sigma_r(z):=\frac{\lb\exp_q^{-1}\rb_{\#}\lb\restr{\dA}{\pt B_r(q)}\rb}{\dd \theta} = 1+ O(r^2).
\]
These expansion formulas immediately prove the lemma.

\end{proof}

\begin{proof}[Proof of Lemma \ref{l.Optestimate.stektoNeum.J2}]

We write 
\be\label{eq.decomposeJ2}
J_2 = -  \int_{M \setminus \Omega} \gd u^\ep \cdot \gd u \dV = -\sum_{p\in S_\ep} \int_{B_p} \gd u^\ep \cdot \gd u \dA = : -\sum_{p\in S_\ep} J_{2,p}.
\ee
For each $p\in S_\ep$ we use a local geodesic coordinate system $z\in B_{r_{\ep,p}}(0)\subset\R^d \cong T_pM$, and then we can rewrite 
\[
J_{2,p}=\int_{B_p} \gd u \cdot \gd u^\ep \dV  = \int_{B_{r_{\ep,p}}(0)} g^{ij}\pt_{z_i} u ~ \pt_{z_j} u^\ep~ \sigma(z) \dd z,
\]
where $g^{ij}=(g^{-1})_{ij}$ is the inverse of the Riemannian metric $g$ and $\sigma(z) = \sqrt{|\det{g}(z)|}$ is the volume element. 
By \eqref{eq.expansionofmetric} we have
\[
g^{ij}(z) = \delta_{ij} + O(|z|^2) \textup{ and }\sigma(z) = 1 + O(|z|^2).
\]
This implies that if we denote 
\[
\Tilde{J}_{2,p} := \int_{B_{r_{\ep,p}}(0)} \gd_z u \cdot \gd_z u^\ep \dd z,
\]
where $\gd_z=(\pt_{z_i})_{i=1}^d$ is the gradient in the normal coordinates $z$, then by Lemma \ref{l.transferingbetween.geodesicchart}
\be\label{eq.flatexpansion.J2p}
\sum_{p\in S_\ep}|J_{2,p} - \Tilde{J}_{2,p}| \le  C \sum_{p\in S_\ep}r_{\ep,p}^2\norm{\gd u}_{L^2(B_p)}\norm{\gd u^\ep}_{L^2(B_p)} = o(\ep^2),
\ee
where $C=C(M,d,g,S)>0$. 
It then suffices to bound $\Tilde{J}_{2,p}$. To that end we denote the vector $\xi : = \gd_z u(0) (=\gd u(p))$  and then we can write
\[
\begin{split}
 \Tilde{J}_{2,p} & =\int_{B_{r_{\ep,p}}(0)} (\gd_z u-\xi) \cdot \gd_z u^\ep \dd z + \int_{B_{r_{\ep,p}}(0)} \xi \cdot\gd_z u^\ep \dd z\\
  &=\int_{B_{r_{\ep,p}}(0)} (\gd_z u-\xi) \cdot\gd_z u^\ep \dd z + \int_{\pt B_{r_{\ep,p}}(0)} u^\ep ~\xi\cdot \frac{z}{|z|}   \dd \theta(z)\\
  &=:\Tilde{J}_{2,p,1}+ \Tilde{J}_{2,p,2},
\end{split}
\]
where $\theta$ is the area measure on corresponding Euclidean spheres. Note that by the smoothness of $u$
\be\label{eq.boundJ2.p.1}
|\Tilde{J}_{2,p,1}| \le C r_{\ep,p} \norm{u}_{C^2(M)} |B_p|^{1/2} \norm{\gd u^\ep}_{L^2(B_p)}.
\ee
This term sums over $p\in S_\ep$ to a small term 
\be
\label{eq.boundJ2.p.1.sumoverp}
\sum_{p\in S_\ep} |\Tilde{J}_{2,p,1}| \le C \ep^{\frac{3}{2}\cdot\frac{d}{d-1}} \le C \omega_d(\ep)^{\frac{3}{2}}.
\ee
For the second term, we denote 
\be\label{eq.define.w.xi.tilde}
w_\xi(z):= \xi\cdot\lb\frac{z}{|z|}\rb\,\frac{r_{\ep,p}^{d-1}\bigl(R_\ep^{d} - |z|^{d}\bigr)}{|z|^{d-1}\bigl(R_\ep^{d} - r_{\ep,p}^{d}\bigr)} \textup{ and } \Tilde{w}_\xi (x) : = w_\xi (\exp_p^{-1}(x)).
\ee
Also by Lemma \ref{l.transferingbetween.geodesicchart}, we have
\be\label{eq.decompose.J2p2}
\lw\Tilde{J}_{2,p,2} - \int_{\pt B_p} u^\ep ~\Tilde{w}_\xi \dA \rw \le C r_{\ep,p}^2 A(\pt B_p)^{1/2} \norm{u^\ep}_{L^2(\pt B_p)},
\ee
the right hand side of which sums over $p\in S_\ep$ to a term of size $o(\ep^2)$.
Therefore, we only have to compute
\[
\int_{\pt B_p} u^\ep ~\Tilde{w}_\xi \dA = \frac{1}{\sigma_k^\ep}\int_{\pt B_p} \pt_\nu^+ u^\ep ~\Tilde{w}_\xi \dA =: T_0(p) = T_0.
\]
By integration by parts, we have
\be\label{eq.decompose.DtN}
T_0 =\frac{1}{\sigma_k^\ep}\int_{\pt B_p} \pt_\nu^+\Tilde{w}_\xi~ u^\ep  \dA - \frac{1}{\sigma_k^\ep}\int_{\pt Q_p} \pt_\nu^-\Tilde{w}_\xi~ u^\ep  \dA - \frac{1}{\sigma_k^\ep} \int_{Q_p\setminus B_p} \Delta_g \Tilde{w}_\xi ~ u^\ep \dV.
\ee
Note that by the explicit formula \eqref{eq.define.w.xi.tilde}, we have on $\pt B_p$
\be
\label{eq.boundary.Neumann.Diri.response.w.xi}
\pt_\nu^+\Tilde{w}_\xi = \frac{1}{r_{\ep,p}} \Tilde{w}_\xi \cdot \underbrace{\lb \frac{(d-1)R_\ep^{d}+r_{\ep,p}^d}{R_{\ep}^d-r_{\ep,p}^d}\rb}_{=: A_\ep}, \textup{ where } A_\ep = d-1 + O(\ep^{\frac{d}{d-1}}).
\ee
This leads to 
\be
\label{eq.decompose.DtN.2}
\begin{split}
    T_0 &= \lb \frac{A_\ep}{r_{\ep,p}} -\sigma_k^\ep \rb^{-1} \lmb\int_{\pt Q_p} \pt_\nu^-\Tilde{w}_\xi~ u^\ep  \dA +  \int_{Q_p\setminus B_p} \Delta_g \Tilde{w}_\xi ~ u^\ep \dV\rmb\\
&=:\lb \frac{A_\ep}{r_{\ep,p}} -\sigma_k^\ep \rb^{-1}\lmb T_1 + T_2\rmb.
\end{split}
\ee
We first write $T_1$ as 
\[
\begin{split}
    T_1 & = \int_{\pt Q_p} \pt_\nu^-\Tilde{w}_\xi~ (u^\ep - \avg{u^\ep}_{Q_p})  \dA + \avg{u^\ep}_{Q_p}\int_{\pt Q_p} \pt_\nu^-\Tilde{w}_\xi  \dA \\
    &= : T_{11} + T_{12}.
\end{split}
\]
Note that in the geodesic coordinate $z$ we have on $\pt Q_p$
\[
|\pt_\nu^-\Tilde{w}_\xi| = \lw \gd_z \Tilde{w}_\xi \cdot \frac{z}{|z|} \rw \le B_\ep \norm{u}_{C^1({M})} \le C
\]
where 
\[
B_\ep = \frac{  r_{\ep,p}^{d-1} \ d}{R_{\ep}^d-r_{\ep,p}^d} = O(1).
\]
This implies that by applying the Poincar\'{e} inequality and trace theorem on $Q_p$ 
\be\label{eq.T11}
\begin{split}
    |T_{11}| & \le C A(Q_p)^{1/2} \norm{u^\ep - \avg{u^\ep}_{Q_p}}_{L^2(\pt Q_p)}\\
    &\le C \ep^{1/2}A(Q_p)^{1/2} \norm{\gd u^\ep}_{L^2(Q_p)}\\
    &\le C |Q_p|^{1/2} \norm{\gd u^\ep}_{L^2(Q_p)},
\end{split}
\ee
which sums up over $p\in S_\ep$ to a term of order $O(1)$. For the term $T_{12}$, we have by Lemma \ref{l.transferingbetween.geodesicchart}
\[
\int_{\pt Q_p} \pt_\nu^-\Tilde{w}_\xi  \dA = \underbrace{-B_\ep\int_{\pt B_{R_\ep}(0)} \lb \xi \cdot \frac{z}{|z|}\rb \dd\theta(z)}_{=0} + O(R_\ep^2 A(\pt Q_p)) = O(\ep |Q_p|).
\]
This shows that 
\[
|T_{12}| \le C \ep |Q_p|^{1/2} \norm{u^\ep}_{L^2(Q_p)},
\]
which sums up over $p\in S_\ep$ to a term of order $O(\ep)$. Combining \eqref{eq.T11}, we obtain
\be
\label{eq.boundT1}
|T_1| \le |T_{11}| +|T_{12}| \le C  |Q_p|^{1/2} \norm{u^\ep}_{H^1(Q_p)}.
\ee
On the other hand, for $|z| \in (r_{\ep,p},R_\ep)$, if we denote $\Delta_z$ as the Euclidean Laplacian in the geodesic coordinate $z$, then by \eqref{eq.expansionofmetric}
\[
|(\Delta_g \Tilde{w}_\xi )(\exp_p^{-1}(z)) |=|(\Delta_g \Tilde{w}_\xi) (\exp_p^{-1}(z))- \Delta_z w_\xi(z)| \le C \lb |z|^2 |\gd _z^2 {w}_\xi(z)| + |z| |\gd _z {w}_\xi(z)| \rb \le C
\]
This implies that 
\be
\label{eq.boundT2}
|T_2| \le C |Q_p|^{1/2} \norm{u^\ep}_{L^2(Q_p)}.
\ee
Combining \eqref{eq.boundT1}, we obtain 
\be
\label{eq.boundT0}
|T_0| \le C r_{\ep,p} |Q_p|^{1/2} \norm{u^\ep}_{H^1(Q_p)} \le C \ep^{\frac{d}{d-1}}|Q_p|^{1/2} \norm{u^\ep}_{H^1(Q_p)}.
\ee
The proof of this lemma is finished by combining \eqref{eq.decomposeJ2}, \eqref{eq.flatexpansion.J2p}, \eqref{eq.boundJ2.p.1}, \eqref{eq.decompose.J2p2}, \eqref{eq.decompose.DtN.2} and \eqref{eq.boundT0}.

\end{proof}

\begin{proof}[Proof of Lemma \ref{l.Optestimate.stektoNeum}]
    The proof is done by combining \eqref{eq.atrickintegrate}, \eqref{eq.atrickintegrate.2}, Lemma \ref{l.Optestimate.stektoNeum.J1} and Lemma \ref{l.Optestimate.stektoNeum.J2}.
\end{proof}

\subsection{Convergence rates of eigenfunctions in \texorpdfstring{$H^1$}{H}}

In this subsection, we show the convergence rates of the Steklov eigenfunctions to the Laplace--Beltrami eigenfunctions as claimed in \eqref{eq.Optestimate.function} and \eqref{eq.Optestimate.function.proj} in Theorem \ref{t.quantitative.StekNeum}. For convenience we recall the measures as defined in \eqref{eq.defineAeps.and.mueps}
\[
\dA_\ep : = \restr{\dA}{\pt \Omega _\ep} \textup{ and } \dd\mu_\ep : = \beta\dV - \dA_\ep.
\]
We also recall the notations in Section \ref{subsection.Steklov.Neumann.definition}. For every $j\ge 0$ and small $\ep$, we denote $(\sigma_j^\ep, v_j^\ep)$ as the Steklov eigenpairs such that
\be
\label{eq.Steklov.normalize}
\int_{\pt\Omega} v_i^\ep v_{j}^\ep \dA = \delta_{ij},
\ee
and $(\lambda_j,u_j)$ as the Laplace--Beltrami eigenpairs such that 
\be
\label{eq.Neumann.normalize}
\int_{M} \beta u_iu_j \dV = \delta_{ij}.
\ee
Note that here $v_0^\ep$ and $u_0$ are all constants and $\lambda_0=\sigma_0^\ep = 0$.
We define $u_j^\ep$ as the harmonic extensions of $v_j^\ep$ to the whole manifold $M$ and we know that by \cite{girouard_large_2021}*{Lemma 3.4}, there is 
\be\label{eq.bound.harmonicextension}
\norm{u_j^\ep}_{H^1(M)} \le C \norm{v_j^\ep}_{H^1(\Omega)}\le C(M,d,g,S,j).
\ee
Now we recall that each $u_j^\ep,\ j\ge 0$ satisfies
\be
\label{eq.eqforUep.eigen.weakform}
\begin{split}
    \int_M \gd u_j^\ep\cdot \gd \eta \dV &= \sigma_j^\ep \int_M u_j^\ep \eta\dA_\ep + \int_M \Lambda_\ep(u_j^\ep) \eta \dA_\ep\\
&= \sigma_j^\ep \int_M u_j^\ep \eta \dA_\ep + \int_{M\setminus\Omega} \gd u_j^\ep\cdot \gd \eta \dV
\end{split}
\ee
and the functions $u_j$ satisfy
\be
\label{eq.eqforU.eigen.weakform}
\int_M \gd u_j \cdot \gd \eta \dV = \lambda \int_M \beta u_j \eta \dV 
\ee
where $\eta \in H^1(M)$ is arbitrary.

\noindent We are concerned about the $H^1$ distance between the eigenfunctions $u_j^\ep$ and $u_j$.

\begin{lemma}\label{l.H1error.function.no.proj}
Let $\lambda$ be an eigenvalue to \eqref{eq.Neumann} that admits $m\ge 1$ eigenfunctions $u_{k},\dots,u_{k+m-1}$ satisfying \eqref{eq.Neumann.normalize}, then there is a family of $m\times m$ orthogonal matrices $\Oa_\ep$ such that $\Oa_\ep \rta \Oa_0$ as $\ep\rta0$, and Steklov eigenfunctions (harmonically extended) $u_k^\ep,\dots,u_{k+m-1}^\ep$ satisfying \eqref{eq.Steklov.normalize}, such that for $\ep < \ep_0(M,d,g,S,\lambda_k)$
    \be
\label{eq.Optestimate.function.no.proj}
\max_{k\le j \le k+m-1}\norm{u_j - (\Oa_\ep)_{jl}u_l^\ep}_{L^2(\pt\Omega)} + \norm{u_j - (\Oa_\ep)_{jl} u_l^\ep}_{H^1(M)} \le C \sqrt{\omega_d(\ep)},
    \ee
    where we have used the Einstein summation convention.
\end{lemma}

\begin{lemma}
    \label{l.H1error.eigenfunction}
    We fix $k\ge 1$ and denote $\lambda=\lambda_k$. We also denote $\Tilde{u}_k$ as the $L_\beta^2(M)$-projection of $u_k^\ep$ in the corresponding eigenspace of $\lambda$. Then there is a constant $C=C(M,d,g,S,\lambda)>0$ such that for $\ep<\ep_0(M,d,g,S,\lambda)$
    \be
\label{eq.H1error.eigenfunction}
\norm{u_k^\ep - \Tilde{u}_k}_{H^1(M)} + \norm{u_k^\ep - \Tilde{u}_k}_{L^2(\pt\Omega)}\le C \sqrt{\omega_d(\ep)}.
    \ee
\end{lemma}

\begin{proof}

Define the $\lambda$-eigenfunction
\[
u_k^* : = \sum_{\lambda_j=\lambda} \lb\Oa_\ep^{-1}\rb_{kj} u_j.
\]
Since $\Oa_\ep$ is orthogonal, by Lemma \ref{l.H1error.function.no.proj}, we immediately obtain
\[
\norm{u_k^\ep - u_k^*}_{H^1(M)} + \norm{u_k^\ep - u_k^*}_{L^2(\pt\Omega)}\le C \sqrt{\omega_d(\ep)}.
\]
Because by definition $\Tilde{u}_k$ is the $L_\beta^2(M)$-projection of $u_k^\ep$ in the eigenspace of $\lambda$, we obtain 
\[
\norm{u_k^\ep - \Tilde{u}_k}_{H^1(M)} \le \norm{u_k^\ep - u_k^*}_{H^1(M)} \le C \sqrt{\omega_d(\ep)}.
\]
On the other hand, we have 
\[
\norm{u_k^\ep - \Tilde{u}_k}_{L^2(\pt\Omega)} \le \norm{u_k^\ep - u_k^*}_{L^2(\pt\Omega)} + \norm{\Tilde{u}_k - u_k^*}_{L^2(\pt\Omega)} \le  \norm{\Tilde{u}_k - u_k^*}_{L^2(\pt\Omega)} + C \sqrt{\omega_d(\ep)}.
\]
We claim that $\norm{\Tilde{u}_k - u_k^*}_{L^2(\pt\Omega)} \le C \sqrt{\omega_d(\ep)}$. To that end, we observe that 
\[
\begin{split}
    \int_{\pt\Omega} \lb\Tilde{u}_k - u_k^*\rb^2 \dA_\ep &= \int_{M} \beta \lb\Tilde{u}_k - u_k^*\rb^2 \dV - \int_{\pt\Omega} \lb\Tilde{u}_k - u_k^*\rb^2 \dd \mu_\ep\\
    &\le C \norm{\Tilde{u}_k - u_k^*}_{L^2(M)}^2 + C \norm{\Phi_\ep}_{L^2(M)} \\
    &\le C \norm{\Tilde{u}_k -u_k^\ep}_{L^2(M)}^2+ C \norm{u_k^* -u_k^\ep}_{L^2(M)}^2 + C \norm{\Phi_\ep}_{L^2(M)} \\
    &\le C \omega_d(\ep),
\end{split}
\]
which proves the claim and hence also the lemma.

\end{proof}

\begin{proof}[Proof of Lemma \ref{l.H1error.function.no.proj}]
    We denote for all $i,j\ge 0$ 
\be
\label{eq.boundaryL2.projection.coefficient}
c_{ij}^\ep : = \int_{\pt\Omega} u_i u_j^\ep \dA_\ep.
\ee
Note that $|c_{ij}^\ep| \le \norm{u_i}_{L^2(\pt\Omega)} \le C(M,d,g,S,\lambda_i)$. We define for each $i \ge 1$ such that $\lambda_i = \lambda$
\be
\label{eq.boundaryL2.projection}
\hat{U}_i : = \sum_{\lambda_j\le \lambda} c_{ij}^\ep u_j^\ep  \textup{ and } U_i^*:= \sum_{\lambda_j=\lambda} c_{ij}^\ep u_j^\ep.
\ee
Note that for each $j $ such that $\lambda_j<\lambda$, by Theorem \ref{t.qualitativeconverge}, there is a fixed lower gap for small $\ep$
\[
|\lambda - \sigma_j^\ep|= |\lambda_k - \sigma_j^\ep|\ge \frac{1}{2}|\lambda_k -\lambda_j|>0.
\]
Therefore, by Remark \ref{r.sizeofprojectionforDifferentOrdereigenv} there is for all $\lambda_i =\lambda$
\be
\label{eq.sizeofProjection}
\begin{split}
    \norm{\hat{U}_i-U_i^*}_{L^2(\pt\Omega)} + \norm{\hat{U}_i-U_i^*}_{H^1(M)} &\le C\sum_{\lambda_j<\lambda} \lw\int_{\pt\Omega} u_i  u_j ^\ep\dA_\ep\rw \\
    &\le C(M,d,g,S,\lambda) ~\omega_d(\ep).
\end{split}
\ee
Therefore
\[
\norm{u_i-U_i^*}_{L^2(\pt\Omega)} \le \norm{\hat{U}_i-U_i^*}_{L^2(\pt\Omega)} + \norm{\hat{U}_i-u_i}_{L^2(\pt\Omega)}  \le \norm{\hat{U}_i-u_i}_{L^2(\pt\Omega)} + C\omega_d(\ep).
\]
This leads us to apply \eqref{eq.eqforUep.eigen.weakform} and \eqref{eq.eqforU.eigen.weakform} by taking $\eta = u_i -  \hat{U}_i$ for the same $i$ as above
\be
\label{eq.alongequationfordifferen.eigen}
\begin{split}
    \int_M |\gd \lb u_i -  \hat{U}_i\rb|^2  \dV& = \lambda \int_M \beta u_i \lb u_i -  \hat{U}_i\rb \dV -\sum_{\lambda_j\le \lambda} c_{ij}^\ep \sigma_j^\ep \int_{\pt\Omega} u_j^\ep \lb u_i -  \hat{U}_i\rb \dA_\ep\\
    &\qquad- \int_{M\setminus\Omega} \gd \hat{U}_i \cdot \gd \lb u_i -  \hat{U}_i\rb\dV\\
    &=\lambda\int_{\pt\Omega} (u_i -\hat{U}_i)^2 \dA_\ep + \lambda\int_M u_i(u_i-\hat{U}_i) \dd\mu_\ep -\sum_{\lambda_j<\lambda} c_{ij}^\ep \sigma_j^\ep \int_{\pt\Omega} u_j^\ep (u_i-\hat{U}_i)\dA_\ep \\
    &\qquad- \sum_{\lambda_j=\lambda}c_{ij}^\ep(\sigma_j^\ep-\lambda)\int_{\pt\Omega}u_j^\ep (u_i-\hat{U}_i)\dA_\ep + \int_{M\setminus\Omega} |\gd \lb u_i -  \hat{U}_i\rb|^2  \dV\\
    &\qquad- \int_{M\setminus\Omega} \gd u_i\cdot \gd \lb u_i -  \hat{U}_i\rb \dV.
\end{split}
\ee
This shows that
\be\label{eq.alongequationfordifferen.eigen.2}
\begin{split}
    \int_{\Omega}|\gd \lb u_i -  \hat{U}_i\rb|^2  \dV -\lambda\int_{\pt\Omega} (u_i -\hat{U}_i)^2 \dA_\ep &=  \lambda\int_M u_i(u_i-\hat{U}_i) \dd\mu_\ep -\sum_{\lambda_j<\lambda} c_{ij}^\ep \sigma_j^\ep \int_{\pt\Omega} u_j^\ep (u_i-\hat{U}_i)\dA_\ep\\
    &\qquad- \sum_{\lambda_j=\lambda}c_{ij}^\ep(\sigma_j^\ep-\lambda)\int_{\pt\Omega}u_j^\ep (u_i-\hat{U}_i)\dA_\ep \\
    &\qquad- \int_{M\setminus\Omega} \gd u_i\cdot \gd \lb u_i -  \hat{U}_i\rb \dV\\
&=: A_1 + A_2+A_3+A_4.
\end{split}
\ee
By Theorem \ref{t.qualitativeconverge}, we know that, if we denote $k^*$ as the smallest index such that the Laplace--Beltrami eigenvalue $\lambda_{k^*}$ is bigger than $\lambda$, there is for all small $\ep$
\[
\sigma_{k^*}^\ep > \frac{\lambda+\lambda_{k^*}}{2}>\lambda.
\]
This shows that the left hand side of \eqref{eq.alongequationfordifferen.eigen.2} becomes
\be\label{eq.alongequationfordifferen.eigen.3}
(\sigma_{k^*}^\ep -\lambda)\int_{\pt\Omega} (u_i -\hat{U}_i)^2 \dA_\ep \le \int_{\Omega}|\gd \lb u_i -  \hat{U}_i\rb|^2  \dV -\lambda\int_{\pt\Omega} (u_i -\hat{U}_i)^2 \dA_\ep,
\ee
where $\sigma_{k^*}^\ep -\lambda > \frac{\lambda_{k^*}-\lambda}{2}>0$ is uniformly lower bounded, and hence we only have to bound $A_i$ with $ i=1,\dots,4$ in \eqref{eq.alongequationfordifferen.eigen.2}.
Now, observe that
\[
|A_1| \le C \lw\int_M \gd \Phi_\ep \cdot \gd (u_i \hat{U}_i) \dV\rw + C \lw\int_M \gd \Phi_\ep \cdot\gd (u_i^2) \dV\rw \le C\omega_d(\ep)
\]
by a similar argument as in \eqref{eq.boundJ1.convergerate}. By \eqref{eq.sizeofProjection}, we have
\[
|A_2|\le C \sum_{\lambda_j <\lambda} |c_{ij}^\ep| \le C \omega_d(\ep).
\]
By Lemma \ref{l.Optestimate.stektoNeum}, we obtain
\[
|A_3|\le C \omega_d(\ep).
\]
Last but not least, we have for any small $\delta>0$
\[
\begin{split}
    |A_4|& \le C_\delta |M\setminus\Omega| \norm{u_i}_{C^1(M)} + \delta \int_{M\setminus\Omega}|\gd \lb u_i -  \hat{U}_i\rb|^2  \dV\\
&\le \delta \int_{M\setminus\Omega}|\gd \lb u_i -  \hat{U}_i\rb|^2  \dV + C \omega_d(\ep).
\end{split}
\]
If we denote $\hat{u}_i$ as the harmonic extension of $\restr{u_i}{\Omega}$ to the whole manifold, we obtain
\be\label{eq.outsideOmegaGradientdifference}
\begin{split}
    \int_{M\setminus \Omega} |\gd \lb u_i -  \hat{U}_i\rb|^2 \dV &= \int_{M\setminus \Omega} |\gd \lb \hat{u}_i -  \hat{U}_i\rb|^2 \dV + \int_{M\setminus \Omega}|\gd \lb \hat{u}_i - u_i\rb|^2 \dV\\
&\le C \int_{\Omega}|\gd \lb {u}_i -  \hat{U}_i\rb|^2  \dV + C \ep^{\frac{d}{d-1}},
\end{split}
\ee
where in the last line we have used the estimate \cite{girouard_large_2021}*{Lemma 3.4}. By combining \eqref{eq.alongequationfordifferen.eigen.3} and the above estimates about $A_1,\dots,A_4$, we get
\[
\int_{\pt\Omega} (u_i -\hat{U}_i)^2 \dA_\ep \le C \omega_d(\ep)
\]
for $C=C(M,d,g,S,\lambda)$. Plugging this inequality into \eqref{eq.alongequationfordifferen.eigen.2} and combining \eqref{eq.outsideOmegaGradientdifference}, we obtain
\[
\int_{M}|\gd \lb {u}_i -  \hat{U}_i\rb|^2  \dV \le C \omega_d(\ep).
\]
Also note that for all $l\ge 1$
\[
\begin{split}
  \lw \int_M \beta u_l^\ep \dV \rw&= \lw\int_M  u_l^\ep (\beta\dV - \restr{\dA}{\pt\Omega})\rw\\
 &=\lw \int_M  \gd \Phi_\ep \cdot \gd u_l^\ep\dV\rw  \\
 &\le C \omega_d(\ep)
\end{split}
\]
by the same argument as in \eqref{eq.boundJ1.convergerate} with $u$ chosen as the trivial Laplace--Beltrami eigenfunction, i.e., the constant. This shows that by the Poincar\'{e} inequality
\[
\int_{M}|  {u}_i -  \hat{U}_i|^2  \dV\le C\int_{M}|\gd \lb {u}_i -  \hat{U}_i\rb|^2  \dV + C\lw\int_M \beta \hat{U}_i \dV\rw^2\le C \omega_d(\ep).
\]
By combining \eqref{eq.sizeofProjection} and the previous displays we obtain
\be
\label{eq.orthogonal.UUep}
\max_{\lambda_i=\lambda}\norm{u_i - U_i^*}_{L^2(\pt\Omega)}+\norm{u_i-U_i^*}_{H^1(M)} \le C \sqrt{\omega_d(\ep)}.
\ee
In particular, for $l,k$ such that $\lambda_l=\lambda_k=\lambda$, we have
\[
\begin{split}
    \sum_{\lambda_j=\lambda} c_{lj}^\ep c_{kj}^\ep &= \int_{\pt\Omega} U_l^* U_k^* \dA_\ep\\
    &=\int_{\pt\Omega} u_l u_k \dA_\ep + \int_{\pt \Omega} u_l (U_k^* - u_k) \dA_\ep\\
    &=\delta_{lk} - \int_{M} u_l u_k \dd\mu_\ep + \int_{\pt \Omega} u_l (U_k^* - u_k) \dA_\ep\\
    &=\delta_{lk} + O(\sqrt{\omega_d(\ep)}).
\end{split}
\]
This shows that the matrix $M_\ep=(c_{lk}^\ep)$ with $l,k$ satisfying $\lambda=\lambda_l=\lambda_k$ is $O(\sqrt{\omega_d(\ep)})$ distance away from an orthogonal matrix $Q_\ep$. Therefore, \eqref{eq.orthogonal.UUep} becomes
\be
\label{eq.orthogonal.UUep.2}
\max_{\lambda_i=\lambda}\norm{u_i - (Q_\ep)_{ik} u_k^\ep}_{L^2(\pt\Omega)}+\norm{u_i-(Q_\ep)_{ik} u_k^\ep}_{H^1(M)} \le C \sqrt{\omega_d(\ep)},
\ee
which proves the lemma.

\end{proof}

\section{Expansion of the Steklov eigenvalues}
\label{s.expansion}
We fix $\lambda =\lambda_k$ as the Laplace--Beltrami eigenvalue to \eqref{eq.Neumann.eigenpairs} such that 
\[
\lambda_{k-1}<\lambda=\lambda_{k}=\cdots = \lambda_{k+m-1}<\lambda_{k+m}.
\]
Correspondingly we denote $\sigma_j^\ep$ for $j=k,\dots, k+m-1$ as the Steklov eigenvalues in \eqref{eq.Steklov.eigenpairs}. In this section we investigate the expansion of these Steklov eigenvalues in $\ep$. 
As in \eqref{eq.normalize.Steklov.and.Neumann}, we assume the Laplace--Beltrami eigenfunctions $u_j$ and Steklov eigenfunctions $u_j^\ep$ satisfy
\[
\int_{\pt\Omega} v_l^\ep v_k^\ep \dA =\int_{\pt\Omega} u_l^\ep u_k^\ep \dA = \int_M \beta u_l u_k \dV = \delta_{lk}.
\]
We denote $\Tilde{u}_j^\ep$ as the $L_\beta^2(M)$-projections of $u_j^\ep$ in the eigenspace spanned by $u_j$ for $j=k,\dots,k+m-1$, and then define for $j=k,\dots, k+m-1$
\[
U_{\ep,j}:=\frac{u_j^\ep}{\sqrt{\int_M \beta |\Tilde{u}_j^\ep|^2 \dV}}\textup{ and }U_j:=\frac{\Tilde{u}_j^\ep}{\sqrt{\int_M \beta |\Tilde{u}_j^\ep|^2 \dV}}.
\]
By Theorem \ref{t.quantitative.StekNeum}, especially \eqref{eq.Optestimate.function.proj}, there is 
\be\label{eq.distance.eigenfunction.remarkversion}
\max_{\lambda_j = \lambda}~\norm{U_j-U_{\ep,j}}_{L^2(\pt\Omega)}+\norm{U_j-U_{\ep,j}}_{H^1(M)}\le C \sqrt{\omega_d(\ep)}.
\ee
We also have
\be\label{eq.eigenfunction.Unormalize}
\int_M \beta U_{\ep,j} U_j \dV = \int_M \beta| U_j|^2 \dV = 1 \textup{ and }\int_M \beta U_i U_j \dV = \delta_{ij} + O(\sqrt{\omega_d(\ep)}).
\ee
In the following for $i$ such that $\lambda_i=\lambda$ we seek an expansion of the following form
\be
\label{eq.expansionofEigenpairs}
\sigma_{i}^\ep = \lambda + \delta_{\ep,i} + \rho_{\ep,i} \textup{ and }U_{\ep,i} = U_i + Z_{\ep,i} + E_{\ep,i}\,,
\ee
defined as follows. For $i=k,\dots,k+m-1$, define
\[
\delta_{\ep,i}^j
:=
\lambda\int_M U_i u_j\dd\mu_\ep,
\qquad
j=k,\dots,k+m-1.
\]
We define $Z_{\ep,i}$ as the unique solution satisfying
\begin{equation}
\label{eq.defineZep.i}
\begin{aligned}
(\Delta+\lambda\beta)Z_{\ep,i}
&=
\lambda U_i\dd\mu_\ep
-
\sum_{j=k}^{k+m-1}
\delta_{\ep,i}^j\beta u_j\dV,
\\
\int_MZ_{\ep,i}u_\ell\beta\dV&=0,
\qquad
\ell=k,\dots,k+m-1.
\end{aligned}
\end{equation}
Note that because $u_k, \dots , u_{k+m-1}$ form an orthonormal basis for the eigenspace of $\lambda$, there is for $i,j\in \{k,\dots,k+m-1\}$
\be\label{eq.define.delta.i.ep.j}
\delta_{\ep,i}^j : = \lambda \int_M  U_i u_j \dd\mu_\ep \textup{ and }\delta_{\ep,i} =\sum_{j=k}^{k+m-1} \delta_{\ep,i}^j \int_M \beta U_i u_j \dV.
\ee
Note that for $\lambda_i=\lambda_j=\lambda$ we have the following bound 
\be
\label{eq.bound.delta.i.ep.j}
|\delta_{\ep,i}^j | =\lambda \lw \int_M \Phi_\ep \Delta(U_i u_j) \dV \rw \le C \norm{\Phi_\ep}_{L^1(M)} \le C \ep^2,
\ee
where the constant $C$ depends at most on $M,d,g,S$ and $\lambda$.

\begin{lemma}
    \label{l.expansion}
    Let $\sigma_i^\ep$ and $\lambda=\lambda_i$ be as above, then there is for $\ep<\ep_0(M,d,g,S,\lambda)$
    \be\label{eq.expansion.sigma}
\sigma_i^\ep = \lambda ~+~\underbrace{\lambda \int_M Z_{\ep,i} U_i \dd\mu_\ep}_{\omega_d(\ep)} ~-~\underbrace{\frac{d}{d-1}\int_{M\setminus\Omega} |\gd U_i|^2 \dV}_{\ep^{d/(d-1)}} ~+~ \underbrace{\lambda\int_M U_i^2 \dd\mu_\ep}_{\ep^2} ~+~O(\omega_d(\ep)^{3/2}).
\ee
\end{lemma}

\begin{lemma}
    \label{l.H1.estimate.Zep.i}
    For some $C=C(M,d,g,S,\lambda)>0$ and all $\ep < \ep_0(M,d,g,S,\lambda)$, we have 
    \be
\label{eq.H1.estimate.Zep.i}
\max_{i=k,\dots, k+m-1}~\norm{Z_{\ep,i}}_{H^1(M)} \le C \sqrt{\omega_d(\ep)}.
    \ee
\end{lemma}

\begin{proof}
Let $\{u_j\}_{j\geq0}$ be an orthonormal basis of
$L^2(M,\beta\dV)$ consisting of weighted Laplace--Beltrami
eigenfunctions. Since $Z_{\ep,i}$ is orthogonal to the
$\lambda$-eigenspace, it has the spectral representation
\[
Z_{\ep,i}
=
\sum_{\lambda_j\neq\lambda}
\frac{\lambda}{\lambda-\lambda_j}
\left(
\int_Mu_jU_i\dd\mu_\ep
\right)u_j.
\]
Note that by a similar argument to \eqref{eq.boundJ1.convergerate}, we know that for all $\lambda_j \le \lambda$ (or equivalently $j\le k+m-1$)
\be\label{eq.boundforlowfrequency.Zep}
\lw\int_M  u_j U_i \dd\mu_\ep\rw = \lw\int_M \Phi_\ep \Delta\lb u_j U_i \rb \dV\rw \le C \norm{\Phi_\ep}_{L^1(M)} \le C \ep^2,
\ee
by applying Theorem \ref{t.summaryofPhi} and Corollary \ref{cor.phibar.model.norms}. 
By testing $\displaystyle\Tilde{Z}_{\ep,i}:=\sum_{\lambda_j >\lambda} \lmb\int_M \beta Z_{\ep,i} u_j \dV \rmb u_j$ with \eqref{eq.defineZep.i}, there is by the Cauchy-Schwarz inequality and the Poincar\'{e} inequality on $(M,g)$, for all $\delta>0$
\be\label{eq.intermediate.estimate.Zep.i}
\int_M |\gd \Tilde{Z}_{\ep,i}|^2 - \lambda\beta|\Tilde{Z}_{\ep,i}|^2 \dV \le C \norm{\gd\Phi_\ep}_{L^2(M)} \norm{\Tilde{Z}_{\ep,i}}_{H^1(M)} \le C_\delta \norm{\gd\Phi_\ep}_{L^2(M)}^2 + \delta \int_M |\gd \Tilde{Z}_{\ep,i}|^2 \dV.
\ee
This shows that by choosing $\delta>0$ sufficiently small for some constant $C>0$
\[
\int_M |\Tilde{Z}_{\ep,i}|^2\dV \le C{\omega_d(\ep)},
\]
due to the fact that the lowest eigenvalue appearing in $\Tilde{Z}_{\ep,i}$ is strictly bounded away from below by $\lambda$. Combining \eqref{eq.boundforlowfrequency.Zep} we have
\be
\label{eq.L2boundforZep.i}
\norm{Z_{\ep,i}}_{L^2(M)} \le C \sqrt{\omega_d(\ep)}.
\ee
By applying $Z_{\ep,i}$ to \eqref{eq.defineZep.i} and a similar argument to \eqref{eq.intermediate.estimate.Zep.i} we obtain
\be
\label{eq.H1boundforZep.i}
\norm{\gd Z_{\ep,i}}_{L^2(M)} \le C \sqrt{\omega_d(\ep)}.
\ee
Note that the constant $C$ in this proof depends at most on $M,d,g,S$ and $\lambda$ whenever $\ep$ is sufficiently small.
\end{proof}

\begin{lemma}
    \label{l.boundEU.integral}
    Let $E_{\ep,i}$ and $U_{i}$ be defined as in \eqref{eq.expansionofEigenpairs}. Then there is a constant $C=C(M,d,g,S,\lambda)>0$ such that for all $d\ge 2$ and   $\ep<\ep_0(M,d,g,S,\lambda)$
    \be
\label{eq.boundEU.integral}
\lw\lambda\int_ME_{\ep,i} U_i \dd\mu_\ep\rw \le C \omega_d(\ep)^{3/2}.
    \ee
\end{lemma}

\begin{proof}
    We simply write $U_{\ep}=U_{\ep,i}$, $E_\ep = E_{\ep,i}$, $Z_{\ep,i}=Z_\ep$, $U = U_i$ and $\sigma_\ep = \sigma_i^\ep$. To bound the term $L:= \lambda\int_M E_\ep U \dd\mu_\ep$, we observe that 
\be\label{eq.boundL.expansion}
\begin{split}
  L & = \lambda\int_M \Phi_\ep  \Delta (E_\ep U) \dV\\
   &=\lambda\int_M \Phi_\ep  (\Delta E_\ep U + 2\gd E_\ep \cdot \gd U + E_\ep \Delta U) \dV\\
   &=\lambda\int_M \Phi_\ep  \Delta (E_\ep U) \dV\\
   &=\lambda\int_M \Phi_\ep(E_\ep \Delta U) \dV+ 2\lambda\int_M \Phi_\ep(\gd E_\ep \cdot \gd U)\dV +\lambda\int_M \Phi_\ep  (\Delta E_\ep U) \dV \\
   &=: L_1 + L_2 + L_3.
\end{split}
\ee
We first observe that by Theorem \ref{t.summaryofPhi}, Lemma \ref{l.H1error.eigenfunction} and Lemma \ref{l.H1.estimate.Zep.i}
\[
|L_1| \le C \norm{\Phi_\ep}_{L^2(M)}\norm{E_\ep}_{L^2(M)}\le C\ep^{\frac{d}{d-1}}\lmb \norm{U-U_\ep}_{L^2(M)}+\norm{Z_{\ep}}_{L^2(M)}\rmb\le C \ep^{\frac{d}{d-1}}\omega_d(\ep)^{1/2}.
\]
By a similar argument
\[
|L_2|\le C\ep^{\frac{d}{d-1}} \lmb \norm{U-U_\ep}_{H^1(M)}+\norm{Z_{\ep,i}}_{H^1(M)}\rmb\le  C \ep^{\frac{d}{d-1}}\omega_d(\ep)^{1/2}.
\]
To bound $L_3$ we note that
\be\label{eq.boundL3.notdone}
\begin{split}
    L_3&=\int_M  U\Phi_\ep \Delta(U_\ep-Z_{\ep}-U) \dV\\
    &=\int_M \Phi_\ep\lmb -U(\sigma_\ep U_\ep +\Lambda_\ep(U_\ep))\dA_\ep -\lambda U^2\dd\mu_\ep + \sum_{\lambda_j=\lambda}\delta_{\ep,i}^j \beta U u_j \dV +\lambda\beta U^2 \dV\rmb\\
    &=(\lambda-\sigma_\ep)\int_M \Phi_\ep U U_\ep \dA_\ep  + \lambda \int_M \Phi_\ep U(U-U_\ep) \dA_\ep\\
    &\qquad-\int_M \Phi_\ep U \Lambda_\ep(U_\ep) \dA_\ep +\sum_{\lambda_j=\lambda}\delta_{\ep,i}^j\int_M \Phi_\ep \beta U u_j \dV\\
    &=: L_{31}+L_{32} +L_{33}+L_{34}.
\end{split}
\ee
By Lemma \ref{l.Optestimate.stektoNeum} and a similar argument as in \eqref{eq.boundJ11.convergerate}
\[
|L_{31}|\le C\omega_d(\ep) \lb \norm{\bar{\Phi}_\ep}_{L^\infty(M)} + \norm{\Ra_\ep}_{L^2(\pt \Omega)} \rb \norm{U}_{L^\infty(M)} \norm{U_\ep}_{L^2(\pt\Omega)}\le C\omega_d(\ep)^2=o(\ep^2).
\]
We also have by Lemma \ref{l.H1error.eigenfunction} and a similar argument as above 
\[
|L_{32}| \le C\lmb\norm{\bar{\Phi}_\ep}_{L^\infty(M)}+  \norm{\Ra_\ep}_{L^2(\pt\Omega)}\rmb \norm{U-U_\ep}_{L^2(\pt\Omega)}\le {C\omega_d(\ep)^{3/2}}.
\]
To bound $L_{33}$, we first write
\[
L_{33}= -\int_M \Phi_\ep U \Lambda_\ep(U_\ep) \dA_\ep = \int_M \Phi_\ep U \pt_\nu^-(U_\ep-U) \dA_\ep +\int_M \Phi_\ep U \pt_\nu^-(U) \dA_\ep=: L_{331} + L_{332}.
\]
Then we decompose
\[
\Phi_\ep = \La_\ep + (\bar{\Phi}_\ep-\La_\ep) + \Ra_\ep
\]
as we did in \eqref{eq.boundJ12.convergerate}, and obtain by the discussions in \eqref{eq.boundJ12.convergerate.done} with $u^\ep, u$ being replaced by $U_\ep, U$ respectively. Specifically, we can write
\be
\label{eq.boundL331.convergerate}
\begin{split}
 L_{331} &= \int_M \Phi_\ep U \pt_\nu^-(U_\ep-U) \dA_\ep\\
 &=\lambda_k\int_{\pt\Omega}  \La_\ep  U \pt_\nu^-(U_\ep-U) \dA  + \lambda_k\int_{\pt\Omega}  (\bar{\Phi}_\ep - \La_\ep) U \pt_\nu^-(U_\ep-U) \dA  + \lambda_k \int_{\pt\Omega} \Ra_\ep  U \pt_\nu^-(U_\ep-U)\dA\\
  &=\lambda_k \sum_{p\in S_\ep}\lb\La_\ep(p) + \frac{\beta(p)}{2d}r_{\ep,p}^2\rb  \int_{\pt B_p}  U \pt_\nu^-(U_\ep-U) \dA - \lambda_k\int_{M\setminus\Omega} \gd \lmb (\bar{\Phi}_\ep - \La_\ep) U \rmb\cdot \gd (U_\ep-U) \dV \\
  &\qquad -\lambda_k \int_{M\setminus\Omega} \gd (\Ra_\ep U) \cdot \gd (U_\ep - U)\dV - \lambda_k^2 \int_{M\setminus\Omega} (\Ra_\ep + \bar{\Phi}_\ep - \La_\ep)U^2 \dV\\
  &=:L_{3311}+L_{3312}+L_{3313}+L_{3314}.
\end{split}
\ee
Note that by a similar argument between the displays \eqref{eq.boundJ12.convergerate} and \eqref{eq.boundJ12.convergerate.done}, we know that 
\[
|L_{3311}+L_{3312}+L_{3313}|\le C \ep^{1+\frac{d}{2(d-1)}}\norm{U_\ep - U}_{H^1(M)} \le C \omega_d(\ep)^{3/2}.
\]
We also have by Lemma \ref{l.difference.modelandPhibar} and Lemma \ref{l.sharperL2estimateforRaep}
\[
|L_{3314}|\le C \lmb\norm{\Ra_\ep}_{L^2(M)}|M\setminus\Omega|^{1/2} + \norm{\bar{\Phi}_\ep-\La_\ep}_{L^\infty(M)}|M\setminus
\Omega|\rmb=o(\ep^2).
\]
Combining these estimates we have
\[
|L_{331}| \le  C \omega_d(\ep)^{3/2}.
\]
On the other hand, observe that
\[
L_{332} = \int_M \La_\ep U \pt_\nu^-(U) \dA_\ep + \int_M (\bar{\Phi}_\ep-\La_\ep) U  \pt_\nu^-(U) \dA_\ep +\int_M \Ra_\ep U  \pt_\nu^-(U) \dA_\ep=: L_{3321}+L_{3322}+L_{3323}
\]
Note that $\La_\ep$ is constant on each $\pt B_p$ and hence
\[
|L_{3321}|=\frac{1}{2}\lw \sum_{p\in S_\ep} \lb\La_\ep(p) + \frac{\beta(p)}{2d}r_{\ep,p}^2\rb\int_{\pt B_p}  \pt_\nu^-(U^2) \dA\rw \le C\omega_d(\ep)|M\setminus\Omega|=o(\ep^2).
\]
By Lemma \ref{l.difference.modelandPhibar}, we have
\[
|L_{3322}|\le C\lw\int_{M\setminus\Omega} (\bar{\Phi}_\ep-\La_\ep) \Delta(U^2)\dV +\int_{M\setminus\Omega} \gd(\bar{\Phi}_\ep-\La_\ep) \cdot\gd(U^2)\dV \rw\le C \ep^2|M\setminus\Omega|=o(\ep^2).
\]
By applying \eqref{eq.Dirichletdecay} and the bound $\norm{\Ra_\ep}_{L^\infty(M)}\le C\ep$ in Theorem \ref{t.summaryofPhi}
\[
\begin{split}
    |L_{3323}|&\le C\lw \int_{M\setminus\Omega} \Ra_\ep \Delta(U^2)\dV +\int_{M\setminus\Omega} \gd\Ra_\ep \cdot\gd(U^2)\dV \rw\\[5pt]
    &\le C(\ep|M\setminus\Omega|+C\ep^{\frac{d}{2(d-1)}+1}|M\setminus\Omega|^{1/2}) \\[5pt]
    &\le C \omega_d(\ep)^{3/2}.
\end{split}
\]
Combining all the estimates above on $L_{331}$, $L_{332}$ and in particular $L_{332i}, \ i= 1,2,3$, we obtain 
\be
\label{eq.boundL33}
|L_{33}|\le C \omega_d(\ep)^{3/2}.
\ee

The term $L_{34}$ in \eqref{eq.boundL3.notdone} is bounded by
\[
|L_{34}|\le C\max_{\lambda_j =\lambda} \delta_{\ep,i}^j ~ \norm{\Phi_\ep}_{L^1(M)} \le C \ep^4,
\]
where we have used \eqref{eq.bound.delta.i.ep.j}.

Combining all the above estimates we obtain an estimate for $L$ as decomposed in \eqref{eq.boundL.expansion}
\be
\label{eq.estimateL}
|L|\le |L_1|+|L_2|+|L_3| \le C \omega_d(\ep)^{3/2}.
\ee
This finishes the proof.

\end{proof}

\begin{lemma}
    \label{l.DtN.fineestimate}
    Let $U_i$ and $U_{\ep,i}$ be defined as in \eqref{eq.expansionofEigenpairs}. Then there is a constant $C=C(M,d,g,S,\lambda)>0$ such that for all $d\ge 2$ and $\ep<\ep_0(M,d,g,S,\lambda)$
    \be
\label{eq.DtN.fineestimate}
\lw\int_{M\setminus \Omega} \gd U_i \cdot \gd U_{\ep,i} \dV - \frac{d}{d-1}\int_{M\setminus\Omega} |\gd U_i|^2 \dV \rw \le C \omega_d(\ep)^{3/2}.
    \ee
\end{lemma}

\begin{proof}

Similar to the proof in Lemma \ref{l.boundEU.integral}, we simply write $U_\ep = U_{\ep,i}$ and $U=U_i$. 

By the same arguments between the displays \eqref{eq.flatexpansion.J2p} and \eqref{eq.decompose.J2p2} with $u, u^\ep$ being replaced by $U,U_\ep$, we have
\be
\label{eq.flatwrite.UiUepi}
\lw\int_{M\setminus \Omega} \gd U\cdot \gd U_{\ep} \dV -\sum_{p\in S_\ep} T_0(p) \rw \le C \omega_d(\ep)^{3/2},
\ee
where for each fixed $p\in S_\ep$
\be
\label{eq.defineT0.secondary}
T_0(p) = \int_{\pt B_p} U_\ep \Tilde{w}_\xi \dA,
\ee
where $\Tilde{w}_\xi$ is defined in \eqref{eq.define.w.xi.tilde} with $\xi = \gd U (p)$.

By \eqref{eq.decompose.DtN.2}, we have
\be
\label{eq.decompose.DtN.U}
\begin{split}
    T_0 &= \lb \frac{A_\ep}{r_{\ep,p}} -\sigma_i^\ep \rb^{-1} \lmb \int_{\pt Q_p} \pt_\nu^{-} \Tilde{w}_\xi U_\ep \dA + \int_{Q_p\setminus B_p} \Delta \Tilde{w}_\xi U_\ep \dV \rmb\\[5pt]
&= \lb \frac{A_\ep}{r_{\ep,p}} -\sigma_i^\ep \rb^{-1} \lmb \int_{\pt Q_p} \pt_\nu^{-} \Tilde{w}_\xi( U_\ep-U) \dA + \int_{Q_p\setminus B_p} \Delta \Tilde{w}_\xi (U_\ep -U)\dV \rmb\\
&\qquad +  \lb \frac{A_\ep}{r_{\ep,p}} -\sigma_i^\ep \rb^{-1} \lmb \int_{\pt Q_p} \pt_\nu^{-} \Tilde{w}_\xi U \dA + \int_{Q_p\setminus B_p} \Delta \Tilde{w}_\xi U \dV \rmb\\[5pt]
&=: T_{01} + T_{02}.
\end{split}
\ee

By the same arguments between \eqref{eq.decompose.DtN.2} and \eqref{eq.boundT0}, we have by also applying Lemma \ref{l.H1error.eigenfunction}
\be
\label{eq.boundT01.U}
\sum_{p\in S_\ep}|T_{01}(p)| \le C \sum_{p\in S_\ep} \ep^{\frac{d}{d-1}} |Q_p|^{1/2} \norm{U_\ep -U}_{H^1(Q_p)} \le C \ep^{\frac{d}{d-1}} \norm{U_\ep-U}_{H^1(M)} \le C \omega_d(\ep)^{3/2}.
\ee

It suffices to compute $T_{02}$. By integration by parts, we have 
\be
\label{eq.T02.U}
\begin{split}
    T_{02} & = \lb \frac{A_\ep}{r_{\ep,p}} -\sigma_i^\ep \rb^{-1} \lmb \int_{\pt B_p} \pt_\nu^+ \Tilde{w}_\xi U \dA - \int_{\pt B_p} \Tilde{w}_\xi \pt_\nu^+ U \dA + \int_{Q_p \setminus B_p} \Tilde{w}_\xi \Delta U \dV\rmb\\
    &=\lb \frac{A_\ep}{r_{\ep,p}} -\sigma_i^\ep \rb^{-1} \lmb \int_{\pt B_p} \pt_\nu^+ \Tilde{w}_\xi U \dA- \int_{\pt B_p} \Tilde{w}_\xi \pt_\nu^+ U \dA\rmb + \lambda \lb \frac{A_\ep}{r_{\ep,p}} -\sigma_i^\ep \rb^{-1}\int_{Q_p \setminus B_p} \Tilde{w}_\xi \beta U \dV\\[5pt]
    &=: T_{021}+T_{022}.
\end{split}
\ee
By Lemma \ref{l.transferingbetween.geodesicchart}, if we use a geodesic coordinate $z\in \R^d \cong T_pM$ (and abuse the notation $f(z)=f(\exp_p(z))$ for any function $f$ in a neighborhood of $p\in M$), then 
\[
 \lw\int_{Q_p \setminus B_p} \Tilde{w}_\xi \beta U \dV - \int_{B_{R_\ep}(0)\setminus B_{r_{\ep,p}}(0)}  {w}_\xi(z) \beta(z) U(z) \dd z \rw \le C R_\ep^2 |Q_p|,
\]
which sums over $p\in S_\ep$ to a term of order $O(\ep^2)$. On the other hand, by the symmetry as shown in the definition of $w_\xi$ in \eqref{eq.define.w.xi.tilde}
\[
\begin{split}
 \lw \int_{B_{R_\ep}(0)\setminus B_{r_{\ep,p}}(0)}  {w}_\xi(z) \beta(z) U(z) \dd z\rw&=\lw\int_{B_{R_\ep}(0)\setminus B_{r_{\ep,p}}(0)}  {w}_\xi(z) (\beta(z) U(z)-\beta(0)U(0)) \dd z \rw\\
  &\le C r_{\ep,p}^{d-1} R_\ep^2,
\end{split}
\]
which sums up over $p\in S_\ep$ to a term of order $O(\ep^{2})$. Combining the above two displays we obtain
\be
\label{eq.boundT022.U}
\sum_{p\in S_\ep}|T_{022}(p)| \le C \ep^{2+d/(d-1)}\le C \omega_d(\ep)^{3/2}.
\ee

On the other hand, by \eqref{eq.boundary.Neumann.Diri.response.w.xi} we have 
\be \label{eq.boundT021.U}
\begin{split}
    T_{021} & = \lb \frac{A_\ep}{r_{\ep,p}} -\sigma_i^\ep \rb^{-1} \lmb \frac{A_\ep}{r_{\ep,p}} \int_{\pt B_p}  \Tilde{w}_\xi U \dA - \int_{\pt B_p} \Tilde{w}_\xi \pt_\nu^+ U \dA\rmb \\[5pt]
    &= \lb \frac{A_\ep}{r_{\ep,p}} -\sigma_i^\ep \rb^{-1} \frac{A_\ep}{r_{\ep,p}}\int_{\pt B_p}  \Tilde{w}_\xi U \dA - \lb \frac{A_\ep}{r_{\ep,p}} -\sigma_i^\ep \rb^{-1} \int_{\pt B_p} \Tilde{w}_\xi \pt_\nu^+ U \dA \\
    &= \int_{\pt B_p}  \Tilde{w}_\xi U \dA - \frac{r_{\ep,p}}{d-1} \int_{\pt B_p} \Tilde{w}_\xi \pt_\nu^+ U \dA \\
    &\qquad+ O\lb\ep^{d/(d-1)} \lb \lw\int_{\pt B_p}  \Tilde{w}_\xi U \dA\rw + \frac{r_{\ep,p}}{d-1}\lw \int_{\pt B_p} \Tilde{w}_\xi \pt_\nu^+ U \dA\rw \rb\rb\\
    &=: T_{0211} - T_{0212} + O\lb\ep^{d/(d-1)}\lb |T_{0211}| +|T_{0212}| \rb\rb.
\end{split}
\ee
We recall Lemma \ref{l.transferingbetween.geodesicchart} with a local geodesic coordinate $z\in \R^d \cong T_pM$ (abusing the notation $f(z)=f(\exp_p(z))$) and that $\xi = \gd_z U(0) (=\gd U(p))$
\be\label{eq.boundT0211.U}
\begin{split}
    T_{0211} &= \int_{\pt B_p}  \Tilde{w}_\xi U \dA \\
    &= \int_{\pt B_{r_{\ep,p}}(0)} \xi \cdot \frac{z}{|z|} ~ U(z) \dd \theta(z) + O(\ep^{2d/(d-1)}A(\pt B_p))\\
    &=\int_{B_{r_{\ep,p}}(0)} \xi \cdot \gd_z U \dd z  + O(\ep^{2d/(d-1)}A(\pt B_p))\\
    &=\int_{B_{r_{\ep,p}}(0)} |\gd_z U|^2 \dd z  + O(\ep^{d/(d-1)}|B_p|)\\
    &=\int_{B_p} |\gd U|^2 \dV + O(\ep^{d/(d-1)}|B_p|).
\end{split}
\ee
Similarly, we have
\be\label{eq.boundT0212.U}
\begin{split}
    T_{0212} & = \frac{r_{\ep,p}}{d-1} \int_{\pt B_p} \Tilde{w}_\xi \pt_\nu^+ U \dA\\
    &= \frac{1}{d-1} \int_{\pt B_{r_{\ep,p}}(0)} (\xi\cdot z)~ \pt_\nu^+ U \dd \theta(z) + O(\ep^{3d/(d-1)}A(\pt B_p))\\
    &=-\frac{1}{d-1} \int_{B_{r_{\ep,p}}(0)} \xi \cdot \gd_z U(z) \dd z - \frac{1}{d-1}\int_{B_{r_{\ep,p}}(0)} (\xi \cdot z) \Delta_z U(z)\dd z + O(\ep^{3d/(d-1)}A(\pt B_p))\\
    &=-\frac{1}{d-1} \int_{B_{r_{\ep,p}}(0)} \xi \cdot \gd_z U(z) \dd z + O(\ep^{d/(d-1)}|B_p|)\\
    &=-\frac{1}{d-1} \int_{B_p} |\gd U|^2 \dV + O(\ep^{d/(d-1)}|B_p|).
\end{split}
\ee
The proof of this lemma is now finished by combining the displays between \eqref{eq.flatwrite.UiUepi} and \eqref{eq.boundT0212.U}.

\end{proof}

\begin{proof}[Proof of Lemma \ref{l.expansion}]
We denote 
\[
W_{\ep,i} : = U_i + Z_{\ep,i},
\]
and then by \eqref{eq.defineZep.i} we have
\be
\label{eq.equation.Wep}
(\Delta + \lambda\beta)W_{\ep,i} = \lambda U_i \dd\mu_\ep -\beta \sum_{j=k}^{k+m-1} \delta_{\ep,i}^j u_j. 
\ee
On the other hand the Steklov eigenfunctions $U_{\ep,i}$ satisfy in the distributional sense
\be
\label{eq.equation.Uep}
(\Delta + \lambda\beta)U_{\ep,i} = \lambda \beta U_{\ep,i} -\lb \sigma_i^\ep U_{\ep,i} + \Lambda_\ep(U_{\ep,i})\rb\dA_\ep,
\ee
where $\Lambda_\ep$ is the Dirichlet-to-Neumann map on each $\pt B_p$ and in particular $\Lambda_\ep(U_{\ep,i})=-\pt_\nu^- U_{\ep,i}$. We can also write
\be
\label{eq.equation.Uep.otherform}
(\Delta + \lambda\beta) U_{\ep,i} = (\lambda-\sigma_i^\ep) \beta U_{\ep,i} + \sigma_i^\ep U_{\ep,i} \dd\mu_\ep - \Lambda_\ep(U_{\ep,i})\dA_\ep.
\ee

By subtracting \eqref{eq.equation.Wep} in \eqref{eq.equation.Uep.otherform}, we obtain
\be
\label{eq.equation.Eep}
\begin{split}
  (\Delta + \lambda\beta) E_{\ep,i} &= (\lambda-\sigma_i^\ep) \beta U_{\ep,i} + \sigma_i^\ep U_{\ep,i} \dd\mu_\ep - \Lambda_\ep(U_{\ep,i})\dA_\ep \\
  &\qquad -\lambda U_i \dd\mu_\ep +\beta \sum_{j=k}^{k+m-1} \delta_{\ep,i}^j u_j\\
  &=-\rho_{\ep,i} \beta U_{\ep,i} + \lambda (U_{\ep,i} -U_i) \dd\mu_\ep- \Lambda_\ep(U_{\ep,i})\dA_\ep \\
  &\qquad+(\sigma_i^\ep-\lambda) U_{\ep,i} \dd\mu_\ep - \delta_{\ep,i} \beta U_{\ep,i}+\beta \sum_{j=k}^{k+m-1} \delta_{\ep,i}^j u_j\\
 &= -\rho_{\ep,i} \beta U_{\ep,i} + \lambda Z_{\ep,i} \dd\mu_\ep + \lambda E_{\ep,i} \dd\mu_\ep- \Lambda_\ep(U_{\ep,i})\dA_\ep \\
 &\qquad+(\sigma_i^\ep-\lambda) U_{\ep,i} \dd\mu_\ep- \delta_{\ep,i} \beta U_{\ep,i}+\beta \sum_{j=k}^{k+m-1} \delta_{\ep,i}^j u_j.
\end{split}
\ee

By testing the above equation with the Laplace--Beltrami eigenfunction $U_{i}$, we obtain by \eqref{eq.eigenfunction.Unormalize}
\be
\label{eq.formula.rhoep}
\begin{split}
    \rho_{\ep,i} &= \lambda\int_M Z_{\ep,i} U_i \dd\mu_\ep + \lambda\int_M E_{\ep,i} U _ i\dd\mu_\ep - \int_M U_i \Lambda_\ep(U_{\ep,i}) \dA_\ep \\
    &\qquad+ \underbrace{(\sigma_i^\ep-\lambda)\int_M U_{\ep,i} U_i \dd\mu_\ep}_{J_a} \ +  \underbrace{ \sum_{j=k}^{k+m-1} \delta_{\ep,i}^j \int_M \beta U_i u_j \dV- \delta_{\ep,i} }_{J_b}.
\end{split}
\ee
Note that by \eqref{eq.define.delta.i.ep.j}, $J_b = 0$. By a similar argument to \eqref{eq.boundJ1.convergerate}, we have
\[
\lw \int_M U_{\ep,i} U_i \dd\mu_\ep \rw \le C \omega_d(\ep).
\]
This shows that by applying Lemma \ref{l.Optestimate.stektoNeum}, $|J_a| \le C \omega_d(\ep)^2 = o(\ep^2).$

The proof is then complete by combining \eqref{eq.formula.rhoep} with Lemma \ref{l.boundEU.integral} and Lemma \ref{l.DtN.fineestimate}. Note that the term 
\[
|\delta_{\ep,i}| =\lambda\lw \int_M \Phi_\ep \Delta(U_i^2) \dV\rw \le C \norm{\Phi_\ep}_{L^1(M)} \le C \ep^2.
\]

\end{proof}

In order to complete the proof of Theorem~\ref{t.interaction energy} it remains to simplify the next-order error terms. For simplicity, we will do this in the case that~$\lambda$ is a simple eigenvalue; the generalization to a multiple eigenvalue is straightforward. 
We carry out these simplifications in a sequence of lemmas. Fix a simple eigenpair~$(\lambda,U)$ satisfying~$(\Delta +\lambda \beta)U = 0$ on~$M,$ and suppose that~$\lambda=\lambda_k.$ We first introduce a Green function~$G_\lambda$ corresponding to the operator~$(\Delta + \lambda \beta),$ whose uniqueness follows from requiring orthogonality to~$U$ with respect to the~$(\cdot,\cdot)_\beta$ inner-product.  
\smallskip\noindent 
Towards defining this Green function, for some large~$N$, set 
\begin{equation*}
G_{\lambda,N}(x,y)
:=
\sum_{\substack{0\leq j\leq N\\j\neq k}}
\frac{u_j(x)u_j(y)}{\lambda-\lambda_j}\,.
\end{equation*}
\begin{lemma}
\label{l.spectral.construction.reduced.Green}
The sequence $\{G_{\lambda,N}\}_{N\in \N}$ converges in $\mathcal D'(M\times M)$ to a
symmetric distribution $G_\lambda$ denoted by 
\begin{equation}
\label{e.Green-1}
G_\lambda(x,y)
:=
\sum_{j\neq k}
\frac{u_j(x)u_j(y)}{\lambda-\lambda_j}\,,
\end{equation}
which is defined as: for every $\varphi,\psi\in C^\infty(M)$,
\begin{equation}
\label{e.Green.kernel}
\iint_{M\times M}
G_\lambda(x,y)\varphi(x)\psi(y)
\beta(x)\beta(y)\dV_x\dV_y
=
\sum_{j\neq k}
\frac{
\langle\varphi,u_j\rangle_\beta
\langle\psi,u_j\rangle_\beta
}{
\lambda-\lambda_j
},
\end{equation}
where
\[
\langle f,h\rangle_\beta
:=
\int_M fh\,\beta\dV.
\]
The series on the right-hand side of~\eqref{e.Green.kernel} converges absolutely. Furthermore, for every fixed $y\in M$,
\begin{equation}
    \label{e.GlambdaHs}
    G_\lambda(\,\cdot\,,y)\in H^s(M)
\qquad
\text{for every }s<2-\frac d2,
\end{equation}
and
\begin{equation}
\label{e.Green.kernel.orthogonality}
\int_M G_\lambda(x,y)U(x)\beta(x)\dV_x=0\,.
\end{equation}
\end{lemma}
\begin{proof}
    See Appendix~\ref{s.app}
\end{proof}

\begin{lemma}
\label{l.PDE.reduced.Green}
The Green function~$G_\lambda$ constructed in
Lemma~\ref{l.spectral.construction.reduced.Green} is the unique
distributional solution to satisfying: for every $\varphi\in C^\infty(M)$ and every
$y\in M$,
\begin{equation}
\label{e.Glambda_PDE_wkform}
\int_M
G_\lambda(x,y)
\bigl(\Delta\varphi(x)+\lambda\beta(x)\varphi(x)\bigr)\dV_x
=
\varphi(y)
-
U(y)\int_M\varphi U\beta\dV\,,
\end{equation}
that is orthogonal to~$U$ in the~$\beta-$weighted inner product: 
\[
\int_MG_\lambda(x,y)U(x)\beta(x)\dV_x=0.
\] 
Moreover,
\[
G_\lambda\in C^\infty\bigl((M\times M)\setminus\Delta_M\bigr),
\qquad
\Delta_M:=\{(x,x):x\in M\},
\]
and, near the diagonal,
\[
|G_\lambda(x,y)|
\leq
\begin{cases}
C\bigl(1+|\log d_g(x,y)|\bigr),&d=2,\\[2mm]
C\bigl(1+d_g(x,y)^{2-d}\bigr),&d\geq3.
\end{cases}
\]
In particular,
\[
G_\lambda\in L^1(M\times M).
\]
\end{lemma}
\begin{proof}
    See Appendix~\ref{s.app}
\end{proof}
We are ready to simplify the term~$\int_M Z_\e U \dd\mu_\e$ from the expansion in Theorem~\ref{t.interaction energy}. 
\begin{corollary}[Representation formula for~$Z_\ep$]
\label{c.Zepsilon}
We have the representation formula
\begin{equation}
\int_M Z_\ep U\dd\mu_\ep
=
\lambda
\iint_{(M\times M)\setminus\Delta_M}
G_\lambda(x,y)U(x)U(y)
\dd\mu_\ep(x)\dd\mu_\ep(y)\,. 
\end{equation}
In particular, the double integral is absolutely convergent.
\end{corollary}

\begin{proof}
Define
\[
\widetilde Z_\ep(x)
:=
\lambda
\int_MG_\lambda(x,y)U(y)\dd\mu_\ep(y).
\]
Using the distributional equation for $G_\lambda$, we obtain
\begin{align*}
(\Delta+\lambda\beta)\widetilde Z_\ep
&=
\lambda
\int_M
\left(
\delta_y-\beta(\cdot)U(\cdot)U(y)
\right)
U(y)\dd\mu_\ep(y)
\\
&=
\lambda U\dd\mu_\ep
-
\lambda\beta U
\left(
\int_MU^2\dd\mu_\ep
\right)dV
\\
&=
\lambda U\dd\mu_\ep
-
\delta_\ep\beta U\dV.
\end{align*}
Moreover,
\[
\int_M\widetilde Z_\ep U\beta\dV=0
\]
by the normalization of $G_\lambda$. Therefore we infer
\[
Z_\ep=\widetilde Z_\ep\,.
\]
Integrating the representation of $Z_\ep$ against
$U(x)\dd\mu_\ep(x)$ gives
\[
\int_MZ_\ep U\dd\mu_\ep
=
\lambda
\iint_{(M\times M)\setminus\Delta_M}
G_\lambda(x,y)U(x)U(y)\dd\mu_\ep(x)\dd\mu_\ep(y).
\]
It remains to verify absolute convergence. Since
\[
\dd|\mu_\ep|
\leq
\beta\dV+\dA_\ep
\]
and $U$ is bounded, it is enough to check the singular part of
$G_\lambda$. For $d\geq3$, the kernel is bounded by a multiple of
$d_g(x,y)^{2-d}$ near the diagonal. This singularity is locally
integrable against $\dV\otimes \dV$ and $\dV\otimes \dA_\ep$, and on the
smooth hypersurface $\partial\Omega_\ep$ one has
\[
\int_{\partial\Omega_\ep\cap B_r(x)}
d_g(x,y)^{2-d}\dA_y
\leq
C\int_0^r
t^{2-d}t^{d-2}\dd t
=
Cr\,,
\]
in~$d\geq 3$, and
\[
\int_0^r|\log t|\dd t<C r|\log r|\,,
\]
in~$d=2.$ 
Thus the double integral is absolutely convergent.
\end{proof}
\smallskip \noindent 
Finally, we are ready to prove:
\begin{proof}[Proof of Theorem~\ref{t.interaction energy}]
    The proof of the theorem directly follows by combining Lemma~\ref{l.expansion} with the representation formula from Corollary~\ref{c.Zepsilon}.
\end{proof}

\appendix
\section{Construction and properties of the a Green function}
\label{s.app}
\begin{proof}[Proof of Lemma~\ref{l.spectral.construction.reduced.Green}]
Let
\[
\mathcal A:=-\beta^{-1}\Delta
\]
acting on $L^2(M,\beta\dV)$, and set
\[
\gamma_\lambda
:=
\min_{j\neq k}|\lambda-\lambda_j|>0.
\]
For $N\geq k$, define
\[
T_Nf
:=
\sum_{\substack{0\leq j\leq N\\j\neq k}}
\frac{\langle f,u_j\rangle_\beta}
{\lambda-\lambda_j}\,u_j.
\]
The spectral theorem defines a bounded operator
\[
Tf
:=
\sum_{j\neq k}
\frac{\langle f,u_j\rangle_\beta}
{\lambda-\lambda_j}\,u_j
\]
on $L^2(M,\beta\dV)$, with
\[
\|T\|_{L^2_\beta\to L^2_\beta}
\leq\gamma_\lambda^{-1}.
\]
Moreover,
\[
\|T-T_N\|_{L^2_\beta\to L^2_\beta}
=
\sup_{\substack{j>N\\j\neq k}}
\frac1{|\lambda-\lambda_j|}
\longrightarrow0,
\]
because $\lambda_j\to\infty$.
Thus, by the Schwartz
kernel theorem, their kernels converge in $\mathcal D'(M\times M)$.
The kernel of $T_N$ is
\[
G_{\lambda,N}(x,y)
=
\sum_{\substack{0\leq j\leq N\\j\neq k}}
\frac{u_j(x)u_j(y)}{\lambda-\lambda_j},
\]
and we denote the limiting kernel by $G_\lambda$. For $\varphi,\psi\in C^\infty(M)$, Parseval's identity, Cauchy-Schwarz and the
definition of $\gamma_\lambda$ give
\begin{align*}
\sum_{j\neq k}
\frac{
|\langle\varphi,u_j\rangle_\beta|
|\langle\psi,u_j\rangle_\beta|
}{
|\lambda-\lambda_j|
}
&\leq
\frac1{\gamma_\lambda}
\sum_{j\neq k}
|\langle\varphi,u_j\rangle_\beta|
|\langle\psi,u_j\rangle_\beta|
\\
&\leq
\frac1{\gamma_\lambda}
\|\varphi\|_{L^2_\beta}
\|\psi\|_{L^2_\beta}.
\end{align*}
Thus the series
\[
\sum_{j\neq k}
\frac{
\langle\varphi,u_j\rangle_\beta
\langle\psi,u_j\rangle_\beta
}{
\lambda-\lambda_j
}
\]
converges absolutely and represents the action of the kernel
$G_\lambda$.

\smallskip \noindent We next establish~\eqref{e.GlambdaHs}. Introduce the
weighted Dirac distribution
\[
\delta_y^\beta:=\beta^{-1}\delta_y,
\]
characterized by
\[
\langle\delta_y^\beta,\varphi\rangle_\beta=\varphi(y).
\]
Its spectral coefficients are
\[
\langle\delta_y^\beta,u_j\rangle_\beta=u_j(y),
\]
and hence
\[
\delta_y^\beta
=
\sum_{j=0}^\infty u_j(y)u_j
\]
in $\mathcal D'(M)$.
For $s\in\mathbb R$, define the spectral Sobolev norm
\[
\|v\|_{H^s}^2
:=
\sum_{j=0}^\infty
(1+\lambda_j)^s
|\langle v,u_j\rangle_\beta|^2.
\]
For every $\eta>0$, point evaluation is continuous on
$H^{d/2+\eta}(M)$, by Sobolev embedding. By duality,
\[
\delta_y^\beta\in H^{-d/2-\eta}(M),
\]
and therefore
\begin{equation}
\label{e.plancherel-delta}
\sum_{j=0}^\infty
(1+\lambda_j)^{-d/2-\eta}
|u_j(y)|^2
<\infty.
\end{equation}
\smallskip \noindent On the other hand,
\[
G_{\lambda,N}(\cdot,y)
=
\sum_{\substack{0\leq j\leq N\\j\neq k}}
\frac{u_j(y)}{\lambda-\lambda_j}u_j.
\]
There exists $C_\lambda>0$ such that
\[
\frac1{|\lambda-\lambda_j|^2}
\leq
C_\lambda(1+\lambda_j)^{-2},
\qquad j\neq k.
\]
Consequently, for $N'>N$,
\[
\begin{aligned}
\|G_{\lambda,N'}(\cdot,y)
-G_{\lambda,N}(\cdot,y)\|_{H^s}^2
&\leq
C_\lambda
\sum_{\substack{N<j\leq N'\\j\neq k}}
(1+\lambda_j)^{s-2}|u_j(y)|^2.
\end{aligned}
\]
If $s<2-d/2$, choose $\eta>0$ such that
\[
s-2<-\frac d2-\eta.
\]
The right-hand side then tends to zero by
\eqref{e.plancherel-delta}. Hence
\[
G_{\lambda,N}(\cdot,y)
\longrightarrow G_\lambda(\cdot,y)
\quad\text{in }H^s(M)
\]
for every $s<2-d/2$.

The symmetry~$G_\lambda(x,y)=G_\lambda(y,x)$
\[
\int_MG_\lambda(x,y)U(x)\beta(x)\dV_x=0\,,
\]
are both immediate. 
\end{proof}

\begin{proof}[Proof of Lemma~\ref{l.PDE.reduced.Green}]
Fix $y\in M$ and let $\varphi\in C^\infty(M)$. Using the spectral
definition of $G_\lambda$, we compute
\begin{equation*}
\int_M
G_\lambda(x,y)
\bigl(\Delta\varphi(x)+\lambda\beta(x)\varphi(x)\bigr)\dV_x
~=~
\sum_{j\neq k}
\frac{u_j(y)}{\lambda-\lambda_j}
\int_M
u_j(x)
\bigl(\Delta\varphi(x)+\lambda\beta(x)\varphi(x)\bigr)\dV_x\,. 
\end{equation*}
Self-adjointness of the Laplace--Beltrami operator together with the PDE satisfied by each eigenpair~$(\lambda_j,u_j)$ yields 
\begin{equation*}
\int_M u_j\Delta\varphi\dV=
\int_M\varphi\Delta u_j\dV
=
-\lambda_j\int_M\varphi u_j\beta\dV.
\end{equation*}
Therefore,
\[
\int_M
u_j(\Delta\varphi+\lambda\beta\varphi)\dV
=
(\lambda-\lambda_j)
\int_M\varphi u_j\beta\dV\,,
\]
and so, continuing, we find 
\begin{equation*}
    \int_M G_\lambda(x,y) \bigl( \Delta \phi(x) + \lambda \beta(x) \phi(x) \bigr) dV_x = \sum_{j\neq k} u_j(y)\int_M\varphi u_j\beta\dV\,. 
\end{equation*}
As the eigenfunctions~$\{u_j\}_j$ form a complete orthonormal basis of
$L^2(M,\beta\dV)$, and so by the computations in the preceding lemma, we find 
\[
\sum_j
u_j(y)\int_M\varphi u_j\beta\dV
=
\varphi(y).
\]
Rearranging, we obtain
\[
\sum_{j\neq k}
u_j(y)\int_M\varphi u_j\beta\dV
=
\varphi(y)
-
U(y)\int_M\varphi U\beta\dV.
\]
This proves~\eqref{e.Glambda_PDE_wkform}. 

\smallskip \noindent We turn to prove uniqueness recalling the standard argument. If~$\widetilde G_\lambda$ is another
kernel satisfying \eqref{e.Glambda_PDE_wkform}, and also orthogonality with~$U,$ then for fixed $y$, the difference
\[
H(\,\cdot\,,y)
:=
G_\lambda(\,\cdot\,,y)
-
\widetilde G_\lambda(\,\cdot\,,y)
\]
satisfies
\[
(\Delta+\lambda\beta)H=0\,;
\]
simplicity of~$\lambda$ together with the orthogonality of~$G_\lambda$ with~$U$ by construction implies~$H=0.$ 

\smallskip \noindent Finally, away from $x=y$, the right-hand side of
\eqref{e.Glambda_PDE_wkform} is smooth, and so elliptic regularity therefore
implies
\[
G_\lambda\in
C^\infty\bigl((M\times M)\setminus\Delta_M\bigr).
\]
The stated near-diagonal estimates follow from the standard estimates for the operator $\Delta+\lambda\beta$, whose principal
part is the Laplace--Beltrami operator. The proof of the theorem is complete. 
\end{proof}

\section*{Data availability statement} No data was generated during the preparation of this article. 

\section*{Conflict of interest}
The authors have no conflict of interest. 
\bibliographystyle{plainnat}
\bibliography{ref}

\end{document}